\documentclass{amsart}

\usepackage{amsmath, amssymb, amsthm, amsfonts, graphicx}
\usepackage[usenames,dvipsnames]{xcolor}

\input xy
\xyoption{all}

\usepackage{hyperref}

\swapnumbers
\numberwithin{equation}{section}

\theoremstyle{plain}
\newtheorem{theorem}[subsubsection]{Theorem}
\newtheorem{lemma}[subsubsection]{Lemma}
\newtheorem{prop}[subsubsection]{Proposition}
\newtheorem{cor}[subsubsection]{Corollary}
\newtheorem{conj}[subsubsection]{Conjecture}

\theoremstyle{definition}
\newtheorem{defn}[subsubsection]{Definition}
\newtheorem{cons}[subsubsection]{Construction}
\newtheorem{remark}[subsubsection]{Remark}
\newtheorem{exam}[subsubsection]{Example}

\def\AA{\mathbb{A}}

\def\CC{\mathbb{C}}

\def\GG{\mathbb{G}}

\def\NN{\mathbb{N}}

\def\PP{\mathbb{P}}
\def\QQ{\mathbb{Q}}
\def\RR{\mathbb{R}}

\def\ZZ{\mathbb{Z}}

\newcommand\cA{\mathcal{A}}
\newcommand\cB{\mathcal{B}}
\newcommand\cC{\mathcal{C}}

\newcommand\cE{\mathcal{E}}
\newcommand\cF{\mathcal{F}}
\newcommand\cG{\mathcal{G}}
\newcommand\cH{\mathcal{H}}

\newcommand\cJ{\mathcal{J}}
\newcommand\cK{\mathcal{K}}
\newcommand\cL{\mathcal{L}}
\newcommand\cM{\mathcal{M}}
\newcommand\cN{\mathcal{N}}
\newcommand\cO{\mathcal{O}}
\newcommand\cP{\mathcal{P}}
\newcommand\cQ{\mathcal{Q}}

\newcommand\cS{\mathcal{S}}

\newcommand\cU{\mathcal{U}}

\newcommand\cX{\mathcal{X}}

\def\bC{\mathbf{C}}

\def\bI{\mathbf{I}}
\def\bJ{\mathbf{J}}
\def\bK{\mathbf{K}}

\def\bP{\mathbf{P}}
\def\bQ{\mathbf{Q}}
\def\bR{\mathbf{R}}

\def\bT{\mathbf{T}}

\def\bW{\mathbf{W}}

\newcommand\frX{\mathfrak{X}}
\newcommand\frY{\mathfrak{Y}}

\newcommand\fra{\mathfrak{a}}
\newcommand\frb{\mathfrak{b}}
\newcommand\frc{\mathfrak{c}}

\newcommand\frg{\mathfrak{g}}
\newcommand\frh{\mathfrak{h}}

\newcommand\frj{\mathfrak{j}}
\newcommand\frk{\mathfrak{k}}
\newcommand\frl{\mathfrak{l}}
\newcommand\fm{\mathfrak{m}}
\newcommand\frn{\mathfrak{n}}
\newcommand\frp{\mathfrak{p}}
\newcommand\frq{\mathfrak{q}}

\newcommand\frt{\mathfrak{t}}

\newcommand\frz{\mathfrak{z}}

\newcommand\tilW{\widetilde{W}}

\def\dG{G^{\vee}}

\newcommand{\Bun}{\textup{Bun}}
\newcommand{\can}{\textup{can}}

\newcommand{\Coh}{\textup{Coh}}
\newcommand{\coker}{\textup{coker}}

\newcommand{\Fl}{\textup{Fl}}

\newcommand{\Gr}{\textup{Gr}}
\newcommand{\gr}{\textup{gr}}

\newcommand\id{\textup{id}}
\renewcommand{\Im}{\textup{Im}}
\newcommand{\Ind}{\textup{Ind}}

\newcommand\Irr{\textup{Irr}}

\newcommand\Lie{\textup{Lie}\ }

\newcommand{\Nm}{\textup{Nm}}

\newcommand{\Pic}{\textup{Pic}}
\newcommand\pr{\textup{pr}}

\newcommand{\red}{\textup{red}}
\newcommand{\reg}{\textup{reg}}
\newcommand\Rep{\textup{Rep}}
\newcommand{\Res}{\textup{Res}}
\newcommand\res{\textup{res}}

\newcommand\rs{\textup{rs}}

\newcommand\Spec{\textup{Spec}\ }

\newcommand{\tors}{\textup{tors}}

\newcommand{\univ}{\textup{univ}}

\newcommand\Aut{\textup{Aut}}
\newcommand\Hom{\textup{Hom}}

\newcommand\Map{\textup{Map}}

\newcommand\uHom{\underline{\Hom}}

\newcommand\GL{\textup{GL}}

\newcommand\SL{\textup{SL}}
\renewcommand\sl{\mathfrak{sl}}

\newcommand{\Gm}{\GG_m}
\def\Ga{\GG_a}

\newcommand{\ad}{\textup{ad}}
\newcommand{\Ad}{\textup{Ad}}

\newcommand\xcoch{\mathbb{X}_*}

\newcommand{\incl}{\hookrightarrow}
\newcommand{\isom}{\stackrel{\sim}{\to}}

\newcommand{\surj}{\twoheadrightarrow}

\newcommand{\Qlbar}{\overline{\QQ}_\ell}

\newcommand{\twtimes}[1]{\stackrel{#1}{\times}}

\newcommand{\htimes}{\widehat{\times}}
\renewcommand{\j}[1]{\langle{#1}\rangle}
\newcommand{\wt}[1]{\widetilde{#1}}
\newcommand{\wh}[1]{\widehat{#1}}
\newcommand\quash[1]{}
\newcommand\un{\underline}
\newcommand{\bu}{\bullet}
\newcommand{\ov}{\overline}
\newcommand{\bs}{\backslash}
\newcommand{\tl}[1]{[\![#1]\!]}
\newcommand{\lr}[1]{(\!(#1)\!)}
\newcommand\sss{\subsubsection}
\newcommand\xr{\xrightarrow}
\newcommand\op{\oplus}
\newcommand\ot{\otimes}

\newcommand{\sslash}{\mathbin{/\mkern-6mu/}}
\renewcommand\c\circ

\newcommand{\cohog}[2]{\textup{H}^{#1}({#2})}     % plain group
\newcommand{\cohoc}[2]{\textup{H}_{c}^{#1}({#2})}     % compact support
\newcommand\upH{\textup{H}}

\renewcommand\a\alpha
\renewcommand\b\beta
\newcommand\g\gamma
\newcommand\G\Gamma
\renewcommand\d\delta
\newcommand\D\Delta
\newcommand{\e}{\epsilon}
\renewcommand{\th}{\theta}

\newcommand{\io}{\iota}
\newcommand{\ph}{\varphi}
\renewcommand\r\rho
\newcommand{\s}{\sigma}
\renewcommand{\t}{\tau}

\newcommand{\y}{\eta}
\newcommand{\z}{\zeta}

\renewcommand{\l}{\lambda}
\renewcommand{\L}{\Lambda}
\renewcommand{\k}{\kappa}
\newcommand{\om}{\omega}
\newcommand{\Om}{\Omega}

\newcommand\na{\natural}
\newcommand\sh{\sharp}
\newcommand\da{\dagger}
\newcommand\dda{\ddagger}

\newcommand\Kir{\textup{Kir}}

\newcommand\muSh{\mu\textup{Sh}}

\newcommand\rot{\textup{rot}}
\newcommand\Grot{\GG_{m}^{\textup{rot}}}
\newcommand\Gdil{\GG_{m}^{\textup{dil}}}

\newcommand\Hod{\textup{Hod}}
\newcommand{\Bet}{\textup{Bet}}
\newcommand\dR{\textup{dR}}
\newcommand\nb{\nabla}

\newcommand{\RH}{\operatorname{RH}}
\newcommand{\realcircle}{\mathbf{S}}

\newcommand{\anglevar}{\theta}

\newcommand{\Connec}{\operatorname{Conn}}
\newcommand{\Stokes}{\operatorname{Stokes}}

\newcommand{\exponenta}{\phi}
\newcommand{\exponentb}{\chi}
\newcommand{\startingangle}{\anglevar_0}

\newcommand{\exponentset}{\Psi}
\newcommand{\Config}{\operatorname{Config}}
\newcommand{\lisom}{\stackrel{\sim}{\leftarrow}}
\newcommand{\vn}{\varnothing}

\newcommand{\Stdir}{\operatorname{St}}
\newcommand{\ord}{\operatorname{ord}}

\title{Non-abelian Hodge moduli spaces and homogeneous affine Springer fibers}

\dedicatory{To George Lusztig with admiration}
\author{Roman Bezrukavnikov}
\address{Department of Mathematics, Massachusetts Institute of Technology, 77 Massachusetts Ave, Cambridge, MA 02139}
\email{bezrukav@math.mit.edu}

\author{Pablo Boixeda Alvarez}
\thanks{P.B.A was partially supported by the NSF under grant DMS-1926686}
\address{Department of Mathematics, Yale University, 10 Hillhouse Ave, New Haven, CT 06511}
\email{pablo.boixedaalvarez@yale.edu}

\author{Michael McBreen}
\thanks{M.M. is supported by an RGC Early Career Scheme grant, project number 24307121.}
\address{Department of Mathematics, Chinese University of Hong Kong, New Territories, Hong Kong SAR}
\email{mcb@math.cuhk.edu.hk}

\author{Zhiwei Yun}
\thanks{Z.Y. is supported by the Simon Foundation and the Packard Foundation.}
\address{Department of Mathematics, Massachusetts Institute of Technology, 77 Massachusetts Ave, Cambridge, MA 02139}
\email{zyun@mit.edu}

\date{}

\begin{document}
	
	\begin{abstract}
		Starting from a homogeneous affine Springer fiber $\Fl_{\psi}$, we construct three moduli spaces that correspond to the Dolbeault, de Rham and Betti aspects of a hypothetical Simpson correspondence with wild ramifications. We show that $\Fl_{\psi}$ is homeomorphic to the central Lagrangian fiber in the Dolbeault space, prove that the Dolbeaut and de Rham spaces both have the same cohomology as $\Fl_{\psi}$, and construct a map from the de Rham space to the Betti space which we conjecture to be an analytic isomorphism.
	\end{abstract}
	
	\maketitle
	\tableofcontents
	\section{Introduction}
	
	\subsection{Springer fiber and Slodowy slice}
	Let $G$ be a reductive group over an algebraically closed characteristic zero field $k$ with Lie algebra $\frg$. For a nilpotent element $e\in \cN\subset \frg$, the usual Springer fiber $\cB_{e}=\{gB\in G/B|\Ad(g^{-1})e\in \Lie B\}$ is the fiber over $e$ of the Springer resolution $\pi: T^{*}(G/B)=\wt\cN\to \cN$. Moreover, $\cB_{e}$ can be realized as a Lagrangian subscheme in a symplectic smooth variety, via the construction of the {\em Slodowy slice} through $e$. More precisely, let $\{e,h,f\}$ be an $\sl_{2}$-triple containing $e$.  Let $\cS_{e}=(e+\frg^{f})\cap \cN$ be the nilpotent part of the Slodowy slice through $e$. Let $\wt \cS_{e}=\cS_{e}\times_{\cN}\wt\cN$ be the preimage of $\cS_{e}$ under the Springer resolution. The following well known properties of these varieties play an important role in geometric representation theory.
	\begin{enumerate}
		\item $\wt\cS_{e}$ is a smooth variety over $k$ with a canonical symplectic form.
		\item The map $\pi_{e}: \wt\cS_{e}\to \cS_{e}$ is a symplectic resolution.
		
		\item There are compatible $\Gm$-actions on $\wt\cS_{e}$ and $\cS_{e}$ such  that the symplectic form on $\wt \cS_{e}$ has weight $2$, and $\cS_{e}$ contracts to the point $e$ under the $\Gm$-action. In particular, if $k=\CC$ the embedding $\cB_{e}\incl \wt\cS_{e}$ induces a homotopy equivalence between the corresponding complex varieties.
The action is the product of the dilation action on  $T^{*}(G/B)$ and conjugation by a cocharacter coming from the homomorphism $SL(2)\to G$ provided by the Jacobson-Morozov
Theorem. 
		
		\item The subvariety $\cB_{e}\incl \wt\cS_{e}$ (the fiber over $e$) is  Lagrangian in $\wt\cS_{e}$. 
		\item The symplectic variety $\wt\cS_{e}$ can be obtained from $T^{*}(G/B)$ by Hamiltonian reduction for a  unipotent subgroup $U_e\subset G$ equipped with
		an additive character.
			\end{enumerate}
	
	The main goal of this paper is to construct  and study an analogue of the resolved Slodowy slice  $\wt\cS_{e}$ when the Springer fiber $\cB_{e}$ is replaced by an affine Springer fiber $\Fl_{\psi}$ for a homogeneous element $\psi$.  Roughly speaking,  a homogeneous element $\psi$ is a regular semisimple element in the loop Lie algebra $\frg\lr{t}$ for which there exists an analogue of the Jacobson-Morozov cocharacter, i.e. a cocharacter of the Kac-Moody group for which $\psi$ is an eigenvector. 
	
	\subsection{Summary of main results}
	Starting from a homogeneous element $\psi$ in $\frg\lr{t}$, we construct three moduli spaces that serve as the Dolbeault, de Rham and Betti aspects in the terminology of the 
	non-abelian Hodge theory. 
	\begin{enumerate}
		\item The Dolbeault space $\cM_{\psi}$ is a moduli space of $G$-Higgs bundles on $\PP^{1}$ with level structures at $0$ and $\infty$, and prescribed irregular part at $\infty$. 
		This is the space we propose as a loop group analogue
		of the resolved Slodowy slice.
		It shares some of its properties although there are also some important differences.
		 In particular, it is a symplectic algebraic space with a $\Gm$ action contracting the space to a Lagrangian subspace homeomorphic to $\Fl_\psi$ (see Theorem \ref{th:geom M});  it can be obtained
		 from $T^{*}\Fl$ by Hamiltonian reduction for a subgroup $\bJ_{\infty}$  
		 in the loop group equipped with an additive character (see \S\ref{ss:red}). 
		 However,  the analogue of the Springer map is the  Hitchin map $f:\cM_{\psi}\to\cA_{\psi}$ which is a completely integrable system rather than a symplectic
		 resolution.
		  
		 Also, the spaces $\Fl_\psi$ and $\cM_\psi$ are finite dimensional but in general have infinite type: in the simplest example $\Fl_\psi$ is 
		 an infinite chain of projective lines. They are both of finite type if $\psi$ is elliptic.

		\item The de Rham space $\cM_{\dR,\psi}$ is a moduli space of $G$-connections on $\PP^{1}$ with level structures and prescribed irregular part at $\infty$. 	
		
			\item The Betti space $\cM_{\Bet, \psi}$ is a complex analytic stack that depends only on the positive braid determined by $\psi$. The essential part of it is, up to a quotient by $G$,  defined as an explicit subvariety in a product $G\times (G/B)^n$. The key property justifying its name is the interpretation of $\cM_{\Bet, \psi}$ as the moduli space of Stokes data (See \S\ref{RHm}).

		\end{enumerate}
	
	The three spaces $\cM_{\psi}$, $\cM_{\dR,\psi}$, $\cM_{\Bet, \psi}$ are related as follows.
	
	There is a one-parameter deformation $\cM_{\Hod, \psi}$ of $\cM_{\psi}$ with general fiber isomorphic to $\cM_{\dR, \psi}$. 
	This family carries a $\Gm$-action contracting it to the Lagrangian subspace $\Fl_\psi$ in the special fiber, which allows one to show that the restriction maps induce isomorphisms on cohomology (see Theorem \ref{th:coho M Fl} and Corollary \ref{c:comp dR}):
	\begin{equation}\label{same_coh}
	\cohog{*}{\Fl_{\psi}} \lisom\cohog{*}{\cM_{\psi}} \lisom \cohog{*}{\cM_{\Hod, \psi}} \isom\cohog{*} {\cM_{\dR,\psi}}.
	\end{equation}
	
	Recall that for other types of connections (for example, for non-singular connections on a projective curve over $\CC$) the non-abelian Hodge theory of Corlette, Donaldson, Hitchin and Simpson defines a hyper-K\"ahler structure on the Hitchin moduli space and an isomorphism between the Dolbeaut and the De Rham moduli spaces as real manifolds, and realizes the Hodge moduli space as
	an open subvariety (preimage of $\CC$ under the projection to $\CC{\mathbb{P}}^1$) in the twistor moduli space. We do not know 
	if that theory can be extended to our setting. 
	
	Another, more elementary comparison available for nonsingular connections over a complete complex curve is provided by the Riemann-Hilbert correspondence: it yields an isomorphism between the de Rham and the Betti moduli spaces as complex analytic (but not as algebraic) varieties. In \S\ref{ss:enh RH}, using Stokes theory for $G$-connections (see \S\ref{RHm}), we defined an enhanced Riemann-Hilbert map
	\begin{equation}\label{intro RH}
\wt\RH: \cM_{\dR, \psi}\to \cM_{\Bet,\psi}.
\end{equation}
We conjecture that this map is an analytic isomorphism.
	
	In the main body of the paper, we also include a variant of the above spaces $\cM_{\psi}, \cM_{\dR, \psi}$ and $\cM_{\Bet, \psi}$ with an arbitrary semisimple part of the residue/monodromy at $0$.
	
\subsection{Relation to earlier results}
 Many of the constructions presented here are related to ones found in the literature.
 
   	Precursors of the Dolbeault moduli space of this type, besides of course the original construction of Hitchin \cite{Hi}, appeared in the work of Beauville \cite{Beauville}, Bottacin \cite{Bottacin}, Markman \cite{Markman} and Oblomkov-Yun \cite{OY}. The paper \cite{FredNeitz} by Fredrickson and Neitzke studies a moduli space of Higgs bundles closely related to the case $G=\GL_{n}, \nu=d/n$ in this paper.

	Biquard and Boalch \cite{BB} have developed non-abelian Hodge theory for irregular connections whose polar part (excepting the residue) is semisimple. When the underlying curve is $\PP^1$ and the polar part is homogeneous, these are a special case of connections considered in the present work. However, the moduli spaces we consider are different in that we endow the connection with a framing at infinity: the moduli space of Biquard and Boalch can be obtained from a special case of our moduli space by Hamiltonian reduction with respect to a torus action. Boalch and Yamakawa \cite{BY} construct moduli spaces of Stokes data for $G$-local systems, and endow them with Poisson structures. Our Betti spaces, although presented differently, should be special cases of their construction. The work of Bremer-Sage \cite{BS} studies moduli of flat connections with level structure on the bundles, although for them, the underlying vector bundle is always trivial.

	Mochizuki \cite{Mo} has extended the irregular non-abelian Hodge correspondence to (unframed) wild Higgs bundles whose polar parts are not semisimple. To our knowledge it is not known if this induces an isomorphism of real manifolds in the general case. 
	
	The Betti space is a variant of a special case of the moduli spaces defined by Minh-Tam Trinh \cite{MT}. Its definition is purely Lie-theoretic  and it uses only the braid group element determined by $\psi$ (or rather its slope). A version of this space already appeared in Lusztig's definition of character sheaves \cite{Lu}. In the work of Shende-Treumann-Williams-Zaslow \cite{STWZ}, when $G=\GL(n)$, similar moduli spaces were interpreted as a moduli space of local systems on $\PP^{1}\bs\{0,\infty\}$ with Stokes data at $\infty$. 		

	In a research proposal, Vivek Shende speculated an irregular non-abelian Hodge correspondence for curves and a P=W conjecture in that setting. In his thesis, Minh-Tam Trinh \cite{MTThesis} formulated a precise P=W conjecture connecting  cohomology of Hitchin fibers and cohomology of Steinberg-like varieties attached to braids (which are closely related to $\cM_{\Bet,\psi}$). Our work gives supporting evidence for these conjectures in the case of homogeneous Hitchin fibers: assuming \eqref{intro RH} is an analytic isomorphism and combing it with the isomorphisms \eqref{same_coh}, one would conclude that $\cM_{\Bet,\psi}$ has the same cohomology as the affine Springer fiber $\Fl_{\psi}$.

	\subsection{Microlocal sheaves on $\Fl_{\psi}$ and wildly ramified geometric Langlands}\label{ss:muintro}
	While we believe the above results to be of independent interest, our motivation for considering those spaces came from our study of the categories of sheaves on the affine flag manifold and related categories of microlocal sheaves. We now briefly explain the setting and the motivating  conjectures, our results in this direction will be presented elsewhere.
	
	Denote by $\muSh_\Lambda(X)$ the category of microlocal sheaves on a polarised conical $X$ supported on a conical Lagrangian $\Lambda$ as constructed from work of \cite{KS}, \cite{Shende} and \cite{NadlerShende}.  
	
	Consider $\cM^{\dG}_{0,\Bet, \psi}$ the Betti space for $\dG$, the Langlands dual group, using the positive braid defined by $\psi$ and the definition in \S\ref{def:MBet}.
	
	\begin{conj}
		There is a fully faithful functor 
		\begin{equation*}
			\muSh_{\Fl_{\psi}}(\cM_{\psi})\incl\Ind\Coh(\cM^{0,\dG}_{\Bet, \psi}).
		\end{equation*}
	\end{conj}

	This conjecture can be viewed as a geometric Langlands correspondence for deeper level structures/wild ramifications. At the same time, it can be viewed as an instance of homological mirror symmetry between $\cM_{\psi}$ and $\cM^{\dG}_{\Bet, \psi}$. A version of this conjecture in the case $\psi$ is homogeneous of slope $1$ is proved in \cite{BBAMY}.

	\subsection*{Acknowledgement} 
	It is an honor to dedicate this paper to Prof. George Lusztig, whose contribution to representation theory is both universal but also particularly relevant to this paper: the study of affine Springer fibers was initiated in his paper with Kazhdan \cite{KL}, and the first examples of homogeneous affine Springer fibers were studied in his paper with Smelt \cite{LS}.
	
	The authors would like to thank Dima Arinkin,  Konstantin Jakob, Takuro Mochizuki and Minh-Tam Trinh for very helpful discussions.

	\section{The Dolbeault moduli space}\label{s:Dol}
		
	Let $G$ be a reductive group over an algebraically closed field $k$ such that the adjoint $G_{\ad}$ is simple. Let $\frg$ be the Lie algebra of $G$. Fix a $G$-invariant non-degenerate symmetric bilinear form $\j{\cdot, \cdot}$ on $\frg$ to identify $\frg$ with $\frg^{*}$. Let $T_{0}\subset G$ be a maximal torus and $W_{0}=W(G,T_{0})$ be the corresponding Weyl group of $G$. We assume $|W_{0}|$ is invertible in $k$.
	
	In this section, we will construct analogues of the resolved Slodowy slice $\wt\cS_{e}$ when the Springer fiber $\cB_{e}$ is replaced with affine Springer fibers $\Fl_{\psi}$ attached to certain elements $\psi\in\frg\lr{t}$ called {\em homogeneous} (see \S\ref{ss:homog}). We will construct a moduli space $\cM_{\psi}$ of Higgs bundles with the following features:
	\begin{enumerate}
		\item $\cM_{\psi}$ is a smooth algebraic space over $k$ with a canonical symplectic form.
		\item There is a Hitchin map $f: \cM_{\psi}\to \cA_{\psi}$ that is a completely integrable system. It is proper when $\psi$ is elliptic.
		\item There are compatible $\Gm$-actions on $\cM_{\psi}$ and $\cA_{\psi}$, (compute the weight of the symplectic form), contracting $\cA_{\psi}$ to the point $a_{\psi}$. The central fiber $f^{-1}(a_{\psi})=\cM_{a_{\psi}}$ is homeomorphic to the affine Springer fiber $\Fl_{\psi}$. 
	\end{enumerate}
	More precise statements will be given in Theorem \ref{th:geom M}. We will give three constructions of $\cM_{\psi}$ each having its own advantages in proving certain geometric properties.

	\subsection{Homogeneous affine Springer fibers}\label{ss:homog} Let $\frg\lr{t}=\frg\ot k\lr{t}$ be the loop Lie algebra. Let $G\lr{t}$ be the loop group of $G$ and $G\tl{t}$ be the positive loop group so that $(G\lr{t})(k)=G(k\lr{t})$ and $(G\tl{t})(k)=G(\tl{t})$. Then $G\lr{t}$ acts on $\frg\lr{t}$ by the adjoint action. Let $\fra=\frg\sslash G$ be the adjoint quotient, and $\fra\lr{t}=\fra(k\lr{t})$. 
	
	Let $\Grot$ be the one-dimensional torus acting on $\frg\lr{t}$ by loop rotation $\Grot\ni s: A(t)\mapsto A(st)$. Let $\Gdil$ be the one-dimensional torus acting on $\frg\lr{t}$ by dilation $\Gdil\ni s: A(t)\mapsto sA(t)$. The action of $\Grot\times \Gdil$ on $\frg\lr{t}$ induces an action on $\fra\lr{t}$. 
	
	Recall from \cite[Definition 3.1.2]{OY} that a regular semisimple element $a\in \fra\lr{t}$ is called {\em homogeneous of slope $\nu=d/m\in \QQ$} (where $d\in \ZZ$, $m\in \NN$ in lowest terms) if it is fixed by the action of the subtorus $\Gm(\nu)\subset \Grot\times\Gdil$ given by $\{(s^{-m}, s^{d})|s\in \Gm\}$. A regular semisimple element $\psi\in \frg\lr{t}$ is called {\em homogeneous of slope $\nu=d/m$} if its image in $\fra\lr{t}$ is. In other words, $\psi(t)\in\frg\lr{t}$ is homogeneous of slope $\nu=d/m$ if and only if it is regular semisimple, and that $s^{-d}\psi(s^{m}t)$ is in the same $G\lr{t}$-orbit of $\psi(t)$ for all $s\in k^{\times}$.
	
	Fix a Borel subgroup $B_{0}\subset G$ containing $T_{0}$, and let $\bI_{0}\subset G\tl{t}$ be the corresponding Iwahori subgroup. Let $\Fl=G\lr{t}/\bI_{0}$ be the affine flag variety of $G$. For $\psi\in\frg\lr{t}$  homogeneous of slope $\nu=d/m\ge0$, following \cite{KL}, we may consider its affine Springer fiber $\Fl_{\psi}$
	\begin{equation*}
		\Fl_{\psi}=\{g\bI_{0}\in\Fl|\Ad(g^{-1})\psi\in \Lie \bI_{0}\}.
	\end{equation*}
	This is an ind-scheme that is a union of projective schemes over $k$. It is an affine analogue of Springer fibers.

	\begin{exam} In the case $G=\SL_{n}$, consider the slope $\nu=d/n$, where $d\ge1$ is coprime to $n$. The element
		\begin{equation*}
			\psi=\left(\begin{array}{ccccc}  & 1  \\ && 1 \\ && &\ddots \\ &&&& 1 \\ t \end{array}\right)^{d}
		\end{equation*}
		is homogeneous of slope $\nu=d/n$. The affine Springer fiber $\Fl_{\psi}$ has been studied by Lusztig and Smelt \cite{LS}. It classifies chains of lattices in $k\lr{t}^{n}$ stable under $\psi$. We may reinterpret $\Fl_{\psi}$ as classifying chains of fractional ideals of the ring $k\tl{t,y}/(y^{n}-t^{d})$.
	\end{exam}

	\subsection{Moy-Prasad filtration}\label{ss:MP}
	
	By \cite[Theorem 3.2.5]{OY}, a homogeneous element $\psi\in \frg\lr{t}$ of slope $\nu=d/m$ exists if and only if $m$ is a regular number for the Weyl group $W_{0}$,  i.e., $m$ is the order of a regular element in $W_{0}$ in the sense of Springer \cite{Springer} (we will describe the conjugacy class of this regular element in Remark \ref{r:w}). Moreover, by \cite[3.3]{OY}, any such $\psi$  can be conjugated under $LG$ to an element of the following form. Let $x_{m}\in \xcoch(T_{0})_{\RR}$ be the point in the standard alcove (corresponding to $\bI_{0}$) that is in the same affine Weyl group orbit as $\r^{\vee}/m$. Let $\bP_{0}\subset G\lr{t}$ be the standard parahoric subgroup defined by  $x_{m}$ (as a point in the apartment  of $T_{0}$ in the building of $G\lr{t}$). 
	
	The point $x_{m}$ defines a Moy-Prasad grading on $\frg[t,t^{-1}]$
	\begin{equation}\label{Zgr}
		\frg[t,t^{-1}]=\bigoplus_{i\in\ZZ}\frg[t,t^{-1}]_{i/m}.
	\end{equation}
	Here $\frg[t,t^{-1}]_{i/m}$ is spanned by the affine root spaces of $\frg\lr{t}$ for the affine roots $\a+n$ such that $\a(x_{m})+n=i/m$. The Lie algebra of $\bP_{0}$ is the $t$-adic completion of $\bigoplus_{i\ge0}\frg\lr{t}_{i/m}$.
	
	Then any homogeneous element $\psi$ of slope $\nu=d/m$ can be $G^{\ad}\lr{t}$-conjugated into $\frg[t,t^{-1}]_{d/m}$. Below we assume $m$ is a regular number for $W_{0}$, $d>0$ is prime to $m$ and $\psi\in \frg[t,t^{-1}]_{d/m}$.  
	
	\subsection{The curve and the parahoric subgroup}
	Let $X=\PP^{1}_{k}$ with affine coordinate $t$. Then $t$ is a uniformizer at $0\in \PP^{1}$,  and $\t=t^{-1}$ is a uniformizer at $\infty\in \PP^{1}$. We identify the loop group $G\lr{t}$ with the loop group of $G$ at $0\in X$.

	We also have the loop group $G\lr{\t}$ at $\infty\in X$. For $i\in \ZZ$, let
	\begin{equation*}
		\frg\lr{\t}_{\le i/m}=\wh{\bigoplus}_{j\le i}\frg[t,t^{-1}]_{j/m}
	\end{equation*}
	where $\wh{\op}$ denotes the $\t$-adic completion of the direct sum. Then $\frg\lr{\t}_{\le i/m}$ is a $k\lr{\t}$-lattice in $\frg\lr{\t}$.

	Let $\bP_{\infty}\subset G\lr{\t}$ be the parahoric subgroup whose Lie algebra is $\frg\lr{\t}_{\le0}$. Let $\bP_{\infty}(\frac{i}{m})\subset \bP_{\infty}$ be the Moy-Prasad subgroups of $\bP_{\infty}$: its Lie algebra is $\frg\lr{\t}_{\le -i/m}$. Let $\bP^{+}_{\infty}=\bP_{\infty}(\frac{1}{m})$ be the pro-unipotent radical of $\bP_{\infty}$.
	
	In particular, $\psi\in \frg[t,t^{-1}]_{d/m}$ is viewed as a linear character of $\frg\lr{\t}_{\le -d/m}$.

	\subsection{The centralizer of $\psi$}\label{ss:C}
	Let $C$ be the torus over $X\bs\{0,\infty\}$ that is the centralizer of $\psi$ (which is a regular semisimple section of $\frg$ over $X\bs\{0,\infty\}$) under $G$. Note that $C$ is not necessarily split; it becomes split over the $\mu_{m}$-cover of $X\bs\{0,\infty\}=\Gm$, with monodromy given by a regular element in $W$ of order $m$. Let $\bC_{\infty}\subset C\lr{\t}$ be the unique parahoric subgroup, and $\bC^{+}_{\infty}$ be the pro-unipotent radical of $\bC_{\infty}$. 
	
	The grading \eqref{Zgr} on $\frg[t,t^{-1}]$ restricts to a grading on $\frc[t,t^{-1}]$, the global sections of the sheaf of Lie algebras $\Lie C$ over $X\bs \{0,\infty\}$
	\begin{equation*}
		\frc[t,t^{-1}]=\bigoplus_{i\in\ZZ}\frc[t,t^{-1}]_{i/m}
	\end{equation*}
	where $\frc[t,t^{-1}]_{i/m}\subset \frg[t,t^{-1}]_{i/m}$ are elements commuting with $\psi$.  
	
	Let $\frc\lr{\t}$ be the $\t=t^{-1}$-adic completion of $\frc[t,t^{-1}]$. Let $\frc\lr{\t}_{\le i/m}=\frg\lr{\t}_{\le i/m}\cap \frc\lr{\t}$.  Then $\Lie\bC_{\infty}=\frc\lr{\t}_{\le 0}$, and $\Lie \bC_{\infty}^{+}=\frc\lr{\t}_{\le -1/m}$.

	\begin{remark}\label{r:w}
		Specializing both sides of \eqref{Zgr} at $t=1$, each $\frg[t,t^{-1}]_{i/m}$ can be identified with a subspace of $\frg$. Thus we get a $\ZZ/m\ZZ$ grading of $\frg$
		\begin{equation}\label{Zmgr}
			\frg=\bigoplus_{i\in\ZZ/m\ZZ}\frg_{i/m}.
		\end{equation}
		Recall $x_{m}$ is the barycenter of $\bP_{0}$. Write $x_{m}=\xi/m$ where $\xi\in\xcoch(T^{\ad}_{0})$. View $\xi$ as a homomorphism $\xi: \Gm\to T^{\ad}_{0}$. 
		Then \eqref{Zmgr} is the grading obtained by the adjoint action of $\mu_{m}$ via $\xi|_{\mu_{m}}$. Note that $\xi_{\mu_{m}}$ is $W_{0}$-conjugate to $\r^{\vee}|_{\mu_{m}}$. Evaluating at $t=1$, $\frc[t,t^{-1}]_{i/m}$ is identified with a subspace $\frc_{i/m}\subset \frg_{i/m}$ depending only on $i/m\mod\ZZ$. Their sum
		\begin{equation*}
			\frt:=\op_{i\in \ZZ/m\ZZ}\frc_{i/m}\subset \frg
		\end{equation*}
		is a Cartan subalgebra of $\frg$ stable under the $\ZZ/m\ZZ$-grading \eqref{Zmgr}. Indeed, let $\ov\psi\in \frg_{d/m}$ be the image of $\psi$, then $\frt$ is the Lie algebra of the maximal torus $T:=C_{G}(\ov\psi)\subset G$, the fiber of $C$ over $1\in X$. Let $\z_{m}\in \mu_{m}$ be a primitive $m$-th root of unity, then $\Ad(\xi(\z_{m}))$ (whose eigenvalues on $\frg$ gives the $\ZZ/m\ZZ$-grading) normalizes $T$, hence determines an element $w=w(\z_{m})$ in the Weyl group $W=W(G,T)$ that is regular of order $m$. Different choices of $\z_{m}$ yields conjugate $w(\z_{m})$. This defines a regular conjugacy class in $W$ of order $m$.
	\end{remark}

	\subsection{Construction of $\cM_{\psi}$}\label{ss:cons1}
	
	Let $D_{0}=\Spec k\tl{t}, D_{0}^{\times}=\Spec k\lr{t}$ and $D_{\infty}=\Spec k\tl{\t}, D_{\infty}^{\times}=\Spec k\lr{\t}$. 
	
	Let $\cM_{\psi}$ be the moduli stack parametrizing pairs $(\cE, \ph)$ where
	\begin{itemize}
		\item $\cE$ is a $G$-bundle over $X$ with $\bK_{\infty}:=\bP_{\infty}(\frac{d}{m})\bC^{+}_{\infty}$-level structure at $\infty$ and $\bI_{0}$-level structure at $0$. We denote by $\Ad(\cE)$ the vector bundle over $X\bs \{0,\infty\}$ associated to the adjoint representation $\frg$ of $G$.
		\item $\ph$ is a section of $\Ad(\cE)\ot\om_{X\bs \{0,\infty\}}$ satisfying the following conditions: 
		\begin{enumerate}
			\item[(i)] After choosing a trivialization of $\cE|_{D_{\infty}}$ together with its $\bK_{\infty}$-level structure, we require
			\begin{equation*}
				\ph|_{D_{\infty}^{\times}}\in (\psi+\frg\lr{\t}_{\le0})d\t/\t.
			\end{equation*}
			Note that the right side is invariant under the adjoint action by $\bK_{\infty}$, therefore this condition is independent of the trivialization of $\cE|_{D_{\infty}}$.
			\item[(ii)] After  choosing a trivialization of $\cE|_{D_{0}}$ together with its $\bI_{0}$-level structure, we require
			\begin{equation*}
				\ph|_{D^{\times}_{0}}\in \Lie(\bI^{+}_{0})dt/t.
			\end{equation*}
		\end{enumerate}
	\end{itemize}

	\subsection{Hitchin base}
	The Hitchin base $\cA_{\psi}$ is defined as follows. Fix homogeneous generators $f_{i}\in k[\frg]^{G}$ of degree $d_{i}$, $1\le i\le r$, which give an isomorphism
	\begin{equation*}
		(f_{1},\cdots, f_{r}): \fra=\frg\sslash G\isom \AA^{r}.
	\end{equation*}
	We identify $\om_{X}(0+\infty)$ with $\cO_{X}$ using $\frac{dt}{t}$. In particular, $f_{i}(\psi)$ is a global section of $\cO([\frac{dd_{i}}{m}]\cdot\infty )$ (the twisting means pole order at $\infty$). Let $\cA_{\psi}\subset \prod_{i=1}^{r}\G(\PP^{1}, \cO([\frac{dd_{i}}{m}]\cdot\infty))$ be the subspace of sections $a=(a_{i})_{1\le i\le r}$ such that for each $i=1,\cdots, r$
	\begin{itemize}
		\item $a_{i}(0)=0$;
		\item $a_{i}\equiv f_{i}(\psi)\mod \t^{-\frac{d(d_{i}-1)}{m}}$ near $\infty$.
	\end{itemize}
	In other words, if we identify $\G(\PP^{1}, \cO([\frac{dd_{i}}{m}]\cdot\infty))$ with polynomials in $t$, then $f_{i}(\psi)$ is a monomial of degree $dd_{i}/m$ if $dd_{i}/m\in\ZZ$ and zero otherwise, and $a_{i}$ is a polynomial of the form
	\begin{equation*}
		a_{i}(t)=f_{i}(\psi)+\sum_{j=1}^{[d(d_{i}-1)/m]}a_{i,j}t^{j}.
	\end{equation*}

	\begin{lemma} Taking a Higgs bundle $(\cE,\ph)\in \cM_{\psi} $ to $(f_{i}(\ph))_{1\le i\le r}$ defines a map
		\begin{equation*}
			f: \cM_{\psi}\to \cA_{\psi}. 
		\end{equation*}
	\end{lemma}
	\begin{proof}
		Since the residue of $\ph$ at $0$ is nilpotent, $f_{i}(\ph)$ vanishes at $0$. To compute the pole order of $f_{i}(\ph)- f_{i}(\psi)$, we choose a trivialization of $\cE|_{D_{\infty}}$ so that $\ph\in (\psi+\frg\lr{\t}_{\le0})d\t/\t$. Write $\ph=(\psi+\th)d\t/\t$ for some $\th\in\frg\lr{\t}_{\le0}$. Inside $\frg\lr{\t^{1/m}}$ we can $G\lr{\t^{1/m}}$-conjugate the grading $\frg\lr{\t}_{i/m}$ (extended $k\lr{\t^{1/m}}$-linearly to $\frg\lr{\t^{1/m}}$) to the standard one given by powers of $\t^{1/m}$, i.e., $\frg\ot \t^{i/m}$. After this conjugation, $\ph$ becomes $\ph'=(x\t^{-d/m}+\sum_{j\ge0}\ph_{j}\t^{j/m})d\t/\t$, for some regular semisimple $x\in \frg$ such that $x\t^{-d/m}$ is in the same $G\lr{\t^{1/m}}$-orbit of $\psi$, and $\ph_{j}\in\frg$ for $j\ge0$. Then it is clear that $f_{i}(\ph)=f_{i}(\ph')$ has leading term $f_{i}(x)\t^{-dd_{i}/m}=f_{i}(\psi)$, with other terms starting with $\t^{-(d_{i}-1)d/m}$.
				\end{proof}

		\subsection{The $\Gm(\nu)$-actions}

		\sss{Action on $\Fl_{\psi}$}\label{sss:Gm action Fl} The one-dimensional torus $\Grot$ acts on $k\lr{t}$ by scaling the parameter $t$. We denote the action of $s\in \Grot$ by $\rot(s)$.
		
		Recall the barycenter $x_{m}=\xi/m$ that gives the parahoric subgroup $\bP_{\infty}$, where $\xi\in \xcoch(T^{\ad}_{0})$.

		We have a $\Gm(\nu)$-action on $\Fl_{\psi}$, where $(s^{-m}, s^{d})\in \Gm(\nu)$ acts by
		\begin{equation*}
		s: g\bI_{0}\mapsto \rot(s^{-m})\Ad(\xi(s^{-1}))g\bI_{0}.
		\end{equation*}

		\sss{Action on $\cM_{\psi}$} Let $\Grot$ act on $X=\PP^{1}$ by scaling the coordinate $t$. We denote the action of $s\in \Grot$ on $X$ by $\rot(s)$.
		Note that for any $s\in \Gm$,
		\begin{equation*}
			s^{d}\cdot \rot(s^{-m})(\Ad(\xi(s^{-1}))\psi)=\psi.
		\end{equation*}
		Since $\rot(s^{-m})\Ad(\xi(s^{-1}))$ fixes $\psi$, it normalizes $\bC_{\infty}$ and $\bC^{+}_{\infty}$. Clearly $\Grot\times T^{\ad}$ normalizes $\bP_{\infty}(\frac{i}{m})$ for all $i$, therefore $\rot(s^{-m})\Ad(\xi(s^{-1}))$ normalizes $\bK_{\infty}=\bP_{\infty}(\frac{d}{m})\bC^{+}_{\infty}$. Similarly, the action of $s^{d}\cdot \rot(s^{-m})\Ad(\xi(s^{-1}))$ stabilizes  $\psi+\frg\lr{\t}_{\le0}\subset \frg\lr{\t}$. Therefore we get a $\Gm(\nu)$-action on $\cM_{\psi}$ with $(s^{-m}, s^{d})\in \Gm(\nu)$ sending $(\cE,\ph)\in \cM_{\psi}$ to $(\cE', \ph')$ defined as follows.  First let  $\cE''$ be the $G$-bundle $\rot(s^{-m})^{*}\cE$ with $\bK''_{\infty}=\rot(s^{-m})\bK_{\infty}$-level at $\infty$ and $\bI_{0}$-level at $0$. Since $\Ad(\xi(s^{-1}))\bK''_{\infty}=\bK_{\infty}$, the action of $\Ad(\xi(s^{-1}))$ on $G\lr{\t}$ induces an equivalence between the groupoids of $G$-bundles with $\bK''_{\infty}$-level and with $\bK_{\infty}$-level at $\infty$. This turns $\cE''$ into a $G$-bundle $\cE'$ with $\bK_{\infty}$-level at $\infty$ and still $\bI_{0}$-level at $0$. Finally $\ph'=s^{d}\textup{rot}(s^{-m})^{*}\ph$.

		\sss{Action on $\cA_{\psi}$}
		The torus $\Gm(\nu)$ also acts on $\cA_{\psi}$ so that $(s^{-m},s^{d})$ sends $(a_{i}(t))_{i}$ to $(s^{dd_{i}}a_{i}(s^{-m}t))_{i}$. This action contracts $\cA_{\psi}$ to the unique fixed point $a_{\psi}=(f_{i}(\psi))_{1\le i\le r}\in \cA_{\psi}$. 
		
		\subsection{Main results on $\cM_{\psi}$}		
		The main geometric results on $\cM_{\psi}$ are:
		\begin{theorem}\label{th:geom M} For a homogeneous element $\psi\in \frg\lr{t}$ of slope $\nu=d/m$, the following hold.
			\begin{enumerate}
				\item\label{th part:M sm} The stack $\cM_{\psi}$ is a smooth algebraic space over $k$ of dimension $\frac{d}{m}|\Phi|-r+\dim \frt^{w}$ (where $w$ is a regular element in $W$ of order $m$ defined in \S\ref{ss:MP}).  
				\item\label{th part:M symp} $\cM_{\psi}$ carries a canonical symplectic structure of weight $d$ under the $\Gm(\nu)$-action.
				\item \label{th part:int sys}The map $f: \cM_{\psi}\to \cA_{\psi}$ is a $\Gm(\nu)$-equivariant completely integrable system (i.e., fibers of $f$ are Lagrangians).
				\item\label{th part:ASF vs HF} There is a natural map $\Fl_{\psi}\to \cM_{a_{\psi}}=f^{-1}(a_{\psi})$ which is a universal homeomorphism.
				\item \label{th part:f proper} When $\psi$ is elliptic (equivalently, $w$ is elliptic), $f$ is proper. 
			\end{enumerate}
		\end{theorem}
		
		\begin{remark}Constructions similar to $\cM_{\psi}$ have appeared before. 
			\begin{enumerate}
				\item When $G=\SL_{n}$, Markman \cite{Markman} has constructed a Poisson moduli of meromorphic Higgs bundles for arbitrary curve $X$ and showed that the Hitchin fibration is a completely integrable system (namely generic fibers are Lagrangian in symplectic leaves of maximal rank). If we take $\psi=t^{d}A$ where $d\ge1$ and $A\in\frg$ is regular semisimple, our $\cM_{\psi}$ is a symplectic leaf in Markman's Poisson moduli space. 
				\item For any $G$ and homogeneous $\psi$, Oblomkov and one of the authors \cite{OY} constructed a Poisson moduli of Higgs bundles (on a weighted projective line) with a contracting $\Gm$-action whose central fiber is closely related to $\Fl_{\psi}$.
			\end{enumerate}
			Both of the constructions above are more closely related to the Poisson moduli space $\cM^{\da}_{\psi}$ in \S\ref{ss:cons3}. 
		\end{remark}

		\begin{remark} Consider the case $k=\CC$. With the realization of $\Fl_{\psi}$ as a conical Lagrangian in the symplectic ambient space $\cM_{\psi}$, it makes sense to consider the category $\muSh_{\Fl_{\psi}}(\cM_{\psi})$ of microlocal sheaves on $\cM_{\psi}$ supported on $\Fl_{\psi}$. More precisely, using the realization of $\cM_{\psi}$ as a Hamiltonian reduction of the cotangent bundle of $\Bun_{G}(\bJ^{1}_{\infty}, \bI_{0})$ in \S\ref{ss:red}, one may define $\muSh_{\Fl_{\psi}}(\cM_{\psi})$ to be a full subcategory of sheaves on $\Bun_{G}(\bJ^{1}_{\infty}, \bI_{0})/\Gm(\nu)$ with singular support in $\Fl_{\psi}$. This can be thought of a quantization of $\cM_{\psi}$, or as an affine analogue of modules over $W$-algebras. In \cite{BBAMY} we study this category in the special case where $\psi$ is homogeneous of slope $1$.\end{remark}
		
		We also prove the following cohomological result.
		
		\begin{theorem}\label{th:coho M Fl} The canonical map $\g: \Fl_{\psi}\to  \cM_{a_{\psi}}\incl\cM_{\psi}$ induces an isomorphism on cohomology $\g^{*}: \cohog{*}{\cM_{\psi}}\isom \cohog{*}{\Fl_{\psi}}$. 		
		\end{theorem}
		Here, $\cohog{*}{\cM_{\psi}}$ is defined to be the projective limit of cohomology of finite-type open subspaces; and $\cohog{*}{\Fl_{\psi}}$ is defined to be the projective limit of cohomology of finite-type closed subschemes. When $\psi$ is elliptic, one can use the fact that $f$ is proper and the contraction principle to deduce Theorem \ref{th:coho M Fl} immediately. In general, the proof is more involved and uses hyperbolic localization.

		\subsection{Proof of Theorem \ref{th:geom M}\eqref{th part:M sm}}\label{ss:pf M sm}
		From the well-known properties of moduli stack of $G$-bundles we know that $\cM_{\psi}$ is an algebraic stack locally of finite type over $k$. 
		
		\sss{}\label{triv aut} We show that $\Aut(\cE,\ph)$ is trivial as an algebraic group for any $k$-point $(\cE,\ph)\in \cM_{\psi}$. By \cite[Cor 4.11.3]{Ngo}, $\Aut(\cE,\ph)$ is isomorphic to a subgroup of a maximal torus of $G$, hence diagonalizable (this is proved in the case without level structure but the argument works with level structure). On the other hand, restricting to $D_{\infty}$,  $\Aut(\cE,\ph)$ is a subgroup of the pro-unipotent group $\bK_{\infty}$, hence itself unipotent. Therefore, $\Aut(\cE,\ph)$ is the trivial algebraic group over $k$. This implies that $\cM_{\psi}$ is an algebraic space locally of finite type over $k$.

		\sss{}\label{tang M} We show that $\cM_{\psi}$ is a smooth Deligne-Mumford stack over $k$. 
		
		Let $\cE$ be a $G$-bundle over $X$ with $\bI_{0}$ and $\bK_{\infty}=\bP_{\infty}(\frac{d}{m})\cdot\bC^{+}_{\infty}$ level structures at $0$ and $\infty$ respectively. For a $\bI_{0}$-invariant lattice $\L_{0}\subset \frg\lr{t}$ and a $\bK_{\infty}$-invariant lattice $\L_{\infty}\subset \frg\lr{\t}$, we define $\Ad(\cE; \L_{0}, \L_{\infty})$ to be the subsheaf of $j_{*}\Ad(\cE)$ (where $j:X\bs \{0,\infty\}\incl X$) that is equal to $\Ad(\cE)$ over $\PP^{1}\bs \{0,\infty\}$ and its local sections near $0$ (resp. $\infty$) lies in $\L_{0}$ (resp. $\L_{\infty}$) after trivializing $\cE|_{D_{0}}$ (resp. $\cE|_{D_{\infty}}$).
		
		The tangent complex of $\cM_{\psi}$ at $(\cE,\ph)\in \cM_{\psi}$ is $\cohog{*}{X, \cK}$ where $\cK$ is the two step complex of vector bundles on $X$ placed in degrees $-1$ and $0$:
		\begin{equation*}
			\cK=\cK_{(\cE,\ph)}:=[\Ad(\cE; \Lie\bI_{0}, \frk_{\infty}) \xr{[-,\ph]} \Ad(\cE; \Lie \bI^{+}_{0}, \frg\lr{\t}_{\le0})].
		\end{equation*}
		Here $\frk_{\infty}=\Lie\bK_{\infty}$. The obstruction to the infinitesimal deformations of $(\cE,\ph)$ lies in $\cohog{1}{X,\cK}$ while the Lie algebra of $\Aut(\cE,\ph)$ is $\cohog{-1}{X,\cK}$. To show $\cM_{\psi}$ is smooth Deligne-Mumford, we need to show that $\cohog{1}{X,\cK}$ and $\cohog{-1}{X,\cK}$ vanish.
		
		We consider the complex $\cK^{\vee}=\uHom(\cK,\om_{X}[1])$, obtained by taking the Serre dual of $\cK$ termwise but still placed in degrees $-1$ and $0$: 
		\begin{equation*}
			\cK^{\vee}=[\Ad(\cE; \Lie\bI_{0}, \frg\lr{\t}_{\le -1/m}) \xr{[-,\ph]} \Ad(\cE; \Lie \bI^{+}_{0}, \frk_{\infty}^{\vee})].
		\end{equation*}
		Here we are using the pairing on $\Ad(\cE)$ induced from $\j{\cdot,\cdot}$ on $\frg$; $\frk_{\infty}^{\vee}\subset \frg\lr{\t}$ is the dual lattice of $\frk_{\infty}=\frg\lr{\t}_{\le -d/m}+\frc\lr{\t}_{\le-1/m}$.  Let $\frc\lr{\t}^{\bot}\subset \frg\lr{\t}$ be the orthogonal complement of $\frc\lr{\t}\subset \frg\lr{\t}$ under $\j{\cdot,\cdot}$. Let $\frc\lr{\t}^{\bot}_{\le i/m}=\frc\lr{\t}^{\bot}\cap\frg\lr{\t}_{\le i/m}$. Then
		\begin{equation*}
			\frk_{\infty}^{\vee}=\frc\lr{\t}^{\bot}_{\le (d-1)/m}\op\frc\lr{\t}_{\le 0}.
		\end{equation*}
		In particular, $\frg\lr{\t}_{\le0}\subset \frk_{\infty}^{\vee}$. Hence there is a natural  inclusion $\io: \cK\incl\cK^{\vee}$.  We claim that $\io$ induces a quasi-isomorphism in $D^{b}\Coh(X)$. Indeed, after trivializing $\cE|_{D_{\infty}}$, $\coker (\io)$ is the two-step complex
		\begin{equation*}
			\frg\lr{\t}_{\le -1/m}/\frk_{\infty}\xr{[-,\ph]}\frk_{\infty}^{\vee}/\frg\lr{\t}_{\le0}.
		\end{equation*}
		Moy-Prasad filtration induces filtrations $(\frk_{\infty}^{\vee}/\frg\lr{\t}_{\le0})_{\le j/m}$ and $(\frg\lr{\t}_{\le -1/m}/\frk_{\infty})_{\le j/m}$ on both sides, with associated graded
		\begin{eqnarray*}
			(\frg\lr{\t}_{\le -1/m}/\frk_{\infty})_{j/m}\cong \begin{cases} \frc\lr{\t}^{\bot}_{j/m}\cong \frg_{j/m}/\frc_{j/m},  & -d+1\le j\le -1 \\
				0 & \mbox{otherwise.}
			\end{cases}\\
			(\frk_{\infty}^{\vee}/\frg\lr{\t}_{\le0})_{j/m}\cong \begin{cases} \frc\lr{\t}^{\bot}_{j/m}\cong \frg_{j/m}/\frc_{j/m},  & 1\le j\le d-1 \\
				0 & \mbox{otherwise.}
			\end{cases} 
		\end{eqnarray*}
		The map $[-,\ph]$ sends $(\frg\lr{\t}_{\le -1/m}/\frk_{\infty})_{\le j/m}$ to $(\frk_{\infty}^{\vee}/\frg\lr{\t}_{\le0})_{\le (j+d)/m}$, and the induced map on the associated graded is $\ad(\psi)_{j/m}: \frg_{j/m}/\frc_{j/m}\to \frg_{(j+d)/m}/\frc_{(j+d)/m}$ for $-d+1\le j\le -1$. Since $\psi$ is regular semisimple, $\ad(\psi)_{j/m}$ is an isomorphism. Therefore $\coker(\io)$ is acyclic, hence $\io: \cK\incl\cK^{\vee}$ is a quasi-isomorphism. Therefore $\io$ induces an isomorphism $\cohog{*}{X,\cK}\isom \cohog{*}{X,\cK^{\vee}}$. On the other hand, the complexes $\cohog{*}{X,\cK}$ and $\cohog{*}{X,\cK^{\vee}}$ are linearly dual to each other. Therefore we conclude that there is a perfect pairing between $\cohog{1}{X,\cK}$ and $\cohog{-1}{X,\cK}$, and a perfect pairing on $\cohog{0}{X,\cK}$ which is easily seen to be symplectic. 
		
		By \S\ref{triv aut},  $\cohog{-1}{X,\cK}=\Lie\Aut(\cE,\ph)=0$. Therefore $\cohog{1}{X,\cK}=0$ as well, hence  $\cM_{\psi}$ is a smooth algebraic space. Moreover,  the tangent space $\cohog{0}{X,\cK}$ at every point $(\cE,\ph)$ carries a canonical symplectic form, hence a globally defined non-degenerate 2-form $\Om$ on $\cM_{\psi}$. The fact that $d\Om=0$ will be shown in \S\ref{ss:red} where we identify $\cM_{\psi}$ as a Hamiltonian reduction from a cotangent bundle, which then carries a canonical symplectic form, and it is easy to check that the form coincides with $\Om$.

		\sss{} We compute $\dim\cM_{\psi}$, or $\dim \cohog{0}{X,\cK}$. By the vanishing of $\cohog{\ne0}{X,\cK}$, we have
		\begin{eqnarray}
			\notag \dim \cohog{0}{X,\cK}&=&\deg\Ad(\cE;  \Lie \bI^{+}_{0}, \frg\lr{\t}_{\le0})-\deg\Ad(\cE;  \Lie \bI_{0}, \frk_{\infty})\\
			\label{H0K} &=&\dim_{k}\frg\lr{\t}_{\le 0}/\frk_{\infty}-\dim_{k}\Lie\bI_{0}/\Lie\bI_{0}^{+}=\dim_{k}\frg\lr{\t}_{\le 0}/\frk_{\infty}-r.
		\end{eqnarray}
		
		By construction,  we have
		\begin{equation*}
			\frg\lr{\t}_{\le 0}/\frk_{\infty}\cong \frg_{0}\op\bigoplus_{i=1}^{d-1}\frg_{i/m}/\frc_{i/m}.
		\end{equation*}
		Therefore
		\begin{equation}\label{defect k}
			\dim\frg\lr{\t}_{\le 0}/\frk_{\infty}=\dim \frc_{0}+\sum_{i=0}^{d-1}\frg_{i/m}/\frc_{i/m}.
		\end{equation}
		We consider the roots $\Phi(\frg,\frc)$ for the Cartan $\frc$. The $\ZZ/m\ZZ$-grading on $\frg$ is induced by $w\in W(\frg,\frc)$ regular of order $m$. Now $w$ permutes $\Phi(\frg,\frc)$ freely with $|\Phi|/m$ orbits. Each orbit contributes $1$-dimension to each $\frg_{i/m}/\frc_{i/m}$ for all $i\in\ZZ/m\ZZ$. Therefore $\dim\frg_{i/m}/\frc_{i/m}=|\Phi|/m$ for all $i\in\ZZ$. Using \eqref{defect k}, we see that
		\begin{equation*}
			\dim\frg\lr{\t}_{\le 0}/\frk_{\infty}=\dim \frc_{0}+\frac{d}{m}|\Phi|=\dim\frt^{w}+\frac{d}{m}|\Phi|.
		\end{equation*}
		Combined with \eqref{H0K} we get
		\begin{equation*}
			\dim_{(\cE,\ph)} \cM_{\psi}=\dim \cohog{0}{X,\cK_{(\cE,\ph)}}=\dim \frt^{w}+\frac{d}{m}|\Phi|-r.
		\end{equation*}
		\qed

		\subsection{Proof of Thereom \ref{th:geom M}\eqref{th part:ASF vs HF}}
		Ng\^o's product formula \cite[Prop. 4.15.1]{Ngo} and its extension by Bouthier and Cesnavicius \cite[Theorem 4.3.8]{BC} has an analogue for our level structures which we spell out. 
		
		Define a reduced sub-ind-scheme of $G\lr{\t}/\bK_{\infty}$:
		\begin{equation*}
			\cX_{\psi, \infty}=\{g\bK_{\infty}\in G\lr{\t}/\bK_{\infty}|\Ad(g^{-1})\psi\in \psi+\frg\lr{\t}_{\le0}\}.
		\end{equation*}
		This is an analog of an affine Springer fiber. 
		
		Recall the torus $C$ over $\PP^{1}\bs\{0,\infty\}$ defined in \S\ref{ss:C} as the centralizer of $\psi$.  Extend $C$ to a group scheme $\cC$ on $\PP^{1}$ with parahoric level structures at $0$ and $\infty$. Let $\Pic_{C}(\wh 0; \wh \infty)$ be the moduli space of $\cC$ torsors over $\PP^{1}$ with trivializations on $D_{0}$ and $D_{\infty}$. Then $C\tl{t}$ and $C\tl{\t}$ act on $\Pic_{C}(\wh 0; \wh \infty)$ by changing the trivializations, and these actions extend to actions of  the loop tori $C\lr{t}$ and $C\lr{\t}$. On the other hand, $C\lr{t}$ and $C\lr{\t}$ act on $\Fl_{\psi}$ and $\cX_{\psi, \infty}$ respectively by left translations.
		
		There is a canonical morphism
		\begin{equation}\label{prod Ma}
			\a: \Pic_{C}(\wh 0; \wh \infty)\twtimes{C\lr{t}\times C\lr{\t}}(\Fl_{\psi}\times \cX_{\psi,\infty})\to \cM_{a_{\psi}}
		\end{equation}
		defined as follows. Given a $\cC$-torsor $\cQ$ over $\PP^{1}$ with trivializations on $D_{0}$ and $D_{\infty}$, we get a $G$-torsor $\cE^{\c}=\cQ\twtimes{C}G$ over $\PP^{1}\bs\{0,\infty\}$ with a Higgs field given by $\psi$ and trivializations on $D_{0}$ and $D_{\infty}$. A point $g_{0}\bI_{0}\in \Fl_{\psi}$ gives a $G$-torsor $\cE_{0}$ over $D_{0}$ with $\bI_{0}$-level structure together with a trivialization over $D_{0}^{\times}$. We can glue $\cE_{0}$ with $\cE^{\c}$ along $D^{\times}_{0}$ using the trivializations. Similarly, a point $g_{\infty}\bK_{\infty}\in \cX_{\psi,\infty}$ gives a $G$-torsor $\cE_{\infty}$ over $D_{\infty}$ with $\bK_{\infty}$-level structure together with a trivialization over $D_{\infty}^{\times}$, which we can glue with $\cE^{\c}$ along $D^{\times}_{0}$. This way we have extended $\cE^{\c}$ to a $G$-torsor $\cE$ on $\PP^{1}$ with $\bI_{0}$ and $\bK_{\infty}$-level structures. The Higgs field $\psi$ on $\cE^{\c}$ extends to $\cE$ because of the conditions defining $\Fl_{\psi}$ and $\cX_{\psi, \infty}$. This gives the map $\a$.  The same argument of \cite[Theorem 4.3.8]{BC} shows that $\a$ is a universal homeomorphism: the reason is that for any $(\cE,\ph)\in \cM_{a_{\psi}}(R)$ where $R$ is a seminormal strictly Henselian local $k$-algebra, $(\cE,\ph)|_{\PP^{1}_{R}\bs \{0,\infty\}}$ reduces to a $C$-torsor, and the restriction of any $C$-torsor over $\Spec R\lr{t}$ and $\Spec R\lr{\t}$ must be trivial, as shown in \cite[Theorem 3.2.4]{BC} (using that $m$ is invertible in $k$, hence $C$ splits over a tamely ramified cover of $\Gm$).
		
		By Lemma \ref{l:C trans} below, the action of $C\lr{\t}$ on $\cX_{\psi, \infty}$ is transitive.   The stabilizer of $C\lr{\t}$ at the base point $1\in \cX_{\psi,\infty}$ is $C\lr{\t}\cap \bK_{\infty}=\bC^{+}_{\infty}$. Therefore the action map $C\lr{\t}/\bC^{+}_{\infty}\to \cX_{\psi, \infty}$ is an isomorphism on the reduced structures. This allows us to simplify the left side of \eqref{prod Ma} to
		\begin{equation}\label{prod Ma0}
			\Pic_{C}(\wh 0; \bC_{\infty}^{+})\twtimes{C\lr{t}}\Fl_{\psi}\to \cM_{a_{\psi}}.
		\end{equation}
		Here $\Pic_{C}(\wh 0; \bC_{\infty}^{+})$ is the moduli space of $\cC$-torsors over $\PP^{1}$ with a trivialization on $D_{0}$ and a $\bC_{\infty}^{+}$-level structure at $\infty$. Let $\bC_{0}\subset C\lr{t}$ be the parahoric subgroup. Then $\Pic_{C}(\bC_{0}; \bC_{\infty}^{+})$ is the discrete space $\xcoch(T)_{w}$. Indeed, a computation of tangent space shows that $\Pic_{C}(\bC_{0}; \bC_{\infty}^{+})$ is discrete; the automorphism group of the identity point is trivial hence the automorphism group of all points are trivial since $\Pic_{C}(\bC_{0}; \bC_{\infty}^{+})$ is a Picard groupoid. By \cite[Lemma 16]{Hein} the connected components of $\Pic_{C}(\bC_{0}; \bC_{\infty})$ are canonically indexed by $\xcoch(T)_{w}$, hence the same is true for $\Pic_{C}(\bC_{0}; \bC_{\infty}^{+})$. On the other hand, the Kottwitz map gives an isomorphism $(C\lr{t}/\bC_{0})^{\red}\isom \xcoch(T)_{w}$ (see \cite[Theorem 5.1, step A]{PR}), and $\Pic_{C}(\bC_{0}; \bC_{\infty}^{+})$ is a trivial torsor under $(C\lr{t}/\bC_{0})^{\red}$. Therefore the action of $C\lr{t}$ on $\Pic_{C}(\wh 0; \bC_{\infty}^{+})$ is transitive, and the reduced stabilizer is trivial. Hence the natural map $\Fl_{\psi}\to \Pic_{C}(\wh 0; \bC_{\infty}^{+})\twtimes{C\lr{t}}\Fl_{\psi}$ is an isomorphism on the reduced structure. Since \eqref{prod Ma0} is a universal homeomorphism, we conclude that the composition
		\begin{equation*}
			\Fl_{\psi}\to \Pic_{C}(\wh 0; \bC_{\infty}^{+})\twtimes{C\lr{t}}\Fl_{\psi}\to \cM_{a_{\psi}}
		\end{equation*}
		is a universal homeomorphism. \qed
		
		\begin{lemma}\label{l:C trans} Let $\psi'\in \psi+\frg\lr{\t}_{\ge0}$ be in the same $G\lr{\t}$-orbit of $\psi$, then there exists $g\in \bP_{\infty}(\frac{d}{m})$ such that $\psi'=\Ad(g)\psi$.
		\end{lemma}
		\begin{proof}We construct inductively a sequence of elements $g_{j}\in \bP_{\infty}(\frac{d}{m})$ (for $j\le 1$) such that
			\begin{enumerate}
				\item $g_{1}=1$;
				\item For $j\le 0$, $g_{j}\in g_{j+1}\bP_{\infty}(\frac{-j+d}{m})$;
				\item For $j\le 0$, $\Ad(g_{j})\psi\equiv \psi'\mod \frg\lr{\t}_{\le j-1}$.
			\end{enumerate}
			Then the limit  $g=\lim_{j\to -\infty} g_{j}$ exists in $\bP_{\infty}(\frac{d}{m})$, and it satisfies $\Ad(g)\psi'=\psi$.
			
			Take $g_{1}=1$. Suppose $g_{j+1}$ has been constructed. Then we have
			\begin{equation*}
				\Ad(g_{j+1})\psi'=\psi+X_{j}+X_{j-1}+\cdots, \quad X_{k}\in \frg\lr{\t}_{j/m}.
			\end{equation*}
			We look for $Y\in \frg\lr{\t}_{(j-d)/m}$ such that $[Y, \psi]=X_{j}$; then $g_{j}=\exp(-Y)g_{j+1}$ satisfies all the requirements.
			
			Now solve the equation $[Y, \psi]=X_{j}$ for $Y\in \frg\lr{\t}_{(j-d)/m}$. Recall $f_{1},\cdots, f_{r}\in k[\frg]^{G}$ is a set of homogeneous generators with degrees $d_{1},\cdots, d_{r}$. By assumption $f_{i}(\psi')=f_{i}(\psi)$, hence
			\begin{equation*}
				f_{i}(\psi)=f_{i}(\Ad(g_{j+1})\psi')=f_{i}(\psi+X_{j}+X_{j-1}+\cdots).
			\end{equation*}
			Taking Taylor expansion of $f_{i}$ at $\psi$ with respect to the Moy-Prasad grading, we get
			\begin{equation*}
				\j{df_{i}(\psi), X_{j}}=0, \quad \forall i=1,\cdots, r.
			\end{equation*}
			Here $df_{i}(\psi)\in \frg^{*}\lr{\t}$, and the pairing $\j{,}$ is the $k\lr{\tau}$-bilinear between $\frg^{*}\lr{\t}$ and $\frg\lr{\t}$. Since $\psi$ is regular semisimple as an element in $\frg\lr{\t}$, the differentials $\{df_{i}(\psi)\}_{1\le i\le r}$ span a subspace of $\frg^{*}\lr{\t}$ which under the Killing form can be identified with $\frc\lr{\t}$ (the centralizer of $\psi$).  Therefore, the annihilator of the span of $\{df_{i}(\psi)\}_{1\le i\le r}$ is $[\frg\lr{\t},\psi]$.  The above equations imply that $X_{j}\in [\frg\lr{\t},\psi]$, so $X_{j}\in [Z,\psi]$ for some $Z\in \frg\lr{\t}$. Let $Y$ be the $\frg\lr{\t}_{(j-d)/m}$ homogeneous component of $Z$, then $[Y,\psi]=X_{j}$. This completes the inductive construction of $g_{j}$.
		\end{proof}

		\subsection{Construction of $\cM_{\psi}$ as a Hamiltonian reduction}\label{ss:red}
		The symplectic structure on $\cM_{\psi}$ mentioned in Theorem \ref{th:geom M} comes from a realization of $\cM_{\psi}$ as a Hamiltonian reduction of a certain cotangent space.
		
		We define a subgroup $\bJ_{\infty}\subset G\lr{\t}$ as follows:
		\begin{enumerate}
			\item If $d$ is odd, then $\bJ_{\infty}=\bP_{\infty}(\frac{(d+1)/2}{m})\cdot \bC^{+}_{\infty}$.
			\item If $d$ is even, then $\frg\lr{\t}_{d/2m}$ carries an alternating form $(x,y)\mapsto \j{\psi, [x,y]}$. Let $\fm\subset \frg\lr{\t}_{-d/2m}$ be a maximal isotropic subspace, and let $\bP_{\infty}(\frac{d/2}{m})_{\fm}$ be its preimage in $\bP_{\infty}(\frac{d/2}{m})$ under the projection $\bP_{\infty}(\frac{d/2}{m})\to \frg\lr{\t}_{-d/(2m)}$. Let $\bJ_{\infty}=\bP_{\infty}(\frac{d/2}{m})_{\fm}\cdot \bC^{+}_{\infty}$.
		\end{enumerate}
		Then $\psi$ has a unique extension to a linear character  $\wt\psi: \bJ_{\infty}\to \Ga$ such that, on the level of Lie algebras,  $\wt\psi$ is trivial on $(\Lie\bJ_{\infty})\cap \frg\lr{\t}_{i/m}$ for $i<d$.		
		
		Recall the notation of Hamiltonian (or Marsden-Weinstein) reduction: let $\frX$ be a smooth stack with the action of an algebraic group $H$. Let $\z:\frh\to k$ be a character of the Lie algebra $\frh=\Lie H$, viewed as an element in $\frh^{*}$. Let $\mu_{H}: T^{*}\frX\to \frh$ be the moment map. Then define the stack
		\begin{equation*}
			T^{*}\frX//_{\z}H=\mu_{H}^{-1}(\z)/H.
		\end{equation*}
		When $\z$ is a regular value of $\mu_{H}$, $T^{*}\frX//_{\z}H$ is a smooth stack that inherits an exact symplectic structure from that of $T^{*}\frX$.
		
		We apply this construction to the case $\frX=\Bun_{G}(\bJ^{1}_{\infty}; \bI_{0})$, the moduli stack of $G$-bundles on $X=\PP^{1}$ with $\bJ^{1}_{\infty}:=\ker(\wt\psi)\subset \bJ_{\infty}$-level structure at $\infty$ and $\bI_{0}$-level structure at $0$, with the action of $H=\Ga=\bJ_{\infty}/\bJ_{\infty}^{1}$ and the character $\z=\wt\psi\in\frh^{*}$.

		\begin{prop}\label{p:red}There is a canonical isomorphism between $\cM_{\psi}$ and the Hamiltonian reduction of $T^{*}\Bun_{G}(\bJ^{1}_{\infty}; \bI_{0})//_{\wt\psi}\Ga$. In particular,  $\cM_{\psi}$ carries a canonical symplectic structure which coincides with the $2$-form $\Om$ defined in \S\ref{ss:pf M sm}.
		\end{prop}
		\begin{proof}
			We describe $\cN_{\psi}:=T^{*}\Bun_{G}(\bJ^{1}_{\infty}; \bI_{0})//_{\wt\psi}\Ga$ as a moduli space of Higgs bundles as follows. Let $\frj_{\infty}=\Lie \bJ_{\infty}$ and $\frj_{\infty}^{\vee}\subset \frg\lr{\t}$ be the dual lattice under the form $\j{\cdot,\cdot}$ extended $k\lr{\t}$-linearly, i.e., $\frj_{\infty}^{\vee}=\{v\in\frg\lr{\t}|\j{v,\frj_{\infty}}\subset k\tl{\t}\}$.
			
			Recall the notations $\frc\lr{\t}^{\bot}, \frc\lr{\t}^{\bot}_{\le j/m}$ from \S\ref{tang M}. Then 
			\begin{equation}\label{jdual}
				\frj^{\vee}_{\infty}=\begin{cases}\frc\lr{\t}^{\bot}_{\le (d-1)/(2m)}\op\frc\lr{\t}_{\le0} & d \textup{ odd;}\\
					\fm^{\bot}\op \frc\lr{\t}^{\bot}_{\le (d/2-1)/m}\op\frc\lr{\t}_{\le0} & d \textup{ even.}
				\end{cases}
			\end{equation}
			Here $\fm^{\bot}\subset \frg\lr{\t}_{d/(2m)}$ is the orthogonal complement of $\fm\subset \frg\lr{\t}_{-d/(2m)}$ under the pairing $\j{\cdot,\cdot}$.
			
			Then $\cN_{\psi}$ classifies pairs $(\cE,\ph)$ where
			\begin{itemize}
				\item $\cE$ is a $G$-bundle over $X$ with $\bJ_{\infty}$-level structure at $\infty$ and $\bI_{0}$-level structure at $0$. We denote by $\Ad(\cE)$ the vector bundle over $X\bs \{0,\infty\}$ associated to the adjoint representation $\frg$ of $G$.
				\item $\ph$ is a section of $\Ad(\cE)\ot\om_{X\bs \{0,\infty\}}$ satisfying the following conditions: 
				\begin{enumerate}
					\item[(i)] Under some (equivalently, any) trivialization of $\cE|_{D_{\infty}}$ together with its $\bJ_{\infty}$-level structure, we require
					\begin{equation*}
						\ph|_{D_{\infty}^{\times}}\in (\psi+\frj^{\vee}_{\infty})d\t/\t.
					\end{equation*}
					\item[(ii)] Under some (equivalently, any) trivialization of $\cE|_{D_{0}}$ together with its $\bI_{0}$-level structure, we require
					\begin{equation*}
						\ph|_{D^{\times}_{0}}\in \Lie(\bI^{+}_{0})dt/t.
					\end{equation*}
				\end{enumerate}
			\end{itemize}
			
			Since $\bK_{\infty}\subset \bJ_{\infty}$ and $\psi+\frg\lr{\t}_{\le0}\subset \psi+\frj^{\vee}_{\infty}$, we have a natural map
			\begin{equation*}
				F: \cM_{\psi}\to \cN_{\psi}.
			\end{equation*}
			We need to show that $F$ is an isomorphism of algebraic stacks.

			Let  $U_{\psi}=(\psi+\frj^{\vee}_{\infty})/\frg\lr{\t}_{\le0}$ (an affine space). Let $\ov J=\bJ_{\infty}/\bP(\frac{d}{m})$. Then $\ov J$ acts on $U_{\psi}$ by the adjoint action. Let $\ov C\subset \ov J$ be the image of $\bC^{+}_{\infty}$, then $\ov C$ stabilizes the point $\psi\in U_{\psi}$. We thus get a morphism of stacks
			\begin{equation*}
				\io: [\{\psi\}/\ov C]\to  [U_{\psi}/\ov J].
			\end{equation*}
			On the other hand, we have an evaluation map
			\begin{equation*}
				\e: \cN_{\psi}\to [U_{\psi}/\ov J]
			\end{equation*}
			by taking the Laurent expansion of $\ph|_{D_{\infty}^{\times}}$ modulo $\frg\lr{\t}_{\le0}d\t/\t$. From the definitions we have a Cartesian diagram
			\begin{equation*}
				\xymatrix{\cM_{\psi}\ar[r]^{F}\ar[d] & \cN_{\psi}\ar[d]^{\e}\\
					[\{\psi\}/\ov C]\ar[r]^{\io} & [U_{\psi}/\ov J]}
			\end{equation*}
			To show $F$ is an isomorphism, it suffices to show that $\io$ is an isomorphism.
			
			Consider the action map $\a:\ov  J\to U_{\psi}$ sending $g\in \ov J$ to $\Ad(g)\psi\in U_{\psi}$. It passes to the quotient
			\begin{equation*}
				\ov \a: \ov J/\ov C\cong \bJ_{\infty}/\bK_{\infty}\to U_{\psi}
			\end{equation*}
			Then $\io$ is an isomorphism if and only if $\ov \a$ is an isomorphism. 
			
			For $d/2\le j\ge d$ and $d\in \ZZ$, let
			\begin{eqnarray*}
				\bQ_{j}=\bP_{\infty}(\frac{j}{m})\bC_{\infty}^{+},\\
				\L_{j}=\frc\lr{\t}^{\bot}_{\le (d-j)/m}+\frg\lr{\t}_{\le0}.
			\end{eqnarray*}
			We have
			\begin{eqnarray*}
				\bK_{\infty}=\bQ_{d}\subset\cdots\subset \bQ_{j}\subset \bQ_{j-1}\subset\cdots \bQ_{\lceil d/2\rceil}\subset \bJ_{\infty},\\
				\frg\lr{\t}_{\le0}=\L_{d}\subset\cdots\subset \L_{j}\subset \L_{j-1}\subset \cdots\subset\L_{\lceil d/2\rceil}\subset \frj^{\vee}_{\infty}.
			\end{eqnarray*}
			Moreover, $\Ad(\bQ_{j})\psi\subset \L_{j}$ for $j\ge d/2$.
			
			We show inductively that for $j\ge d/2$, the map
			\begin{equation*}
				\a_{j}: \bJ_{\infty}/\bQ_{j}\to (\psi+\frj^{\vee}_{\infty})/\L_{j}
			\end{equation*}
			defined by $g\mapsto \Ad(g)\psi\mod \L_{j}$ is an isomorphism. Since $\a_{d}=\ov \a$, this would finish the proof.
			
			When $d$ is odd, the initial step $j=(d+1)/2$ is trivial since both sides of $\a_{j}$ are singletons. When $d$ is even, we have $\bJ_{\infty}/\bQ_{d/2}\cong\fm/\frc_{-d/(2m)}$, and $(\psi+\frj^{\vee}_{\infty})/\L_{d/2}\cong\psi+\fm^{\bot}$, and the map $\a_{d/2}$ is $[-,\psi]$. By definition $\fm$ is a maximal isotropic subspace of $\frg_{-d/(2m)}$ under the form $(x,y)\mapsto \j{x,[y,\psi]}$. This form has kernel $\frc_{-d/(2m)}$, hence $[-,\psi]$ maps $\fm/\frc_{-d/(2m)}$ isomorphically to $\fm^{\bot}$.

			Now assume $\a_{j}$ is an isomorphism (where  $d/2\le j<d$). We have a commutative diagram
			\begin{equation*}
				\xymatrix{\bJ_{\infty}/\bQ_{j+1}\ar[r]^-{\a_{j+1}}\ar[d]^{q_{j}} &  (\psi+\frj^{\vee}_{\infty})/\L_{j+1}\ar[d]^{p_{j}}\\
					\bJ_{\infty}/\bQ_{j}\ar[r]^-{\a_{j}} & (\psi+\frj^{\vee}_{\infty})/\L_{j}
				}
			\end{equation*}
			Since the above diagram is $\bJ_{\infty}$-equivariant and the map $\a_{j}$ is assumed to be an isomorphism, to show $\a_{j+1}$ is an isomorphism it suffices to show that
			\begin{equation*}
				\a_{j+1}|_{q_{j}^{-1}(1)}: q^{-1}_{j}(1)=\bQ_{j}/\bQ_{j+1}\to p_{j}^{-1}(\psi)=(\psi+\L_{j})/\L_{j+1}
			\end{equation*}
			is an isomorphism. This map can be identified with
			\begin{equation*}
				[-,\psi]: \frc\lr{\t}^{\bot}_{-j/m}\to \frc\lr{\t}^{\bot}_{(d-j)/m}.
			\end{equation*}
			Since $\psi$ is regular semisimple with centralizer $\frc\lr{\t}$ under $\frg\lr{\t}$, the above map is an isomorphism. This completes the inductive step. 
		\end{proof}

		\begin{lemma}\label{l:dimA} The Hitchin base $\cA_{\psi}$ is an affine space of dimension
			$\dim \cA_{\psi}=\frac{1}{2}(\frac{d}{m}|\Phi|-r+\dim \frt^{w})$.
		\end{lemma}
		\begin{proof}
			From the definition we have $\dim \cA_{\psi}$ is the same as the space of $(a_{1},\cdots,a_{r})$ where $a_{i}$ is a section of $\cO([\frac{d(d_{i}-1)}{m}]\cdot \infty)$ vanishing at $0$. Therefore
			\begin{equation*}
				\dim \cA_{\psi}=\sum_{i=1}^{r}\left[\frac{d(d_{i}-1)}{m}\right].
			\end{equation*}
			We have 
			$$\sum\frac{d(d_{i}-1)}{m}=\frac{d}{m}\sum_{i}(d_{i}-1)=\frac{d}{m}\frac{|\Phi|}{2}.$$
			Therefore we reduce to showing
			\begin{equation}\label{sum frac}
				\sum_{i}\{(d_{i}-1)/m\}=(r-\dim \frt^{w})/2.
			\end{equation}
			Here $\{\cdots\}$ denotes the fractional part. Let $\z\in\mu_{m}$ be a primitive $m$th root of unity. We claim that
			\begin{equation}\label{eval w}
				\mbox{The eigenvalues of $w$ on $\frt$ are  $\{\z^{d_{i}-1}\}_{1\le i\le r}$ as a multi-set.}
			\end{equation}
			Indeed, the $\ZZ/m\ZZ$-grading on $\frg$ is conjugate to one given by $\Ad(\r^{\vee}(\z))$, therefore we may assume it is given by $\Ad(\r^{\vee}(\z))$.  Consider the space $\frg_{1/m}^{\reg}=\frg^{\reg}\cap \frg_{1/m}$. We claim that for $x\in \frg_{1/m}^{\reg}$, the action of $\r^{\vee}(\z)$ on $\frz_{x}$ (the centralizer of $x$ in $\frg$) has eigenvalues $\{\z^{d_{i}-1}\}_{1\le i\le r}$ as a multi-set.  Applying this to a regular semisimple $x$ gives \eqref{eval w}. Since $\frg_{1/m}^{\reg}$ is open dense in $\frg_{1/m}$, hence it is connected. The Lie algebra universal centralizer $\frz\to \frg^{\reg}_{1/m}$ is a vector bundle with an action of $\Ad(\r^{\vee}(\z))$, hence all fibers have the same multi-set of eigenvalues under $\Ad(\r^{\vee}(\z))$. Now we take $x$ to be the regular nilpotent element $x_{0}=\sum e_{\a_{i}}$ (for simple roots $\a_{i}$, noting that $\frg_{\a_{i}}\subset \frg_{1/m}$). Since the weights of $\Ad(\r^{\vee})$ on $\frz_{x_{0}}$ are the exponents of $\frg$ by definition, i.e., $\{d_{i}-1\}$, we see that the eigenvalues of  $\Ad(\r^{\vee}(\z))$ on $\frz_{x_{0}}$ are $\{\z^{d_{i}-1}\}$.
			
			Therefore, when summing up $\{(d_{i}-1)/m\}$, each pair of eigenvalues $\l,\l^{-1}$ ($\l\ne\pm1$) of $w$ on $\frt$ contributes $1$; $\l=-1$ contributes $1/2$ and $\l=1$ contributes zero. Hence \eqref{sum frac}.
		\end{proof}

		\subsection{Proof of Theorem \ref{th:geom M}\eqref{th part:int sys}}
		By Theorem \ref{th:geom M}\eqref{th part:M sm} and Lemma \ref{l:dimA}, we see that $\dim\cM_{\psi}=2\dim\cA_{\psi}$. Therefore all fibers of $f$ have dimension $\ge \dim\cA_{\psi}$. Also, by Theorem \ref{th:geom M}\eqref{th part:ASF vs HF}, the central fiber of $f$ has dimension $\dim\Fl_{\psi}=\dim \cA_{\psi}=\dim\cM_{\psi}-\dim \cA_{\psi}$. Since all points in $\cA_{\psi}$ contract to $a_{\psi}$ under the $\Gm$-action,  all fibers of $f$ have dimension $\le \dim\Fl_{\psi}$. Combine both inequalities, we conclude that all fibers of $f$ have dimension equal to $\dim\cA_{\psi}$.  Since $\cM_{\psi}$ and $\cA_{\psi}$ are both smooth, $f$ is flat of relative dimension equal to $\dim\cA_{\psi}$.
		
		It remains to check that the functions given by coordinates of $\cA_{\psi}$ are Poisson commuting. Let $(\cE,\ph)\in \cM_{\psi}$ with image $a\in \cA_{\psi}$. We need to show that the image of the cotangent map $f^{*}: T_{a}\cA_{\psi}\to T^{*}_{(\cE,\ph)}\cM_{\psi}$ is isotropic. 
		
		Let  $\cF=\bigoplus_{i=1}^{r}\cO([d(d_{i}-1)/m]\cdot \infty - \un 0)$, where $\un 0$ denotes the point $0\in X$. Then $T_{a}\cA_{\psi}$ is identified with $\cohog{0}{X,\cF}$. Recall from the proof of Theorem \ref{th:geom M} that the tangent space $T_{(\cE,\ph)}\cM_{\psi}\cong \cohog{0}{X,\cK}$. The tangent map $f_{*}: T_{(\cE,\ph)}\cM_{\psi}\to T_{a}\cA_{\psi}$ is induced from the following map of coherent complexes on $X$ by taking global sections
		\begin{equation*}
			\xymatrix{ \Ad(\cE; \Lie\bI_{0}, \Lie \bK_{\infty}) \ar[r]^{[-,\ph]}& \Ad(\cE; \Lie\bI^{+}_{0}, \frg\lr{\t}_{\le0})\ar[d]^{(df_{i})_{1\le i\le r}}\\
				& \cF}
		\end{equation*}
		Here $df_{i}: \Ad(\cE; \Lie\bI^{+}_{0}, \frg\lr{\t}_{\le0})\to \cO([d(d_{i}-1)/m]\cdot \infty - \un 0)$ is the $\cO_{X}$-linear map given fiberwise by the differential of $f_{i}$ at $\ph$. The map $(df_{i})_{1\le i\le r}$ above factors through $p: \cH^{0}\cK\to \cF$. The tangent map $f_{*}$ at $(\cE,\ph)$ is thus given by
		\begin{equation*}
			T_{(\cE,\ph)}\cM_{\psi}=\cohog{0}{X,\cK}\surj \cohog{0}{X,\cH^{0}\cK}\xr{p} \cohog{0}{X,\cF}=T_{a}\cA_{\psi}.
		\end{equation*}
		Dually, the cotangent map $f^{*}$ at $(\cE,\ph)$ is given by
		\begin{equation*}
			T^{*}_{a}\cA_{\psi}=\cohog{1}{X, \cF^{*}\ot\om_{X}}\xr{p^{*}}\cohog{1}{X, \cH^{-1}(\cK^{\vee})}\incl \cohog{0}{X,\cK^{\vee}}=T^{*}_{(\cE,\ph)}\cM_{\psi}.
		\end{equation*} 

		Here $\cK^{\vee}=\uHom(\cK,\om_{X}[1])$ is the Serre dual of $\cK$, and $(-)^{*}$ denotes linear dual $\uHom(-,\cO_{X})$. In the proof of Theorem \ref{th:geom M}, we showed that the obvious map $\io: \cK\to \cK^{\vee}$ is a quasi-isomorphism, which then induces 
		an isomorphism of short exact sequences
		\begin{equation}\label{two ex K}
			\xymatrix{0\ar[r] & \cohog{1}{X,\cH^{-1}\cK}\ar[r]\ar[d]^{\cong} &  T_{(\cE,\ph)}\cM_{\psi}\ar[d]^{\cong}_{\io} \ar[r] & \cohog{0}{X,\cH^{0}\cK}\ar[d]^{\cong} \ar[r] & 0\\
				0\ar[r] & \cohog{1}{X,\cH^{-1}(\cK^{\vee})}\ar[r]&  T^{*}_{(\cE,\ph)}\cM_{\psi} \ar[r] & \cohog{0}{X,\cH^{0}(\cK^{\vee})}\ar[r] & 0
			}
		\end{equation}
		By construction, the middle vertical map gives the symplectic form on $T_{(\cE,\ph)}\cM_{\psi}$. The composition
		\begin{equation*}
			T^{*}_{a}\cA_{\psi}\xr{f^{*}}T^{*}_{(\cE,\ph)}\cM_{\psi}\xr{\io^{-1}}T_{(\cE,\ph)}\cM_{\psi}\xr{f_{*}}T_{a}\cA_{\psi}
		\end{equation*}
		factors through the composition of either row in \eqref{two ex K}, hence is zero. This shows that the image of $T^{*}_{a}\cA_{\psi}$ in $T^{*}_{(\cE,\ph)}\cM_{\psi}$ is isotropic. \qed

		\subsection{Construction of $\cM_{\psi}$ as a symplectic leaf}\label{ss:cons3}
		Let $\cM^{\da}_{\psi}$ be the moduli stack parametrizing pairs $(\cE, \ph)$ where
		\begin{itemize}
			\item $\cE$ is a $G$-bundle over $X$ with $\bP^{+}_{\infty}$-level structure at $\infty$ and $\bI_{0}$-level structure at $0$. 
			\item $\ph$ is a section of $\Ad(\cE)\ot\om_{X\bs \{0,\infty\}}$ satisfying the following conditions: 
			\begin{enumerate}
				\item[(i)] Under some (equivalently, any) trivialization of $\cE|_{D_{\infty}}$ together with its $\bP^{+}_{\infty}$-level structure, we require
				\begin{equation*}
					\ph|_{D_{\infty}^{\times}}\in (\psi+\frg\lr{\t}_{\le (d-1)/m})d\t/\t.
				\end{equation*}
				\item[(ii)] Under some (equivalently, any) trivialization of $\cE|_{D_{0}}$ together with its $\bI_{0}$-level structure, we require
				\begin{equation*}
					\ph|_{D^{\times}_{0}}\in \Lie(\bI^{+}_{0})dt/t.
				\end{equation*}
			\end{enumerate}
		\end{itemize}
		
		The Hitchin base $\cA^{\da}_{\psi}$ for $\cM^{\da}_{\psi}$ is the closed subscheme $\cA^{\da}_{\psi}\subset \prod_{i=1}^{r}\G(\PP^{1}, \cO([\frac{dd_{i}}{m}]\cdot\infty))$ of sections $a=(a_{i})_{1\le i\le r}$ such that for each $i=1,\cdots, r$
		\begin{itemize}
			\item $a_{i}(0)=0$;
			\item $a_{i}\equiv f_{i}(\psi)\mod \t^{-\frac{dd_{i}-1}{m}}$ near $\infty$ (i.e., the leading coefficient of degree $\t^{-dd_{i}/m}$, if $m|d_{i}$, of $a_{i}$ at $\infty$ is the same as that of $f_{i}(\psi)$).
		\end{itemize}
		
		We have the Hitchin map 
		\begin{equation*}
			f^{\da}: \cM_{\psi}^{\da}\to \cA^{\da}_{\psi}.
		\end{equation*}
		It is also clear from the construction that there is a canonical map $\cM_{\psi}\to \cM^{\da}_{\psi}$ and an inclusion $\cA_{\psi}\subset \cA^{\da}_{\psi}$.
		
		\begin{prop}\label{p:Cart Mda} The canonical maps give a Cartesian diagram
			\begin{equation*}
				\xymatrix{ \cM_{\psi}\ar[d]^{f}\ar[r] & \cM^{\da}_{\psi}\ar[d]^{f^{\da}}\\
					\cA_{\psi}\ar@{^{(}->}[r] & \cA^{\da}_{\psi}}
			\end{equation*}
			In particular, $\cM_{\psi}\cong \cM^{\da}_{\psi}\times_{\cA^{\da}_{\psi}}\cA_{\psi}$.
		\end{prop}
		\begin{proof}
			Let $Q=\bP_{\infty}^{+}/\bP_{\infty}(\frac{d}{m})$. Let $V_{\psi}=(\psi+\frg\lr{\t}_{\le(d-1)/m})/\frg\lr{\t}_{\le 0}$. Then $Q$ acts on the affine space $V_{\psi}$ by the adjoint action. Each $(\cE,\ph)\in \cM^{\da}_{\psi}$ gives an $Q$-torsor $\cQ$ induced from the $\bP_{\infty}^{+}$-level structure of $\cE$, and the polar terms of $\ph(d\t/\t)^{-1}$ give a section of associated the affine space $\cQ\twtimes{Q}V_{\psi}$. This construction gives a map
			\begin{equation*}
				\e: \cM^{\da}_{\psi}\to [V_{\psi}/Q].
			\end{equation*}
			The stabilizer $Q_{\psi}$ of $\psi\in V_{\psi}$ is the image of $\bC^{+}_{\infty}$ in $Q$. By the definition of $\cM_{\psi}$, we have a Cartesian diagram
			\begin{equation}\label{Mpsi fiber e}
				\xymatrix{\cM_{\psi}\ar[d]\ar[r] & \cM^{\da}_{\psi}\ar[d]^{\e}\\
					[\{\psi\}/Q_{\psi}]\ar@{^{(}->}[r] & [V_{\psi}/Q]}
			\end{equation}
			
			Let $\fra_{\psi}$ be the affine space of $(\ov a_{i})_{1\le i\le r}$ where $\ov a_{i}\in \t^{-[dd_{i}/m]}k\tl{\t}/k\tl{\t}$ with leading term $f_{i}(\psi)$ in degree $\t^{-dd_{i}/m}$ if $m|d_{i}$. Then $(f_{i})_{1\le i\le r}$ gives a map
			\begin{equation*}
				\ov f: V_{\psi}\to \fra_{\psi}.
			\end{equation*}
			We have a map $\pi: \cA^{\da}_{\psi}\to \fra_{\psi}$ by taking the first few terms of the Laurent expansion $a_{i}$ of at $\infty$. Let $\ov a_{\psi}=\pi(a_{\psi})\in \fra_{\psi}$. Then $\cA_{\psi}=\pi^{-1}(\ov a_{\psi})$. Let $V^{0}_{\psi}:=\ov f^{-1}(\ov a_{\psi})$. Therefore
			\begin{equation*}
				\cM^{\da}_{\psi}\times_{\cA^{\da}_{\psi}}\cA_{\psi}=\e^{-1}([V^{0}_{\psi}/Q]).
			\end{equation*}
			In view of \eqref{Mpsi fiber e}, to show that $\cM_{\psi}$ is equal to the left side above, it suffices to show that $[\{\psi\}/Q_{\psi}]\cong [V^{0}_{\psi}/Q]$, or equivalently, 
			\begin{equation}\label{Q transitive}
				\mbox{The action map $\a: Q\to V^{0}_{\psi}$, $q\mapsto\Ad(q)\psi$, is smooth and surjective. }
			\end{equation}
			
			We first show that $\a$ is surjective on $k$-points. Let $\psi'\in\psi+\frg\lr{\t}_{\le (d-1)/m}$ be such that $f_{i}(\psi')=f_{i}(\psi)$ for all $1\le i\le r$. We want to construct $q\in \bP_{\infty}^{+}$ such that $\Ad(q)\psi-\psi'\in \frg\lr{\t}_{\le 0}$. The same argument as in the proof of Lemma \ref{l:C trans} works to construct $h_{j}$ inductively modulo $\bP_{\infty}(\frac{j}{m})$ for $j=1,2,\cdots, d$ such that $\Ad(q)\psi-\psi'\in \frg\lr{\t}_{\le (d-j)/m}$. We omit details.
			
			We then show that $\a$ is smooth. Since $(V^{0}_{\psi})(k)$ is a single orbit of $Q(k)$, it suffices to show that the tangent map of $\a$ is surjective at $1$. The tangent map of $\a$ at $1$ is 
			$$[-,\psi]: \Lie Q=\bigoplus_{j=1}^{d-1}\frg\lr{\t}_{(j-d)/m}\to \bigoplus_{j=1}^{d-1}\frc\lr{\t}^{\bot}_{j/m},$$ 
			which is surjective since $\psi$ is regular semisimple. This finishes the proof of \eqref{Q transitive}. The proposition is proved.
		\end{proof}

		\begin{remark} One can show that $\cM^{\da}_{\psi}$ is a smooth Poisson algebraic space, and $\cM_{\psi}$ is a symplectic leaf in $\cM^{\da}_{\psi}$.  We omit the proof.
		\end{remark}
		
		\subsection{Proof of Theorem \ref{th:geom M}\eqref{th part:f proper}}
		Assuming $w$ is elliptic, we show that $f$ is proper. By Proposition \ref{p:Cart Mda}, it suffices to show that $f^{\da}$ is proper. We introduce a variant $\cM^{\dda}$ of $\cM^{\da}_{\psi}$: it classifies $(\cE,\ph)$ where $\cE$ has $\bP_{\infty}$ and $\bI_{0}$-level structures, and the Higgs field is required to lie in $\frg\lr{\t}_{\le d/m}d\t/\t$ near $\infty$ such that its projection to $\frg\lr{\t}_{d/m}$ is regular semisimple, and in $\Lie \bI_{0}^{+}dt/t$ near 0 (after trivializations). Let $\wt\cA^{\dda}$ be the affine space of $(a_{i}\in\G(X,\cO([dd_{i}/m])\cdot \infty)_{1\le i\le r}$ with the condition that $a_{i}(0)=0$.  Evaluating the leading coefficient at $\infty$ gives a map $\wt\cA^{\dda}\to \frg_{d/m}\sslash L_{\bP_{\infty}}$, and let $\cA^{\dda}\subset \wt\cA^{\dda}$ be the preimage of $\frg^{\rs}_{d/m}\sslash L_{\bP_{\infty}}$ (where $L_{\bP_{\infty}}$ is the Levi quotient of $\bP_{\infty}$; it is identified the connected subgroup of $G$ with Lie algebra $\frg_{0}$). We have the Hitchin fibration $f^{\dda}: \cM^{\dda}\to \cA^{\dda}$. Then the fiber product $\cM^{\dda}\times_{\cA^{\dda}}\cA^{\da}_{\psi}$ admits a description that is almost identical to $\cM^{\da}_{\psi}$, except that the $\bP^{+}_{\infty}$-level structure is replaced with the slightly larger level group $\bP^{+}_{\infty}\bC'_{\infty}$, where $\bC'_{\infty}\subset \bP_{\infty}$ is the centralizer of $\psi$ in $\bP_{\infty}$. Since $(\bP^{+}_{\infty}\bC'_{\infty})/\bP^{+}_{\infty}\cong C_{L_{\bP_{\infty}}}(\psi)$ is finite over $k$ for $w$ elliptic, $\cM^{\da}_{\psi}\to \cM^{\dda}\times_{\cA^{\dda}}\cA^{\da}_{\psi}$ is finite. Therefore, to show $f^{\da}$ is proper it suffices to show that $f^{\dda}$ is proper. Now $f^{\dda}$ is a parahoric Hitchin fibration, and \cite[Proposition 6.3.7(2)]{OY} implies that $f^{\dda}$ is proper over the elliptic locus. However, we shall argue that the whole $\cA^{\dda}$ consists of elliptic points. Indeed, the non-elliptic locus $Z\subset \cA^{\dda}$ is closed and $\Gm(\nu)$-stable, so must contain a $\Gm(\nu)$-fixed point if non-empty. But $\cA^{\dda,\Gm(\nu)}$ consists of images of $\psi'$ for $\psi'\in \frg[t,t^{-1}]_{d/m}$, and they are all elliptic.
		\qed

		\subsection{Comparison of cohomology}\label{ss:comp M Fl}	
		The goal of this subsection is to prove Theorem \ref{th:coho M Fl} about the cohomology of $\cM_{\psi}$ and $\Fl_{\psi}$.
		
		\sss{The situation}\label{sss:Gm action}We consider the following general situation. Let $\frX$ be an algebraic space locally of finite type over $k$, equipped with a $\Gm$-action. Let $\frX^{\Gm}=\coprod_{\a\in I} Z_{\a}$ be an open-closed decomposition of the fixed point locus. 
		
		In \cite{Dr} Drinfeld introduces the attractor $\frX^{+}:=\Map_{\Gm}(\AA^{1}, \frX)$. It is equipped with two maps
		\begin{equation*}
\xymatrix{ \frX^{\Gm} & \frX^{+}\ar[r]^{\frp^{+}}\ar[l]_{\frq^{+}} & \frX}
\end{equation*}
Here $\frp^{+}$ (resp. $\frq^{+}$) is evaluation at $1\in \AA^{1}$ (resp. $0\in \AA^{1}$). It is shown in \cite[Theorem 1.4.2]{Dr} that $\frX^{+}$ is represented by an algebraic space of finite type over $k$, and $\frq^{+}$ is an affine morphism. 

		One defines the repeller $\frX^{-}$  to be the attractor for the inverted $\Gm$-action, and we have maps $\frp^{-}:\frX^{-}\to \frX$ and $\frq^{-}: \frX^{-}\to \frX^{\Gm}$. 
		
		For each $\a\in I$, let $\frX^{\pm}_{\a}=\frq^{\pm, -1}(Z_{\a})$. Let $\frp^{\pm}_{\a}=\frp^{\pm}|_{\frX^{\pm}_{\a}}$; $\frq^{\pm}_{\a}=\frq^{\pm}|_{\frX^{\pm}_{\a}}$.
						
		We make the following assumptions:
		\begin{enumerate}
		\item For each $\a\in I$,  $\frq^{+}_{\a}: \frX^{+}_{\a}\to \frX$ is a locally closed embedding. 
		
		\item For each $\a\in I$,  the reduced image of $\frq^{-}_{\a}: \frX^{-}_{\a}\to \frX$ is a locally closed subspace $X^{-}_{\a}$ of $\frX$, and the induced map $\frX^{-}_{\a}\to X_{\a}^{-}$ is a homeomorphism.

\item $\cup_{\a\in I}\frX^{+}_{\a}=\frX$, i.e., for any $x\in \frX$, the limit $\lim_{t\to 0}t\cdot x$ exists.

\item There exists a partial order $\le$ on $I$ such that for each $\a\in I$
		\begin{itemize}
\item The set $\{\a'\in I;\a'\in \a\}$	is finite.	
\item $\frX^{+}_{\le \a}:=\cup_{\a'\le \a}\frX^{+}_{\a}$ is open in $\frX$, and is of finite type over $k$.
\item $X^{-}_{\le \a}:=\cup_{\a'\le \a}X^{-}_{\a}$ is proper over $k$. In particular, each $Z_{\a}$ is proper over $k$.
\end{itemize}

\end{enumerate}
 		The cohomology of a locally finite algebraic space is the projective limit of the cohomology of finite type open subspaces. Therefore 
		\begin{equation*}
\cohog{*}{\frX,\Qlbar}=\varprojlim_{\a\in I}\cohog{*}{X_{\le \a}^{+}, \Qlbar}.
\end{equation*}
		 
  		On the other hand, we define an ind-space
		\begin{equation*}
\frY:=\varinjlim_{\a\in I}X^{-}_{\le\a}
\end{equation*}
as a union of finite type closed subspaces. We define the cohomology of $\frY$ to be 
		\begin{equation*}
\cohog{*}{\frY,\Qlbar}:=\varprojlim_{\a\in I}\cohog{*}{X_{\le \a}^{-}, \Qlbar}.
\end{equation*}
		
  	\begin{prop}\label{p:res coho} Under the above assumptions,  the restriction maps $\cohog{*}{\frX,\Qlbar}\to \cohog{*}{\frY,\Qlbar}$ and $\upH^{*}_{\Gm}(\frX,\Qlbar)\to \upH^{*}_{\Gm}(\frY,\Qlbar)$ are isomorphisms.
	\end{prop}	
	\begin{proof}
	First note that $X^{-}_{\le \a}\subset \frX^{+}_{\le \a}$. Indeed, if $x\in X^{-}_{\le \a}$, the limit $\lim_{t\to 0}t\cdot x$ exists in $X^{-}_{\le \a}$ since $X^{-}_{\le \a}$ is proper. But $(X^{-}_{\le \a})^{\Gm}=\coprod_{\a'\in\a}Z_{\a'}$, hence $\lim_{t\to 0}t\cdot x\in Z_{\a'}$ for some $\a'\le \a$, therefore $x\in \frX^{+}_{\le \a}$. 	
	
	We extend $\le$ to a total ordering on $I$ and add the minimal element $0$ to $I$. Let $\frX^{+}_{0}=\frX^{-}_{0}=X^{-}_{0}=\varnothing$. We denote by $i^{\pm}_{\a}: Z_{\a}\incl \frX^{\pm}_{\a}$ the inclusions. Let $s_{\a}: X^{-}_{\le \a}\incl \frX^{+}_{\le \a}$ be the inclusion. If $\a'$ is the predecessor of $\a$, let $\frX^{\pm}_{<\a}:=\frX^{\pm}_{\le\a'}$, $X^{-}_{<\a}:=X^{-}_{\le\a'}$ and let $s_{<\a}=s_{\a'}$.
	
	We prove by induction on $\a$ that the restriction map $s_{\a}^{*}: \cohog{*}{\frX^{+}_{\le\a}}\to \cohog{*}{X^{-}_{\le\a}}$ is an isomorphism, and the same is true for $\Gm$-equivariant cohomology. This would imply the proposition by taking projective limits. The case $\a=0$ is clear. 
	
	Suppose the $s_{<\a}^{*}$ is an isomorphism, we show $s_{\a}^{*}$ is also an isomorphism. Consider the commutative diagram (for simplicity we have omitted most of the subscripts $\a$)
	\begin{equation*}
\xymatrix{Z_{\a}\ar[d]^{i^{-}}\ar[rr]^{i^{+}} &  & \frX^{+}_{\a}\ar[d]^{k^{+}}\\
\frX^{-}_{\a}\ar[r]^{j^{-}}\ar@/^{4ex}/[rr]^{k^{-}} & X^{-}_{\le \a} \ar[r]^{s_{\a}} & \frX^{+}_{\le \a}\\
& X^{-}_{<\a} \ar[u]^{v}\ar[r] & \frX^{+}_{<\a}\ar[u]^{u}}
\end{equation*}
Here, $i^{+},i^{-}, k^{+}, s_{\a}, s_{<\a}, v$ are closed embeddings, $j^{-}$ is the composition $\frX^{-}_{\a}\to X^{-}_{\a}\incl X^{-}_{\le\a}$ hence a topological open embedding;  $u$ is an open embedding.

	We consider the following diagram of distinguished triangles in $D^{b}_{\Gm}(X^{+}_{\le\a})$:
\begin{equation}\label{two ex}
\xymatrix{ k^{+}_{!}k^{+!}\Qlbar\ar[r]\ar[d]^{\z} &   \Qlbar \ar[r] \ar[d]& u_{*}u^{*}\Qlbar \ar[d]\\
k^{-}_{!}k^{-*}\Qlbar\ar[r] & s_{\a*}s_{\a}^{*}\Qlbar \ar[r] & s_{<\a*}s_{<\a}^{*}\Qlbar
}
\end{equation}
The rows are given by the open-closed decompositions $\frX^{+}_{\le\a}=\frX^{+}_{<\a}\cup \frX^{+}_{\a}$ and $X^{-}_{\le\a}=X^{-}_{\a}\cup X^{-}_{<\a}$ (by assumption $\frX_{\a}^{-}\to X_{\a}^{-}$ is a homeomorphism). The only map that requires explanation is $\z$, which is the composition natural transformations
\begin{equation*}
\z: k^{+}_{*}k^{+!}\to s_{\a*}s_{\a}^{*}k^{+}_{*}k^{+,!}\xr{\xi} s_{\a*}h^{-}_{*}i^{+*}k^{+!}\xr{\r}s_{\a*}h^{-}_{*}i^{-!}k^{-*}=s_{\a*}j^{-}_{!}i^{-}_{*}i^{-!}j^{-*}s_{\a}^{*}\to s_{\a*}j^{-}_{!}j^{-*}s_{\a}^{*}=k^{-}_{!}k^{-*}.
\end{equation*}
Here $\xi: s_{\a}^{*}k^{+}_{*}\to h^{-}_{*}i^{+*}$  is given by $k^{+}\c i^{+}=s_{\a}\c h^{-}$ and adjunction; $\r: i^{+*}k^{+!}\to i^{-!}k^{-*}$ is the comparison map of hyperbolic  localization functors, see Braden \cite[top of page 212]{Br}. The commutativity of \eqref{two ex} is a diagram chase that we omit.

	Taking global sections of \eqref{two ex} we get a map between long exact sequences of cohomology groups
\begin{equation*}
\xymatrix{ \cdots \ar[r] & \cohog{*}{\frX^{+}_{\a}, k^{+!}\Qlbar}\ar[r]\ar[d]^{\cohog{*}{\z}} & \cohog{*}{\frX_{\le \a}^{+}} \ar[r]\ar[d]^{s_{\a}^{*}}& \cohog{*}{\frX_{< \a}^{+}}\ar[d]^{s_{<\a}^{*}}\ar[r] & \cdots\\
\cdots \ar[r] &\cohoc{*}{\frX^{-}_{\a}, k^{-*}\Qlbar}\ar[r] & \cohog{*}{X_{\le \a}^{-}}\ar[r] & \cohog{*}{X_{< \a}^{-}}\ar[r]& \cdots}
\end{equation*}
By induction hypothesis, $s_{<\a}^{*}$ is an isomorphism, therefore to show $s_{\a}^{*}$ is an isomorphism, it suffices to show $\cohog{*}{\z}$ is.  Since $\frX^{+}_{\a}$ contracts to $Z_{\a}$ under $\frq^{+}_{\a}:\frX^{+}_{\a}\to Z_{\a}$, the contraction principle gives an isomorphism $\frq^{+}_{\a*}k^{+!}\Qlbar\cong i^{+*}k^{+!}\Qlbar$. Taking global sections we get 
\begin{equation*}
\cohog{*}{\frX^{+}_{\a}, k^{+!}\Qlbar}\cong \cohog{*}{Z_{\a}, i^{+*}k^{+!}\Qlbar}.
\end{equation*}
Similarly,  using the contraction $\frq^{-}_{\a}: \frX^{-}_{\a}\to Z_{\a}$ we have an isomorphism $\frq^{-}_{\a!}k^{-*}\Qlbar\cong i^{-!}k^{-*}\Qlbar$; taking global sections with compact support and using that $Z_{\a}$ is proper, we get
\begin{equation*}
 \cohoc{*}{\frX^{-}_{\a}, k^{-*}\Qlbar}\cong \cohog{*}{Z_{\a}, i^{-!}k^{-*}\Qlbar}.
\end{equation*}
Under these isomorphisms, the induced map on cohomology by $\z$ is the comparison map of hyperbolic localizations to $Z_{\a}$:
\begin{equation*}
\cohog{*}{\z}=\cohog{*}{\r}: \cohog{*}{Z_{\a}, i^{+*}k^{+!}\Qlbar}\to \cohog{*}{Z_{\a}, i^{-!}k^{-*}\Qlbar}.
\end{equation*}
By Braden's theorem \cite[Theorem 1]{Br}  and its extension to algebraic spaces by Drinfeld-Gaitsgory \cite[Theorem 3.1.6]{DG}, $\cohog{*}{\r}$ is an isomorphism. This shows $\cohog{*}{\z}$ is an isomorphism, hence $s_{\a}^{*}$ is an isomorphism.

	Taking $\Gm$-equivariant global sections of \eqref{two ex}, the same argument proves that $s_{\a}^{*}$ is an isomorphism on $\Gm$-equivariant cohomology.
	\end{proof}
	
	\sss{Attractors and repellers for $\cM_{\psi}$}\label{sss:att rep M} We would like to apply the above discussions to $\cM_{\psi}$ with the $\Gm(\nu)$-action. For this we collect some facts about its attractors and repellers.

		Since the action of $\Gm(\nu)$ on $\cA_{\psi}$ contracts to the point $a_{\psi}$, the $\Gm(\nu)$-fixed points $\cM_{\psi}^{\Gm(\nu)}$ necessarily lie in the central fiber $\cM_{a_{\psi}}$, hence homeomorphic to $\Fl^{\Gm(\nu)}_{\psi}$ by Theorem \ref{th:geom M}\eqref{th part:ASF vs HF}. 
		
		Recall $\bP_{0}\subset G\lr{t}$ is the parahoric subgroup whose Lie algebra is $\frg\lr{t}_{\ge0}$. The $\bP_{0}$-orbits on $\Fl$ are parametrized by $W_{\bP}\bs \tilW$, where $W_{\bP}$ is the Weyl group of the Levi $L_{\bP}$ of $\bP_{\infty}$. For $w\in W_{\bP}\bs \tilW$, let $\Fl_{w}\subset \Fl$ be the $\bP_{0}$-orbit containing any lifting of $w$. Then $W_{\bP}\bs \tilW$ is equipped with the partial order $\le$ such that $w\le w'$ if and only if $\Fl_{w}\subset \ov{\Fl_{w'}}$. This is the partial order induced from the Bruhat order on $\tilW$, if we identify $W_{\bP}\bs \tilW$ as a subset of $\tilW$ using longest representatives. Let $\Fl_{w,\psi}=\Fl_{w}\cap\Fl_{\psi}$.

		Consider the action of $\Gm(\nu)$ on $\Fl$ given in \S\ref{sss:Gm action Fl} that stabilizes $\Fl_{\psi}$. Since the affine roots $\a+n\d$ in $\bP_{0}$ are those with $\a(\xi/m)+n\ge0$, and $\a(\xi/m)+n=0$ if and only if $\a+n\d$ is a root of the Levi $L_{\bP}$, the fixed points $\Fl^{\Gm(\nu)}$ is the disjoint union of $L_{\bP}w\bI_{0}/\bI_{0}$ for $w\in W_{\bP}\bs \tilW$. Therefore $\Fl^{\Gm(\nu)}_{\psi}$ is admits a decomposition
		\begin{equation}\label{Gm fixed Hess}
			\Fl^{\Gm(\nu)}_{\psi}=\coprod_{w\in W_{\bP}\bs \tilW} \cH_{\psi}(w), \quad \cH_{\psi}(w)=L_{\bP}w\bI_{0}/\bI_{0}\cap \Fl_{\psi}.
		\end{equation}
		If we choose a representative $\wt w\in \tilW$ of $w$, then $\cH_{\psi}(w)$ is isomorphic a Hessenberg variety for $L_{\bP}$ defined using the $L_{\bP}$-module $\frg\lr{t}_{d/m}$ and its subspace ${}^{\wt w}(\Lie\bI^{+}_{0})\cap \frg\lr{t}_{d/m}$:
		\begin{equation*}
			\cH_{\psi}(w)\cong \{h\in L_{\bP}/(L_{\bP}\cap {}^{\wt w}\bI_{0})|\Ad(h^{-1})(\psi)\in {}^{\wt w}(\Lie\bI^{+}_{0})\cap \frg\lr{t}_{d/m}\}.
		\end{equation*}
		
		Let  $\g: \Fl_{\psi}\to \cM_{\psi}$ be the canonical map. 
		
		We shall use the notations from \S\ref{sss:Gm action}. For $w\in W_{\bP}\bs \tilW$, let $Z_{w}=\g(\cH_{\psi}(w))$ (it is open-closed in $\cM_{\psi}^{\Gm}$). Let $\frX^{+}_{w}$ and $\frX^{-}_{w}$ be the attractor and repeller of $Z_{w}$, and we have maps $\frp^{\pm}_{w}: \frX^{\pm}_{w}\to \cM_{\psi}$ and $\frq^{\pm}_{w}: \frX^{\pm}_{w}\to Z_{w}$.

		\begin{lemma}\label{l:Zw}
The map $\g_{w}=\g|_{\cH_{\psi}(w)}: \cH_{\psi}(w)\to Z_{w}$ is an isomorphism.
\end{lemma}
\begin{proof}
Since $\cM_{\psi}$ is smooth, so is $Z_{w}$ by \cite[Prop. 1.4.20]{Dr}. It is well-known that $\cH_{\psi}(w)$ is smooth. The map $\g_{w}$ is a homeomorphism between smooth spaces, hence an isomorphism.
\end{proof}

		\begin{lemma}\label{l:repel}
		The image of the map $\frp^{-}_{w}: \frX^{-}_{w}\to \cM_{\psi}$ is the locally closed subspace $X^{-}_{w}:=\g(\Fl_{w,\psi})$. There is a unique isomorphism $\frX^{-}_{w}\cong \Fl_{w,\psi}$ compatible with the maps $\frp^{-}_{w}$ and $\g$. In particular, $\frp^{-}_{w}$ is a homeomorphism onto its image $X_{w}^{-}=\g(\Fl_{w,\psi})$.

		\end{lemma}
		\begin{proof}
 			We first show that the image of $\frp^{-}_{w}$ is $X^{-}_{w}=\g(\Fl_{w,\psi})$ (with the reduced structure). Since the action of $\Gm(\nu)$ is contracting to $a_{\psi}$, if $\lim_{s\to \infty}s\cdot (\cE,\ph)$ exists, $f(\cE,\ph)$ must be equal to $a_{\psi}$, i.e., $(\cE,\ph)\in \cM_{a_{\psi}}$. 	Now we can identify $(\cE,\ph)$ with a geometric point $g\bI_{0}\in \Fl_{\psi}$ under $\g$. Since the inverted $\Gm(\nu)$-action $s\cdot g=\rot(s^{m})\Ad(\xi(s))g$ contracts $\bP_{0}$ to $L_{\bP}$, $\Fl_{w}=\bP_{0}w\bI_{0}/\bI_{0}$ is the repelling subscheme of $L_{\bP}w\bI_{0}/\bI_{0}$ in $\Fl$. Therefore, $\lim_{s\to \infty}s\cdot g\bI_{0}\in\cH_{\psi}(w)$ if and only if $g\bI_{0}\in \Fl_{w,\psi}$.	
			
			From the above we also see that each geometric point of $\cM_{a_{\psi}}$  has a unique limit point under the action of $s\in \Gm(\nu), s\to\infty$. By \cite[Prop. 1.4.11(i)]{Dr}, $\frp^{-}_{w}:\frX^{-}_{w}\to \cM_{\psi}$ is unramified with image $X^{-}_{w}$. The uniqueness of limit points as $s\to \infty$ implies that $\frp^{-}_{w}$ is a monomorphism (geometric fibers are a reduced singleton). 
						
Since $\cM_{\psi}$ is smooth, by \cite[Prop. 1.4.20]{Dr}, the map $\frq^{-}_{w}: \frX^{-}_{w}\to Z_{w}$ is smooth. By \cite[\S4.5]{GKM}, the contraction map $q: \Fl_{w,\psi}\to \cH_{\psi}(w)$ is an iterated affine space bundle, hence also smooth. Therefore the homeomorphism $\Fl_{w,\psi}\to X_{w}^{-}$ is the normalization map. The map $\frp^{-}_{w}: \frX^{-}_{w}\to X_{w}^{-}$ thus uniquely lifts to $\wt p: \frX^{-}_{w}\to \Fl_{w,\psi}$. We have a commutative diagram
			\begin{equation*}
			\xymatrix{ \frX_{w}^{-}\ar[r]^{\wt p}\ar[d]^{\frq^{-}_{w}} & 	\Fl_{w,\psi}\ar[d]^{q}\\
Z_{w}\ar[r]^{\sim} & \cH_{\psi}(w)}
			\end{equation*}
The bottom map is an isomorphism by Lemma \ref{l:Zw}. We claim that $\wt q$ is an isomorphism. 

Note that $\wt p$ is a monomorphism because $\frq^{-}_{w}$ is; it is also surjective. Replacing $Z_{w}$ by its connected components, and taking preimages in $\frX_{w}^{-}$ and $\Fl_{w,\psi}$, we may assume $Z_{w}$ is connected. In this case, $\frX_{w}^{-}$ is connected because every point of it contracts to $Z_{w}$ under $s\to\infty$. Now $\wt p$ is a monomorphic surjection between two smooth connected spaces. By considering differentials we conclude that $\wt p$ is an isomorphism.  
		\end{proof}		
		
		On the other hand, recall that the isomorphism classes of $\Bun_{G}(\bP_{\infty}, \bI_{0})$ are also indexed by $W_{\bP}\bs \tilW$. Denote the locally closed substack with isomorphism class $w$ by $\Bun^{w}_{G}(\bP_{\infty}, \bI_{0})$. It is well-known that $\Bun^{w}_{G}(\bP_{\infty}, \bI_{0})\subset \ov{\Bun^{w}_{G}(\bP_{\infty}, \bI_{0})}$ if and only if $w\ge w'$ under the partial order already defined on $W_{\bP}\bs \tilW$.
		
		Let $\om: \cM_{\psi}\to \Bun_{G}(\bP_{\infty}, \bI_{0})$ be the forgetful map.
		
		\begin{lemma}\label{l:attract} For any $(\cE,\ph)\in \cM_{\psi}$, the  limit $\lim_{s\to 0}s\cdot (\cE,\ph)$ exists in $\cM^{\Gm(\nu)}_{\psi}$. Moreover, the map $\frp^{+}_{w}: \frX^{+}_{w}\to \cM_{\psi}$ is a locally closed embedding whose image is  $\om^{-1}(\Bun^{w}_{G}(\bP_{\infty}, \bI_{0}))$.
		\end{lemma}
		\begin{proof}

			We first show the limit $\lim_{s\to 0}s\cdot (\cE,\ph)$ always exists. Consider the uniformization map $u: \Fl=G\lr{t}/\bI_{0}\to \Bun_{G}(\bK_{\infty}, \bI_{0})$. We think of $\Fl$ as classifying a $\cE$-torsor on $X$ with $\bK_{\infty}$-level at $\infty$ and $\bI_{0}$-level structure at $0$ and a trivialization on $X\bs\{0\}$. Let $\cN$ be the base change of $\cM_{\psi}$ along $u$. Then $\cN$  can be described as the moduli space of pairs $(g\bI_{0}, \th)$ where $g\bI_{0}\in\Fl$ and $\th\in \frg[t,t^{-1}]_{\le0}$ such that
			\begin{equation*}
				\Ad(g^{-1})(\psi+\th)dt/t\in \Lie\bI_{0}^{+}dt/t.
			\end{equation*}
			Indeed, given $(g\bI_{0},\th)$, we define $\cE$ as the image of $g\bI_{0}$ under $u$, equipped with the Higgs field $\ph=(\psi+\th)dt/t$ over $X\bs \{0\}$.
			
			The action of $\Gm(\nu)$ on $\cM_{\psi}$ lifts to $\cN$, and is given by 
			\begin{equation*}
				s\cdot (g\bI_{0}, \th)=(\rot(s^{-m})\Ad(\xi(s^{-1}))g\bI_{0}, s^{d}\rot(s^{-m})\Ad(\xi(s^{-1}))\th).
			\end{equation*}
			Since $\th\in \frg[t,t^{-1}]_{\le0}$, $\lim_{s\to 0}s^{d}\rot(s^{-m})\Ad(\xi(s^{-1}))\th=0$. On the other hand, since $\Fl$ is ind-proper, $\lim_{s\to 0}(\rot(s^{-m})\Ad(\xi(s^{-1}))g)\bI_{0}\in \Fl^{\Gm(\nu)}$ exists. This shows that $\lim_{s\to 0} s\cdot (g\bI_{0}, \th)$ exists for any point $(g\bI_{0}, \th)\in \cN$. Since $\cN\to \cM_{\psi}$ is surjective and $\Gm(\nu)$-equivariant, the  same is true for $\cM_{\psi}$.

			We next show that the image of $\frp^{+}_{w}$ is $\om^{-1}(\Bun^{w}_{G}(\bP_{\infty}, \bI_{0}))$. In other words, for a geometric point $(\cE,\ph)$ of $\cM_{\psi}$, $\lim_{s\to 0}s\cdot (\cE,\ph)\in \g(\cH_{\psi}(w))$ if and only if the image of $\cE$ in $\Bun_{G}(\bP_{\infty}, \bI_{0})$ is the point $w$. Let $(g\bI_{0}, \th)\in \cN$ be a preimage of $(\cE,\ph)$. Then $\lim_{s\to 0}s\cdot (\cE,\ph)\in \g(\cH_{\psi}(w))$ if and only if $\lim_{s\to 0}\rot(s^{-m})\Ad(\xi(s^{-1}))g\bI_{0}\in L_{\bP}w\bI_{0}/\bI_{0}$. Now $\Fl^{\Gm(\nu)}=\coprod_{w\in W_{\bP}\bs \tilW} L_{\bP}w\bI_{0}/\bI_{0}$.  Note that $\Bun_{G}(\bP_{\infty}, \bI_{0})=\G_{\infty}\bs \Fl$ where $\G_{\infty}=\bP_{\infty}\cap G[t,t^{-1}]$, and the action of $\rot(s^{-m})\Ad(\xi(s^{-1}))$ contracts the group $\G_{\infty}$ to the Levi $L_{\bP}$. Therefore, the limit point  $\lim_{s\to 0}\rot(s^{-m})\Ad(\xi(s^{-1}))g\bI_{0}\in L_{\bP}w\bI_{0}/\bI_{0}$ if and only if $g\bI_{0}$ lies in the $\G_{\infty}$-orbit of $w$, i.e., the image of $g\bI_{0}$ (or $\cE$) in $\Bun_{G}(\bP_{\infty}, \bI_{0})$ is in $\Bun^{w}_{G}(\bP_{\infty}, \bI_{0})$. 
			
			So far we have shown that $\frp^{+}_{w}$ induces a map  $\wt p: \frX^{+}_{w}\to \Om_{w}:=\om^{-1}(\Bun^{w}_{G}(\bP_{\infty}, \bI_{0}))$. Finally we show that $\wt p$ is an isomorphism, therefore $\frp^{+}_{w}$ is a locally closed embedding. In Lemma \ref{l:g} below, we construct a smooth map $\g: \Om_{w}\to \cH_{\psi}(w)$ such that the following diagram is commutative
			\begin{equation}\label{XOm}
\xymatrix{ \frX^{+}_{w}\ar[d]^{\frq^{+}_{w}}\ar[r]^{\wt p} & \Om_{w}\ar[d]^{\g}\\
Z_{w} \ar[r]^{\sim} & \cH_{\psi}(w)}
\end{equation}
By \cite[Proposition 1.4.20]{Dr}, $\frq^{+}_{w}$ is smooth with connected fibers (since its contracting to $Z_{w}$). On the other hand $\g$ is also smooth. Moreover,  by \cite[Proposition 1.4.11(i)]{Dr}, $\wt p$ is unramified. In our situation $\wt p$ is a monomorphism and a surjection, therefore the . By comparing relative differentials of $\frq^{+}_{w}$ and $\g$, we conclude that $\wt p$ is an isomorphism. 
\end{proof}

\begin{lemma}\label{l:g}
There is a smooth map $\g: \Om_{w}\to \cH_{\psi}(w)$ making \eqref{XOm} commutative.
\end{lemma}
\begin{proof}
To see this we first construct a smooth map $\g: \Om_{w}\to \cH_{\psi}(w)$. Fix a point $\cE_{w}\in \Bun_{G}(\bK_{\infty}, \bI_{0})$ whose image in $\Bun_{G}(\bP_{\infty}, \bI_{0})$ is $w$. Let $H=\bP_{\infty}/\bP_{\infty}(d/m)$, $H^{+}=\bP^{+}_{\infty}/\bP_{\infty}(d/m)$ which acts on $V=\op_{i=1}^{d}\frg\lr{\t}_{i/m}$. Let $V_{w}=V\cap \Ad(w)(\frn+t\frg[t])$ and $\wt V_{w}=\frg\lr{\t}_{\le d/m}\cap \Ad(w)(\frn+t\frg[t])$.  Let
			\begin{equation*}
\wh\Om_{w}=\{(h,v)\in H\times \wt V_{w}|\Ad(h)v\in \psi+\frg\lr{\t}_{\le0}\}.
\end{equation*}
Let $A_{w}=\bP_{\infty}\cap \Ad(w)(G[t]\cap \bI^{+}_{0})$, which is the automorphism group of the point $w$ in $\Bun_{G}(\bP_{\infty}, \bI_{0})$. It acts on $\wh\Om_{w}$ on the right by $ (h,v)\cdot a=(ha, \Ad(a^{-1})v)$. Let $\psi_{V}$ be $\psi$ viewed as an element in $V$, and $C_{H}(\psi_{V})$ be its stabilizer under $H$. Then $C_{H}(\psi_{V})$ acts on $\wh\Om_{w}$ by left translation on $h\in H$. Then there is an isomorphism
			\begin{equation*}
C_{H^{+}}(\psi)\bs \wh \Om_{w}/A_{w}\cong \Om_{w}.
\end{equation*}

Let $\wt\Om_{w}=\{h\in H|\Ad(h^{-1})\psi_{V}\in V_{w}\}$. There is a natural map $\b: \wh\Om_{w}\to \wt\Om_{w}$ (sending $(h,v)$ to $h$) is an affine space bundle with fibers isomorphic to $\frg\lr{\t}_{\le 0}\cap \Ad(w)(\frn+t\frg[t])$. In particular, $\b$ is smooth.

Let $V_{w,i/m}=\frg\lr{\t}_{i/m}\cap \Ad(w)(\frn+t\frg[t])$.  Let $\wt \cH_{w}=\{\ell\in L_{\bP}|\Ad(\ell^{-1})\psi\in V_{w, d/m}\}$. Then $\wt\cH_{w}\to \cH_{\psi}(w)$ is a torsor under $B_{w}=L_{\bP}\cap \Ad(w)(G[t]\cap \bI^{+}_{0})$ (a Borel subgroup of $L_{\bP}$). We have a map $\a: \wt\Om_{w}\to \wt\cH_{w}$ sending $h\in H$ to its image in $L_{\bP}$.  Note that $B_{w}$ is a quotient of $A_{w}$, and the composition $\a\b: \wh\Om_{w}\to \wt\cH_{w}$ is $C_{H^{+}}(\psi_{V})$ invariant and $A_{w}$-equivariant (via the quotient $B_{w}$ on $\wt\cH_{w}$). Therefore $\a\b$ induces a map $\g: \Om_{w}\to \cH_{\psi}(w)$ by passing to the quotients. To summarize, we have a commutative diagram
\begin{equation*}
\xymatrix{\wh\Om_{w} \ar[r]^{\b}\ar[d]^{\pi} & \wt\Om_{w}\ar[r]^{\a} & \wt \cH_{w}\ar[d]^{\pi'}\\
\Om_{w}\ar[rr]^{\g} && \cH_{\psi}(w)}
\end{equation*}
where $\pi,\pi'$ and $\b$ are smooth. To show $\g$ is smooth, it suffices to show that $\a$ is smooth. 

We have a commutative diagram
\begin{equation}\label{HV}
\xymatrix{ H\ar[r]^{\th}\ar[d]^{\wt\a} & V/V_{w}\ar[d]^{p}\\
L_{\bP} \ar[r]^-{\y} & \frg\lr{\t}_{d/m}/V_{w,d/m}}
\end{equation}
Here $\wt\a$ and $p$ are the projections, $\th(h)=\Ad(h^{-1})\psi_{V}\mod V_{w}$; $\y(\ell)=\Ad(\ell^{-1})\psi\mod V_{w,d/m}$. By definition, $\wt\Om_{w}=\th^{-1}(0)$, $\wt \cH_{w}=\y^{-1}(0)$ and $\a$ is the restriction of $\wt\a$. Let $h\in \wt \Om_{w}$ with image $\ov h=\wt\a(h)\in L_{\bP}$. The tangent maps of \eqref{HV} at $h$ are given by
\begin{equation*}
\xymatrix{ \frh=\op_{i=0}^{d-1}\frg\lr{\t}_{-i/m}\ar[d]\ar[rr]^-{T_{h}\th=[-,\Ad(h^{-1})\psi_{V}]} && V/V_{w}\ar[d]\\
\frl=\frg\lr{\t}_{0}\ar[rr]^-{T_{\ov h}\y=[-,\Ad(\ov h^{-1})\psi]} && \frg\lr{\t}_{d/m}/V_{w,d/m}
}
\end{equation*}
We need to show that $T_{h}\th$ is surjective (which implies that $\wt\Om_{w}$ is smooth at $h$) and that the induced map $T_{h}\Om_{w}=\ker(T_{h}\th)\to \ker(T_{\ov h}\y)=T_{\ov h}\wt\cH_{w}$ is surjective. Consider the filtrations $F_{-i}\frh=\op_{i\le i'\le (d-1)}\frg\lr{\t}_{-i'/m}$, $F_{i}(V/V_{w})=\op_{1\le i'\le i}\frg\lr{\t}_{i'/m}/V_{w,i'/m}$. Then $T_{h}\th$ sends $F_{-i}\frh$ to $F_{d-i}(V/V_{w})$, and the induced map on the associated graded is
\begin{equation*}
\Gr_{-i}(T_{h}\th): F_{-i}\frh=\frg\lr{\t}_{-i/m}\xr{[-,\psi']} F_{d-i}(V/V_{w})=\frg\lr{\t}_{(d-i)/m}/V_{w,(d-i)/m}
\end{equation*}
where $\psi'=\Ad(\ov h^{-1})\psi\in \frg\lr{\t}_{d/m}$. In particular, $T_{\ov h}\y=\Gr_{0}(T_{h}\th)$. It remains to show that each $\Gr_{-i}(T_{h}\th)$ is surjective, for $0\le i\le d-1$. 

Passing to the dual spaces and using the Killing form to identify $\frg\lr{\t}_{i/m}$ with the dual of $\frg\lr{\t}_{-i/m}$, we reduce to showing that
\begin{equation*}
(\Gr_{-i}(T_{h}\th))^{*}: \frg\lr{\t}_{(i-d)/m}\cap\Ad(w)(\frn+t\frg[t])\xr{[-,\psi']} \frg\lr{\t}_{i/m}
\end{equation*}
is injective for $0\le i\le d-1$. Let $\frc'\lr{\t}\subset \frg\lr{\t}$ be the centralizer of $\psi'$. This is a Cartan subalgebra of $\frg\lr{\t}$ with the induced grading $\frc'\lr{\t}_{j/m}\subset \frg\lr{\t}_{j/m}$. Then the kernel of $(\Gr_{-i}(T_{h}\th))^{*}$ is $\frc'\lr{\t}_{(i-d)/m}\cap \Ad(w)(\frn+\t^{-1}\frg[\t^{-1}])$. Let $x\in \frc'\lr{\t}_{(i-d)/m}\cap \Ad(w)(\frn+t\frg[t])$ and let $f\in k[\frg]^{G}$ be a homogeneous invariant polynomial of positive degree. Since $i<d$, $x\in \frg\lr{\t}_{<0}$, $f(x)\in \t k\tl{\t}$; since $x\in \Ad(w)(\frn+t\frg[t])$, $f(x)\in tk[t]$. Therefore $f(x)=0$ for all any non-constant homogeneous $f\in k[\frg]^{G}$, hence $x$ is nilpotent. Since $\frc'\lr{\t}$ consists of semisimple elements only, $x$ must be $0$. This shows $(\Gr_{-i}(T_{h}\th))^{*}$ is injective and $\Gr_{-i}(T_{h}\th)$ is surjective, for all $i<d$. This concludes the proof that $\a$ is smooth. By previous discussion, it follows that $\g$ is also smooth.

The commutativity of \eqref{XOm} is clear by checking geometric points (at which level $\wt p$ is a bijection). 
\end{proof}

\sss{Proof of Theoreom \ref{th:coho M Fl}}
We need to verify that the $\Gm(\nu)$-action on $\frX=\cM_{\psi}$ satisfies the conditions in \S\ref{sss:Gm action}. We use the decomposition \eqref{Gm fixed Hess}
 for the fixed point subspace, and the partial order $W_{\bP}\bs\tilW$ defined in \S\ref{sss:att rep M}. We use notations $\frX^{\pm}_{w}$ for attractors and repellers. 
 
Condition  (1) and (3) follow from Lemma \ref{l:attract}.

Condition (2) follows from Lemma \ref{l:repel}.

Condition (4). Since $\le$ on $W_{\bP}\bs \tilW$ is defined to be the closure order of $\bP_{0}$-orbits on $\Fl$, we have $X^{-}_{\le w}=\g(\ov{\Fl_{w}}\cap \Fl_{\psi})$ is proper since $\ov{\Fl_{w}}$ is. This also shows that $\{w'; w'\le w\}$ is finite. On the other hand, $\Bun^{\le w}_{G}(\bP_{\infty}, \bI_{0})=\cup_{w'\le w}\Bun^{ w'}_{G}(\bP_{\infty}, \bI_{0})$	is open in $\Bun_{G}(\bP_{\infty}, \bI_{0})$, therefore its preimage $\frX^{+}_{\le w}$ in $\cM_{\psi}$ is open. This verifies all conditions so Proposition \ref{p:res coho} applies. \qed

		\subsection{A further Hamiltonian reduction}\label{ss:C0}
		We introduce a variant $\cM^{\flat}_{\psi}$ of $\cM_{\psi}$ by making a Hamiltonian reduction with respect to the torus $C_{0}:=\bC_{\infty}/\bC_{\infty}^{+}\cong T^{w,\c}$. We define $\cM^{\flat}_{\psi}$ to be the moduli stack of pairs $(\cE,\ph)$ where
		\begin{itemize}
			\item $\cE$ is a $G$-bundle over $X$ with $\bP_{\infty}(\frac{d}{m})\bC_{\infty}$-level structure at $\infty$ and $\bI_{0}$-level structure at $0$.
			\item $\ph$ is a section of $\Ad(\cE)\ot\om_{X\bs \{0,\infty\}}$ such that under some (equivalently any) trivialization of $\cE|_{D_{\infty}}$ together with its $\bP_{\infty}(\frac{d}{m})\bC_{\infty}$-level structure, we have
			\begin{equation*}
				\ph|_{D_{\infty}^{\times}}\in (\psi+\frc\lr{\t}^{\bot}_{\le0}+\frc\lr{\t}_{<0})d\t/\t.
			\end{equation*} 
			and, under some (equivalently, any) trivialization of $\cE|_{D_{0}}$ together with its $\bI_{0}$-level structure, $\ph|_{D^{\times}_{0}}\in \Lie(\bI^{+}_{0})dt/t$.
		\end{itemize}
		
		There is a Hamiltonian action of $C_{0}=\bC_{\infty}/\bC_{\infty}^{+}$ on $\cM_{\psi}$ since $\bP_{\infty}(\frac{d}{m})\bC_{\infty}$ normalizes $\bK_{\infty}$. The moment map for this action is taking residue at $\infty$
		\begin{equation*}
			\res_{\infty}: \cM_{\psi}\to \frc_{0}=\frc\lr{\t}_{0}=\Lie C_{0}
		\end{equation*}
		defined as follows: for $(\cE,\ph)\in \cM_{\psi}$ such that under some local trivialization $\ph|_{D_{\infty}^{\times}}=(\psi+x_{0}+\frg\lr{\t}_{\le -1/m})d\t/\t$, where $x_{0}\in\frg_{0}$, we let $\res_{\infty}(\cE,\ph)$ be the projection of $x_{0}$ to $\frc_{0}=\frg_{0}/\frc^{\bot}_{0}$. This projection is independent of the trivialization of $\cE|_{D_{\infty}}$. By definition, we have
		\begin{equation}\label{Mpsi flat}
			\cM_{\psi}^{\flat}=\res_{\infty}^{-1}(0)/C_{0}.
		\end{equation}
		This realizes $\cM_{\psi}^{\flat}$ as the Hamiltonian reduction of $\cM_{\psi}$ by $C_{0}$.
		
		On the other hand, we have a map
		\begin{equation*}
			\res_{\cA}: \cA_{\psi}\to \prod_{1\le i\le r, m|(d_{i}-1)}\AA^{1}
		\end{equation*}
		sending $(a_{i})_{1\le i\le r}$ to the coefficient of $\t^{-\frac{d(d_{i}-1)}{m}}$ of $a_{i}$, for those $1\le i\le r$ such that $m|(d_{i}-1)$. There is a commutative diagram
		\begin{equation*}
			\xymatrix{ \cM_{\psi} \ar[d]^{f} \ar[r]^-{\res_{\infty}} & \frc_{0}\ar[d]^{\phi}\\
				\cA_{\psi} \ar[r]^-{\res_{\cA}} & \prod_{1\le i\le r, m|(d_{i}-1)}\AA^{1}}
		\end{equation*}
		Here $\phi$ sends $x_{0}\in \frc_{0}$ to $((\partial_{x_{0}}f_{i})(\psi))_{1\le i\le d, m|(d_{i}-1)}$. It is easy to see that $\phi$ is a linear isomorphism. We define
		\begin{equation*}
			\cA^{\flat}_{\psi}=\res_{\cA}^{-1}(0)\subset \cA_{\psi}.
		\end{equation*}
		In other words, in addition to the conditions defining $\cA_{\psi}$, we require the coefficient of $\t^{-\frac{d(d_{i}-1)}{m}}$ of $a_{i}$ (when $m|(d_{i}-1)$) to be zero. Then $f$ induces a map
		\begin{equation*}
			f^{\flat}:\cM^{\flat}_{\psi}\to\cA_{\psi}^{\flat}.
		\end{equation*}
		One can prove the following properties of $\cM^{\flat}_{\psi}$, analogous to those of $\cM_{\psi}$, using either the same idea of proof or formal deduction from Hamiltonian reduction \eqref{Mpsi flat}. 
		
		\begin{theorem}\label{th:geom M flat} For a homogeneous element $\psi\in \frg\lr{t}$ of slope $\nu=d/m$, the following hold.
			\begin{enumerate}
				\item The stack $\cM^{\flat}_{\psi}$ is a smooth algebraic stack over $k$ of dimension $\frac{d}{m}|\Phi|-r-\dim \frt^{w}$.
				\item $\cM^{\flat}_{\psi}$ carries a canonical symplectic structure of weight $d$ under the $\Gm(\nu)$-action.
				\item The map $f^{\flat}: \cM_{\psi}\to \cA_{\psi}$ is a $\Gm(\nu)$-equivariant completely integrable system.
				\item There is a natural homeomorphism $[\Fl_{\psi}/C_{0}]\to \cM^{\flat}_{a_{\psi}}=f^{\flat, -1}(a_{\psi})$.
				\item When $\psi$ is elliptic (equivalently, $w$ is elliptic), $\cM^{\flat}_{\psi}=\cM_{\psi}, \cA^{\flat}_{\psi}=\cA_{\psi}$, and in particular $f^{\flat}=f$ is proper. 
			\end{enumerate}
		\end{theorem}
		
		\begin{remark}
			One could also consider the variant of $\cM^{\flat}_{\psi}$ by specifying an arbitrary residue in $\frc_{0}$ at $\infty$ and an arbitrary residue in $\frt$ at $0$. They do not have $\Gm(\nu)$-action in general but their Hitchin fibrations are still completely integrable systems.
		\end{remark}

		%%%%%%%%%%%%%%%%%
		%%%%%%%%%%%%%%%%%
		\section{The de Rham moduli space}
		
		Similar to the usual Hitchin moduli space, $\cM_{\psi}$ admits a one-parameter deformation into the moduli space of certain $\l$-connections. We denote the $\l=1$ fiber by $\cM_{\dR,\psi}$. The main result in this section gives a canonical isomorphism between the cohomologies of $\cM_{\dR,\psi}$ and $\cM_{\psi}$.
		
		\subsection{Moduli of $\l$-connections}
		
		\sss{$\cM_{\Hod, \psi}$ and $\cM_{\dR,\psi}$}
		Let $\cM_{\Hod, \psi}$ be the moduli stack of triples $(\l, \cE, \nb)$ where
		\begin{itemize}
			\item $\l\in \AA^{1}$.
			\item $\cE$ is a $G$-bundle over $X$ with $\bK_{\infty}:=\bP_{\infty}(d/m)\bC^{+}_{\infty}$-level structure at $\infty$ and $\bI_{0}$-level structure at $0$. 
			\item $\nb$ is a $\l$-connection on $\cE|_{X\bs \{0,\infty\}}$ satisfying the following conditions: 
			\begin{enumerate}
				\item[(i)] Under any (equivalently, some) trivialization of $\cE|_{D_{\infty}}$ together with its $\bK_{\infty}$-level structure, $\nb|_{D_{\infty}^{\times}}$ takes the form 
				\begin{equation*}
					\nb|_{D_{\infty}^{\times}}\in \l d+(\psi+\frg\lr{\t}_{\le0})d\t/\t.
				\end{equation*}
			
				\item[(ii)] Under any (equivalently, some) trivialization of $\cE|_{D_{0}}$ together with its $\bI_{0}$-level structure,  $\nb|_{D^{\times}_{0}}$ takes the form
				\begin{equation*}
					\nb|_{D^{\times}_{0}}\in \l d+ \Lie(\bI^{+}_{0})dt/t.
				\end{equation*}
			\end{enumerate}
		\end{itemize}
		
		We have the projection $\l: \cM_{\Hod,\psi}\to \AA^{1}$ recording $\l$. The fiber $\l^{-1}(0)$ is identified with $\cM_{\psi}$. Define
		$$\cM_{\dR, \psi}:=\l^{-1}(1).$$
		
		\sss{$\Gm(\nu)$-action}
		The $\Gm(\nu)$-action on $\cM_{\psi}$ extends to an action on $\cM_{\Hod, \psi}$: we interpret the scaling by $s^{d}$ as multiplying $\nb$ by $s^{d}$, so that the function $\l$ has weight $d$ under the $\Gm(\nu)$-action. We denote this action by
		\begin{equation*}
			s: (\l,\cE,\nb)\mapsto s\cdot (\l,\cE,\nb), \quad s\in\Gm(\nu), (\l,\cE,\nb)\in \cM_{\Hod,\psi}.
		\end{equation*}
		
		Since the function $\l$ has weight $d>0$,  the $\Gm(\nu)$-fixed points $\cM_{\Hod,\psi}^{\Gm(\nu)}$ necessarily lie in the central fiber of $\cM_{\psi}$, hence $\cM_{\Hod,\psi}^{\Gm(\nu)}=\cM^{\Gm(\nu)}_{\psi}$, which isomorphic to $\Fl^{\Gm(\nu)}_{\psi}=\coprod \cH_{\psi}(w)$ by Lemma \ref{l:Zw}.
		
		Similar to Lemma \ref{l:attract}, we have a version for $\cM_{\Hod,\psi}$.
		
		\begin{lemma}\label{l:Hod attract} For any $(\l,\cE,\nb)\in \cM_{\Hod,\psi}$, the limit $\lim_{s\to 0}s\cdot (\l,\cE,\nb)$ exists in $\cM^{\Gm(\nu)}_{\psi}$. Moreover,  $\lim_{s\to 0}s\cdot (\l,\cE,\nb)\in Z_{w}$ if and only if the image of $\cE$ in $\Bun_{G}(\bP_{\infty}, \bI_{0})$ is $w$.  
		\end{lemma}
		\begin{proof}
			The proof is almost identical to that of Lemma \ref{l:attract}. We only indicate the modifications. Let $\cN_{\Hod}=\cM_{\Hod,\psi}\times_{\Bun_{G}(\bK_{\infty}, \bI_{0})}\Fl$, where $u: \Fl\to \Bun_{G}(\bK_{\infty}, \bI_{0})$ is the the uniformization map. Then $\cN_{\Hod}$ is the moduli space of triples $(\l, g\bI_{0}, \th)$ where $\l\in\AA^{1}$, $g\bI_{0}\in\Fl$ and $\th\in \frg[t,t^{-1}]_{\le0}$ such that
			\begin{equation*}
				\l g^{-1}dg-\Ad(g^{-1})(\psi+\th)dt/t\in \Lie\bI_{0}^{+}dt/t.
			\end{equation*}
			Indeed, given $(\l,g\bI_{0},\th)$, we define $\cE$ as the image of $g\bI_{0}$ under $u$, equipped with the $\l$-connection $\nb=\l d-(\psi+\th)dt/t$ over $X\bs \{0\}$.
			
			The action of $\Gm(\nu)$ on $\cM_{\Hod,\psi}$ lifts to $\cN$, and is given by 
			\begin{equation*}
				s\cdot (\l, g\bI_{0}, \th)=(s^{d}\l, (\rot(s^{-m})\Ad(\xi(s^{-1}))g)\bI_{0}, s^{d}\rot(s^{-m})\Ad(\xi(s^{-1}))\th).
			\end{equation*}
			The rest of the argument can be copied verbatim from that of Lemma \ref{l:attract}.
		\end{proof}

		\begin{theorem}\label{th:MHod sm} The stack $\cM_{\Hod, \psi}$ is an algebraic space  smooth over $\AA^{1}$ with pure relative dimension equal to $\dim \cM_{\psi}$. 
		\end{theorem}
		\begin{proof} We first remark that $\cM_{\Hod, \psi}$ is an algebraic stack locally of finite type over $\Bun_{G}(\bK_{\infty}, \bI_{0})\times\AA^{1}$. Indeed, letting $\cE^{\univ}$ be the universal $G$-bundle over $\Bun_{G}(\bK_{\infty}, \bI_{0})\times X$, there is an extension of vector bundles over $\Bun_{G}(\bK_{\infty}, \bI_{0})\times X\times \AA^{1}$ of the form 
			\begin{equation}\label{At}
				0\to p_{\Bun\times X}^{*}\Ad(\cE^{\univ}; \Lie \bI^{+}_{0}, \frg\lr{\t}_{\le d/m})\to \cA\to p_{X}^{*}T_{X}(-0-\infty)\to 0
			\end{equation}
			whose splittings over $\l\in \AA^{1}$ classify $\l$-connections on $\cE$ which locally looks like $\l d+\Lie\bI^{+}_{0}dt/t$ near $0$ and looks like $\l d+\frg\lr{\t}_{\le d/m}d\t/\t$ near $\infty$. The moduli stack $\wt\cM$ of splittings of \eqref{At} is an algebraic stack locally of finite type over $\Bun_{G}(\bK_{\infty}, \bI_{0})\times\AA^{1}$. Our $\cM_{\Hod, \psi}$ is a closed substack of $\wt \cM$, hence locally of finite type over $k$.

			Now we show $\l: \cM_{\Hod, \psi}\to \AA^{1}$ is smooth. Let $Z\subset \cM_{\Hod, \psi}$ be the closed substack where $\l$ fails to be smooth. Suppose $Z\ne\vn$. Since $\l$ is $\Gm(\nu)$-equivariant, $Z$ is stable under the $\Gm(\nu)$-action. By Lemma \ref{l:Hod attract}, $Z$ contains a fixed point under $\Gm(\nu)$, hence in particular, $Z\cap \cM_{\psi}\ne\vn$. At a geometric point $(\l,\cE,\nb)\in \cM_{\Hod, \psi}$, the relative tangent complex of $\l: \cM_{\Hod, \psi}\to \AA^{1}$ is the de Rham cohomology $\cohog{*}{X, \cK_{(\cE,\nb)}}$, where $\cK_{(\cE,\nb)}$ is the complex in degrees $-1$ and $0$:
			\begin{equation*}
				\Ad(\cE; \Lie \bI_{0}, \frk_{\infty})\xr{\nb_{\Ad}}\Ad(\cE; \Lie \bI^{+}_{0}, \frg\lr{\t}_{\le0})
			\end{equation*}
			and $\nb_{\Ad}$ is the $\l$-connection on $\Ad(\cE)$ induced from $\nb$. In particular, at a point $(\cE,\ph)\in \cM_{\psi}$, the relative tangent complex of $\l$ is the same as the tangent complex of $\cM_{\psi}$ at $(\cE,\ph)$, whose obstruction group vanishes by Theorem \ref{th:geom M}\eqref{th part:M sm}. Therefore $\l$ is smooth at any point of $\cM_{\psi}$, i.e., $Z\cap \cM_{\psi}=\vn$. Contradiction! This shows $Z=\vn$ hence $\l$ is smooth.

			To see $\cM_{\Hod, \psi}$ has trivial automorphism groups, let $Z'\subset \cM_{\Hod, \psi}$ be the closed substack where the automorphism group is nontrivial. Then $Z$ is $\Gm(\nu)$-stable but $Z\cap  \cM_{\psi}=\vn$ since $\cM_{\psi}$ is known to be an algebraic space. Therefore the same argument as above using Lemma \ref{l:Hod attract} shows that $Z=\vn$.

			Finally, we compute the dimension of $\cM_{\Hod,\psi}$ at any geometric point $(\l,\cE,\nb)$. Since $\cM_{\Hod, \psi}$ is smooth over $\AA^{1}$ with trivial automorphism group at  $(\l,\cE,\nb)$, the complex $R\G(X, \cK_{(\cE,\nb)})$ is concentrated in degree $0$, and $\cohog{0}{X,\cK_{(\cE, \nb)}}$ is the relative tangent space of $\l$ at $(\l,\cE,\nb)$. Therefore the relative dimension of $\l$ at $(\l,\cE,\nb)$ is $$\chi(X, \cK_{(\cE, \nb)})=\chi(X, \Ad(\cE; \Lie \bI^{+}_{0}, \frg\lr{\t}_{\le0})) -\chi(X,  \Ad(\cE; \Lie \bI_{0}, \frk_{\infty})),$$ which is the same as $\dim \cM_{\psi}$.

		\end{proof}
		
		\begin{remark} The smooth map $\l: \cM_{\Hod,\psi}\to \AA^{1}$ carries a canonical symplectic structure constructed as follows. Consider the ``Serre dual'' complex $\cK^{\vee}_{(\cE,\nb)}$ given by \begin{equation*}
				\Ad(\cE; \Lie \bI_{0}, \frg\lr{\t}_{\le -1/m})\xr{\nb_{\Ad}}\Ad(\cE; \Lie \bI^{+}_{0}, \frk^{\vee}_{\infty}).
			\end{equation*}
Here Serre dual is in quotation marks because the differential in $\cK_{(\cE,\nb)}$ is not $\cO_{X}$-linear. Here we take the termwise Serre dual (and the Killing form to identify $\Ad(\cE)|_{\Gm}$ with $\Ad(\cE)^{*}|_{\Gm}$), and the differential in $\cK^{\vee}_{(\cE,\nb)}$ is still given by the adjoint connection.

			The same argument as in \S\ref{tang M} shows that the natural map $\cK_{(\cE, \nb)}\to \cK_{(\cE,\nb)}^{\vee}$ is a quasi-isomorphism in the derived category of sheaves of abelian groups on $X$, hence a canonical isomorphism $\cohog{*}{X, \cK_{(\cE,\nb)}}\isom \cohog{*}{X, \cK^{\vee}_{(\cE,\nb)}}$. We claim that there is a perfect pairing between $\cohog{*}{X, \cK_{(\cE,\nb)}}$ and $\cohog{*}{X, \cK^{\vee}_{(\cE,\nb)}}$ even though $\cK_{(\cE,\nb)}$ is not an $\cO_{X}$-linear complex. Indeed, by construction there is a $k$-linear pairing of complexes of sheaves
			\begin{equation}\label{K pairing}
				\cK_{(\cE,\nb)}\ot_{k}\cK^{\vee}_{(\cE,\nb)}\to \om_{X}[1].
			\end{equation}
			Taking cohomology induces a pairing between $\cohog{i}{X, \cK_{(\cE,\nb)}}$ and $\cohog{-i}{X, \cK^{\vee}_{(\cE,\nb)}}$ valued in $\cohog{0}{X,\om_{X}[1]}\cong k$. Writing $\cK=\cK_{(\cE,\nb)}$ as $[\cK^{-1}\to \cK^{0}]$,  $\cohog{*}{X, \cK_{(\cE,\nb)}}$ fits into a long exaxt sequence exact at the two ends
			\begin{equation}\label{long K}
				\cohog{-1}{\cK}\to\cohog{0}{\cK^{-1}}\to \cohog{0}{\cK^{0}}\to  \cohog{0}{\cK}\to \cohog{1}{\cK^{-1}}\to \cohog{1}{\cK^{0}}\to \cohog{1}{\cK}  
			\end{equation}
			Similarly for $\cohog{*}{X, \cK^{\vee}_{(\cE,\nb)}}$, or its dual
			\begin{equation}\label{long K dual}
				\cohog{1}{\cK^{\vee}}^{*}\to \cohog{1}{\cK^{-1,\vee}}^{*}\to \cohog{1}{\cK^{0,\vee}}^{*}\to  \cohog{0}{\cK^{\vee}}^{*}\to \cohog{0}{\cK^{-1,\vee}}^{*}\to \cohog{0}{\cK^{0,\vee}}^{*}\to\cohog{-1}{\cK^{\vee}}^{*} 
			\end{equation}
			The pairing gives a map of long exact sequences from \eqref{long K} to \eqref{long K dual}. One checks that the maps $\cohog{i}{\cK^{j}}\to \cohog{1-i}{\cK^{j,\vee}}^{*}$ are the usual Serre duality ($i=0,1$, $j=-1,0$), hence are isomorphisms. Therefore the maps $\cohog{i}{\cK}\to\cohog{-i}{\cK^{\vee}}^{*}$ given by the pairing \eqref{K pairing} is also an isomorphism.
			It is easy to check that this is gives a symplectic form on $\cohog{0}{X, \cK_{(\cE,\nb)}}$, and a perfect pairing between $\cohog{-1}{X,\cK_{(\cE,\nb)}}$ and $\cohog{1}{X,\cK_{(\cE,\nb)}}$. It can be shown by calculations similar to that done by R.Fedorov \cite[\S6.3]{Fed} that this 2-form is closed. Therefore the map $\l$ has a relative symplectic structure. In particular, $\cM_{\dR,\psi}$ is a symplectic algebraic space.
		\end{remark}

		\subsection{Comparison of cohomology}\label{ss:comp coho}
		The non-abelian Hodge theory suggests that $\cM_{\dR,\psi}$ should be diffeomorphic to $\cM_{\psi}$. In particular they should have isomorphic cohomology. In this subsection we prove the cohomological isomorphism without showing they are diffeomorphic.
		
		\sss{The situation}\label{sss:A1} Consider the following general situation. Let $f: \frX\to \AA^{1}$ be a regular function on an algebraic space $\frX$ locally of finite type over an algebraically closed field $k$. Let $\frX_{\l}=f^{-1}(\l)$ for $\l\in k$. Let $\Gm$ act on $\frX$ such that $f$ has weight $d>0$. Let $\frX^{\Gm}=\coprod_{\a\in I}Z_{\a}$ be an open-closed  decomposition. We use notation $\frX^{+}_{\a}$ and $\frq^{+}_{\a}: \frX^{+}_{\a}\to Z_{\a}$ from \S\ref{sss:Gm action} for attractors.
		
		We make the following assumptions:
		\begin{enumerate}
			\item The function $f$ is a smooth morphism $f:\frX\to \AA^{1}$. In particular, $\frX$ is smooth over $k$.
			\item The map $\frq^{+}_{\a}: \frX^{+}_{\a}\to \frX$ is a locally closed embedding. 
			\item $\cup_{\a\in I}\frX^{+}_{\a}=\frX$, i.e., for any $x\in \frX$, the limit $\lim_{t\to 0}t\cdot x$ exists.

			\item There exists a partial order $\le$ on $I$ such that for each $\a\in I$
			\begin{itemize}
			\item The set $\{\a'\in I;\a'\in \a\}$	is finite.	
			\item $\frX^{+}_{\le \a}:=\cup_{\a'\le \a}\frX^{+}_{\a}$ is open in $\frX$, and is of finite type over $k$.
			\end{itemize}		
		\end{enumerate}

		\begin{lemma}\label{l:contracting coho} Under the above assumptions, the restriction map $\cohog{*}{\frX}\to \cohog{*}{\frX_{1}}$ is an isomorphism.
		\end{lemma}
		\begin{proof}
			Let $i_{1}: \frX_{1}\incl \frX$ be the inclusion. The canonical map $ i_{1*}\Qlbar[-2](-1)\cong i_{1*}i_{1}^{!}\Qlbar\to \Qlbar$ induces a map
			\begin{equation*}
r: \cohoc{*}{\frX_{1}, i_{1}^{!}\Qlbar}\cong \cohoc{*-2}{\frX_{1},\Qlbar(-1)}\to \cohoc{*}{\frX,\Qlbar}.
\end{equation*}
It suffices to prove $r$ is an isomorphism, and the statement on cohomology follows by Poincar\'e duality.

			Extend $\le$ to a total order on $I$. Let $\frX^{+}_{1,\le\a}=\frX^{+}_{\le\a}\cap \frX_{1}$, and similarly define $\frX^{+}_{1,\a}$ and $\frX^{+}_{1,<\a}$. We have an analogue $r_{\le\a}$ of $r$ for the inclusion $\frX_{1,\le\a}^{+}\incl \frX^{+}_{\le\a}$, and similarly we have $r_{\a}$ and $r_{<\a}$. We have a map of long exact sequences
\begin{equation*}
\xymatrix{\cdots\ar[r] & \cohoc{*-2}{\frX_{1,<\a}^{+}}(-1) \ar[r]\ar[d]^{r_{<\a}}& \cohoc{*-2}{\frX_{1,\le\a}}(-1) \ar[r]\ar[d]^{r_{\le\a}}& \cohoc{*-2}{\frX_{1,\a}^{+}}(-1)\ar[d]^{r_{\a}}\ar[r] & \cdots\\
\cdots\ar[r] &\cohoc{*}{\frX_{<\a}^{+}}\ar[r] & \cohoc{*}{\frX_{\le\a}}\ar[r] & \cohoc{*}{\frX_{\a}^{+}}\ar[r] & \cdots}
\end{equation*}

			By induction on $\a$, it suffices to show that $r_{\a}$ is an isomorphism for each $\a$.

			Let $f_{\a}=f|_{\frX_{\a}^{+}}: \frX^{+}_{\a}\to \AA^{1}$, and $\pi_{\a}=(f_{\a}, \frq^{+}_{\a}): \frX^{+}_{\a}\to \AA^{1}\times Z_{\a}$. We first claim that $\pi_{\a}$ is smooth. Since the critical locus of $\pi_{\a}$ is closed and stable under the $\Gm$-action, it must intersect $Z_{\a}$ if not empty. Therefore it suffices to show that $\pi_{\a}$ is smooth along $Z_{\a}$. Let $z\in Z_{\a}$ be a geometric point. By \cite[Prop.1.4.11(vi)]{Dr}, $(T_{z}\frX^{+}_{\a})=(T_{z}\frX)^{+}$ is the summand where $\Gm$ acts with positive weights; $T_{z}Z_{\a}=(T_{z}\frX)^{0}$ is the zero weight space. The tangent map of $\pi_{\a}$ at $z\in Z_{\a}$ is $(T_{z}f, \pr_{0})$, where $\pr_{0}$ is the natural projection.  By assumption, $f$ has weight $d>0$, therefore $T_{z}f: T_{z}\frX\to T_{0}\AA^{1}=\AA^{1}$ factors through the weight $d$ summand $T_{z}\frX\surj (T_{z}\frX)^{+}\surj (T_{z}\frX)^{d}$. Since $f$ is a smooth morphism, $df(z)$ is nonzero. Therefore $T_{z}\pi_{\a}=(T_{z}f, \pr_{0})$ is surjective, and $\pi_{\a}$ is smooth.
			
			We then show that the geometric fibers of $\pi_{\a}$ are isomorphic to affine spaces.   By \cite[Theorem 1.4.2]{Dr}, $\frq^{+}_{\a }:\frX^{+}_{\a}\to Z_{\a}$ is an affine morphism of finite type. Let $z\in Z_{\a}$ be a geometric point valued in $K$, then $\frX^{+}_{z}=\frq^{+,-1}_{\a}(z)$ is a smooth affine scheme over $K$ with a contracting $\Gm$-action. We have $\frX^{+}_{z}=\Spec A$ where $A=\op_{n\ge0}A_{n}$ is a finitely generated graded $K$-algebra with $A_{0}=K$. Since $\Spec A$ is smooth at the cone point $z$, one can choose liftings $t_{1},\cdots, t_{m}\in A_{+}=\op_{n>0}A_{n}$ of a homogeneous basis of the cotangent space $A_{+}/A_{+}^{2}$, so that $K[t_{1},\cdots, t_{m}]\isom A$ as graded algebras. Now $df_{\a}(z)$ is a nonzero homogeneous element in $T^{*}_{z}(\Spec A)$ by the smoothness of $\pi_{\a}$, we may assume $t_{1}=f_{\a}$ so that it lifts $df_{\a}(z)$. This way, we have an isomorphism of the pair $(\frX^{+}_{z}, f_{\a})$ with $(\Spec K[t_{1},\cdots, t_{m}], t_{1})$, so the geometric fibers of $f_{\a}|\frX^{+}_{z}$ are affine spaces of dimension $m-1$.
			
			If needed we may decompose $Z_{\a}$ further to ensure $\pi_{\a}$ is equidimensional. Since $\pi_{\a}$ is smooth with geometric fibers isomorphic to $\AA^{d_{\a}}$, we have a canonical isomorphism
			\begin{equation*}
				K_{\a}:=R\pi_{\a!}\Qlbar\cong \Qlbar[-2d_{\a}](-d_{\a}) \in D^{b}_{c}(\AA^{1}\times Z_{\a}).
			\end{equation*}
			Let $\io_{1}: \{1\}\times Z_{\a}\to \AA^{1}\times Z_{\a}$ be the inclusion. The map $r_{\a}$ is induced from the canonical map $\io_{1*}\io_{1}^{!}K_{\a}\to K_{\a}$ by taking $\cohoc{*}{-}$. Now $K_{\a}$ is constant, it is clear that $\io_{1*}\io_{1}^{!}K_{\a}\to K_{\a}$  induces an isomorphism on compactly supported cohomology on $\AA^{1}\times Z_{\a}$. Therefore  $r_{\a}$ is an isomorphism.
		\end{proof}

		\begin{cor}\label{c:comp dR}
		Both restriction maps
		\begin{equation*}
\xymatrix{\cohog{*}{\cM_{\psi}} & \cohog{*}{\cM_{\Hod,\psi}}\ar[r]^{i_{1}^{*}}\ar[l]_{i_{0}^{*}} & \cohog{*}{\cM_{\dR,\psi}}}
\end{equation*}
are  isomorphisms.
		\end{cor}
		\begin{proof} To show $i_{1}^{*}$ is an isomorphism, we would like to apply Lemma \ref{l:contracting coho}. We need to check that $\l: \cM_{\Hod,\psi}\to \AA^{1}$ satisfies the conditions in \S\ref{sss:A1}. Condition (1) follows from Theorem \ref{th:MHod sm}; (2) can be proved in the same way as Lemma \ref{l:attract}; (3) follows from Lemma \ref{l:Hod attract}; for (4), the partial order on $W_{\bP}\bs\tilW$ is the same one used \S\ref{sss:att rep M}.
		
		Now we show that $i_{0}^{*}$ is an isomorphism. Consider the further restriction map along $\g_{\Hod}: \Fl_{\psi}\to \cM_{\Hod,\psi}$. We have a factorization
		\begin{equation*}
\g^{*}_{\Hod}: \cohog{*}{\cM_{\Hod,\psi}}\xr{i_{0}^{*}}\cohog{*}{\cM_{\psi}}\xr{\g^{*}}\cohog{*}{\Fl_{\psi}}.
\end{equation*}
By Theorem \ref{th:coho M Fl}, $\g^{*}$ is an isomorphism. Observe that Prop. \ref{p:res coho} also applies to $\cM_{\Hod,\psi}$, which proves that $\g_{\Hod}^{*}$ is an isomorphism. Therefore $i_{0}^{*}$ is also an isomorphism.				
		\end{proof}

		\subsection{Variants}

		\sss{Changing the level group}
		We have a one-parameter deformation $\cM^{\da}_{\Hod,\psi}$ of the Poisson moduli space $\cM^{\da}_{\psi}$ introduced in \S\ref{ss:cons3}: it classifies $(\l, \cE, \nb)$ where
		\begin{itemize}
			\item $\l\in \AA^{1}$.
			\item $\cE$ is a $G$-bundle over $X$ with $\bP^{+}_{\infty}$-level structure at $\infty$ and $\bI_{0}$-level structure at $0$. 
			\item $\nb$ is a $\l$-connection on $\cE|_{X\bs \{0,\infty\}}$ satisfying the following conditions: 
			\begin{enumerate}
				\item[(i)] Under some (equivalently, any) trivialization of $\cE|_{D_{\infty}}$ together with its $\bP^{+}_{\infty}$-level structure, we require
				\begin{equation}\label{conn at infty}
					\nb|_{D_{\infty}^{\times}}\in \l d+(\psi+\frg\lr{\t}_{\le (d-1)/m})d\t/\t.
				\end{equation}
				\item[(ii)] Under some (equivalently, any) trivialization of $\cE|_{D_{0}}$ together with its $\bI_{0}$-level structure, we require
				\begin{equation*}
					\nb|_{D^{\times}_{0}}\in \l d+\Lie(\bI^{+}_{0})dt/t.
				\end{equation*}
			\end{enumerate}
		\end{itemize}
		
		A small part of the Hitchin map $f^{\da}$ for $\cM^{\da}_{\psi}$ continues to make sense for $\cM^{\da}_{\Hod,\psi}$. Recall in the proof of Proposition \ref{p:Cart Mda} we have introduced an affine space $\fra_{\psi}$ with a surjection $\cA^{\da}_{\psi}\to \fra_{\psi}$ that records a Laurent tail of $a_{i}$ at $\infty$.  Let $\ov a_{\psi}\in \fra_{\psi}$ be the image of $a_{\psi}$. We also introduced in the proof of Proposition \ref{p:Cart Mda} an affine space $V_{\psi} = (\psi+\frg\lr{\t}_{\le(d-1)/m})/\frg\lr{\t}_{\le0}$ with the action of $Q=\bP^{+}_{\infty}/\bP_{\infty}(\frac{m}{d})$. The invariant polynomials 
$(f_{1},\cdots, f_{r})$ give a map $[V_{\psi}/Q]\to \fra_{\psi}$. Taking the irregular part of the connection $\nb$ at $\infty$ yields a map
		\begin{equation*}
 \ov f^{\da}_{\psi}: \cM^{\da}_{\Hod,\psi}\to [V_{\psi}/Q]\to \fra_{\psi}.
		\end{equation*}
		
		We have an analogue of Proposition \ref{p:Cart Mda} with the same proof. 		
		\begin{prop}\label{p:Cart MHodda} The natural map $\cM_{\Hod,\psi}\to \cM^{\da}_{\Hod,\psi}$ identifies $\cM_{\Hod, \psi}$ with the fiber of $ \ov f^{\da, -1}_{\psi}(\ov a_{\psi})$. Equivalently,  $\cM_{\Hod,\psi}$ can be identified with the closed subspace of $\cM^{\da}_{\Hod,\psi}$ obtained by replacing the condition \eqref{conn at infty} with:  under {\em some} trivialization of $\cE|_{D_{\infty}}$,  $ \nb|_{D_{\infty}^{\times}}\in \l d+(\psi+\frg\lr{\t}_{\le 0})d\t/\t$.
		\end{prop}

		\sss{Hamiltonian reduction by $C_{0}$}
		As in \S\ref{ss:C0}, $C_{0}$ acts on $\cM_{\Hod,\psi}$, and we have the map of taking formal residue at $\infty$ 

		\begin{equation*}
			\res_{\Hod, \infty}: \cM_{\Hod,\psi}\to \frc_{0}.
		\end{equation*}
		For $(\cE,\nb)\in \cM_{\Hod,\psi}$, such that $\nb|_{D^{\times}_{\infty}}=\l d+(\psi+x_{0}+\frg\lr{\t}_{\le -1/m})d\t/\t$  under some local trivialization, where $x_{0}\in\frg\lr{\t}_{0}=\frg_{0}$, $\res_{\Hod, \infty}(\cE,\nb)$ is the projection of $x_{0}$ to $\frc_{0}=\frg_{0}/\frc_{0}^{\bot}$. 
		
		We define
		\begin{equation*}
			\cM^{\flat}_{\Hod,\psi}=\res^{-1}_{\Hod, \infty}(0)/C_{0}, \quad \cM^{\flat}_{\dR,\psi}=(\res^{-1}_{\Hod, \infty}(0)\cap \cM_{\dR,\psi})/C_{0}.
		\end{equation*}
		Then $\cM^{\flat}_{\dR,\psi}$ is a smooth algebraic stack of the same dimension as $\cM_{\psi}^{\flat}$. The analog of Corollary \ref{c:comp dR} holds, giving a canonical isomorphism $\cohog{*}{\cM^{\flat}_{\dR, \psi}}\cong\cohog{*}{\cM^{\flat}_{\Hod, \psi}}\cong\cohog{*}{\cM^{\flat}_{\psi}}$.
		
		\sss{Varying semisimple monodromy at $0$}\label{sss:var s}
		The space $\cM_{\Hod,\psi}$ admits a deformation over the universal Cartan $\frh$ as follows. Consider the moduli stack ${}_{\frh}\cM_{\Hod, \psi}$ of triples $(\l,\cE,\nb)$ as in the definition of $\cM_{\Hod, \psi}$, except that we relax the condition near $0$ to be: under any (equivalently, some) trivialization of $\cE|_{D_{0}}$ together with its $\bI_{0}$-level structure,  $\nb|_{D^{\times}_{0}}$ takes the form
		\begin{equation*}
			\nb|_{D^{\times}_{0}}\in \l d+ \Lie(\bI_{0})dt/t.
		\end{equation*}
		We have a map
		\begin{equation*}
			\r: {}_{\frh}\cM_{\Hod, \psi}\to \AA^{1}\times\frh
		\end{equation*}
		where the $\frh$-factor sends $(\l,\cE,\nb)$ to the image of $\Res_{0}\nb$ in the universal Cartan $\frh=\Lie \bI_{0}/\Lie \bI_{0}^{+}$.
		
		We define
		\begin{eqnarray*}
			{}_{\frh}\cM_{\dR, \psi}=\r^{-1}(\{1\}\times\frh)\subset {}_{\frh}\cM_{\Hod, \psi},\\
			{}_{\frh}\cM_{\psi}=\r^{-1}(\{0\}\times\frh)\subset {}_{\frh}\cM_{\Hod, \psi}.
		\end{eqnarray*}
		
		The map $\r$ is equivariant with respect to the $\Gm(\nu)$-action on ${}_{\frh}\cM_{\Hod, \psi}$ and the scaling action on $\AA^{1}\times\frh$ by the $d$-th power. If we fix $s\in \frh$, then we get a $\Gm(\nu)$-equivariant one-parameter family by restricting ${}_{\frh}\cM_{\Hod, \psi}$  along the line of $\AA^{1}\times \frh$ through $(1,s)$:
		\begin{equation*}
			\l_{s}: {}_{s}\cM_{\Hod,\psi}\to \AA^{1}_{s}:=\{(\l,\l s)|\l\in \AA^{1}\}\subset \AA^{1}\times \frh
		\end{equation*}
		whose fiber over $\l=1$ we denote by ${}_{s}\cM_{\dR,\psi}$. Note its fiber over $\l=0$ is $\cM_{\psi}$.

		The formal residue construction extends to ${}_{\frh}\cM_{\Hod,\psi}$. In particular it restricts to a formal residue map on the de Rham space
		\begin{equation*}
			\res_{\dR, \infty} : {}_{\frh}\cM_{\dR, \psi} \to \frc_{0}.
		\end{equation*}
		For $\th\in \frc_{0}$, let
		\begin{equation*}
			{}_{\frh}\cM_{\dR, \psi, \th} = \res^{-1}_{\dR, \infty}(\th).
		\end{equation*}

		\begin{theorem}\label{th:comp coho s} Both restriction maps 
		\begin{equation*}
\xymatrix{\cohog{*}{\cM_{\psi}} & \cohog{*}{{}_{s}\cM_{\Hod,\psi}}\ar[r]^{i_{1}^{*}}\ar[l]_{i_{0}^{*}} & \cohog{*}{{}_{s}\cM_{\dR,\psi}}}
\end{equation*}
are isomorphisms. In particular, there is a canonical isomorphism
\begin{equation*}
\cohog{*}{{}_{s}\cM_{\dR,\psi}}\cong \cohog{*}{\cM_{\psi}}
\end{equation*}
for any $s\in \frh$.
		
		\end{theorem}
		The proof is along the same lines as Corollary \ref{c:comp dR}, using the general results about $\Gm$-contracting families in \S\ref{ss:comp M Fl} and \S\ref{ss:comp coho}.	
		\begin{remark} For $s\in \frh$, let ${}_{s}\cM_{\psi}=\r^{-1}(0,s)\subset{}_{\frh}\cM_{\psi}$. This is the Higgs moduli space analogous to $\cM_{\psi}$ but with residue at $0$ mapping to $s\in\frh$ under $\frb\to \frh$. By scaling $s$, there is an $\AA^{1}$-family connecting with general fiber ${}_{s}\cM_{\psi}$ and $0$-fiber $\cM_{\psi}$. The same argument as  Corollary \ref{c:comp dR} gives a canonical isomorphism $\cohog{*}{{}_{s}\cM_{\psi}}\cong \cohog{*}{\cM_{\psi}}$.
		\end{remark}

		\subsection{Symmetry of $\cM_{\Hod,\psi}$ coming from $\infty$}
		We shall construct an action of $C\lr{\t}/\bC_{\infty}^{+}$ on $\cM_{\Hod,\psi}$. 
		
			\sss{} First we describe a filtration on the loop torus $C\lr{\t}$.
		
		Recall from Remark \ref{r:w} we defined a maximal torus $T\subset G$ to be the fiber of $C$ at $t=1$. Also $T$ is the centralizer of $\ov\psi\in \frc_{d/m}$. Fix an $m$th root of $\t$ and denote it by $\t^{1/m}$. Now $\psi$ and $\ov\psi \t^{-d/m}$ are in the same $G_{\ad}\lr{\t^{1/m}}$-orbit, we get a canonical isomorphism between their centralizers (which are abelian) inside $G\lr{\t^{1/m}}$: 
		\begin{equation*}
 \can: C\lr{\t^{1/m}}\cong T\lr{\t^{1/m}}.
\end{equation*}
This allows us to identify $C\lr{\t}$ 
with the fixed points under the diagonal $\mu_{m}$-action on $T\lr{\t^{1/m}}$
\begin{equation}\label{C fxdpt}
C\lr{\t}\cong (T\lr{\t^{1/m}})^{\mu_{m}}
\end{equation}
where  $\z\in\mu_{m}$ acts on $\t^{1/m}$ via the Galois action $\t^{1/m}\mapsto \z\t^{1/m}$, and $\mu_{m}$ acts on $T$ via an injective homomorphism
\begin{equation*}
\om: \mu_{m}\to W=W(G,T) 
\end{equation*}
that sends a primitive element $\z\in\mu_{m}$ to a regular element of order $m$.

Later when we  work with $k=\CC$, we shall let $w$ be the image of $\z_{m}:=\exp(2\pi i/m)$ under $\om$. In general, we use $\j{w}$ to denote the image of $\om$, keeping in mind that $w$ is a regular element of $m$ up to taking prime to $m$ powers. 
		
				Recall $\bC_{\infty}\subset C\lr{\t}$ is the parahoric subgroup, with pro-unipotent radical $\bC_{\infty}^{+}$. Then the isomorphism \eqref{C fxdpt} gives a canonical isomorphism $\bC_{\infty}/\bC_{\infty}^{+}\cong T^{\j{w},\c}$, the neutral component of the $\j{w}$-fixed points on $T$. We have a Kottwitz isomorphism between $\pi_{0}(C\lr{\t})=(C\lr{\t}/\bC_{\infty})^{\red}$ and the $\j{w}$-coinvariants on $\xcoch(T)$
		\begin{equation}\label{comp C}
			\k_{C}: (C\lr{\t}/\bC_{\infty})^{\red}\cong \xcoch(T)_{\j{w}}.
		\end{equation}
See \cite[\S2.a.2, Theorem 5.1 step A]{PR}. This isomorphism makes the following diagram commutative
\begin{equation*}
\xymatrix{ (T\lr{\t^{1/m}}/\bT_{\infty})^{\red} \ar[d]^{\Nm}\ar[r]^-{\k_{T}} & \xcoch(T) \ar[d]^{p}\\
(C\lr{\t}/\bC_{\infty})^{\red}\ar[r]^-{\k_{C}} & \xcoch(T)_{\j{w}}
}
\end{equation*}
Here $\bT_{\infty}=T\tl{\t^{1/m}}$, $\k_{T}$ is given by the $\t^{1/m}$-adic valuation,  $\Nm$ is the norm map $x\mapsto x\z(x)\cdots \z^{m-1}(x)$ (for $\z\in \mu_{m}$ acting on $T\lr{\t^{1/m}}$ diagonally), and $p$ is the canonical quotient map.

		Let $\bC^{\na}_{\infty}\subset C\lr{\t}$ be the maximal bounded subgroup that corresponds to $(T\tl{\t^{1/m}}^{\j{w}}$ under \eqref{C fxdpt}. Then under the isomorphism \eqref{comp C}, $\bC_{\infty}^{\na}/\bC_{\infty}$ corresponds to the torsion subgroup of $\xcoch(T)_{\j{w}}$. On the other hand, $\bC^{\na}_{\infty}/\bC_{\infty}^{+}\cong T^{\j{w}}$, and $\bC_{\infty}^{\na}/\bC_{\infty}$ can also be identified with $\pi_{0}(T^{\j{w}})$. To summarize, we have a filtration of $C\lr{\t}$ with reduced associated graded as follows:
		\begin{equation}\label{fil C}
			\bC_{\infty}^{+}\underbrace{\subset}_{T^{\j{w},\c}}\bC_{\infty}\underbrace{\subset}_{\xcoch(T)_{\j{w},\tors}}\bC^{\na}_{\infty}\underbrace{\subset}_{\xcoch(T)_{\j{w}}/\tors} C\lr{\t}
		\end{equation}
	
		\sss{Residue map} We construct a residue map
		\begin{equation}\label{resC}
\res_{C,\infty}: C\lr{\t}/\bC^{\na}_{\infty}\cong \xcoch(T)_{\j{w}}/\tors\to \frc_{0}
\end{equation}
as follows. First, for the split torus $T\lr{\t^{1/m}}$ we have the usual residue map
\begin{equation*}
\res_{T, \infty}: T\lr{\t^{1/m}}/\bT_{\infty}=\frac{1}{m}\xcoch(T)\incl \frt.
\end{equation*}
defined by $x\mapsto \res_{\t=0}(x^{-1}dx)$ (taking the coefficient of $d\t/\t$). Then $\res_{C,\infty}$ is obtained from $\res_{T,\infty}$ by restricting to $\mu_{m}$-fixed points (noting that $\frc_{0}=\frt^{\j{w}}$).

We use the norm map to identify $\frt_{\j{w}}\isom \frt^{\j{w}}$ ($x\mapsto x+wx+\cdots+w^{m-1}x$). Then the residue map fits into a commutative diagram
\begin{equation*}
\xymatrix{ m\res_{T,\infty}: & T\lr{\t^{1/m}}/\bT_{\infty} \ar[r]^{\k_{T}}\ar[d]^{\Nm} & \xcoch(T) \ar[d]^{p}\ar[r] & \frt\ar[d]^{p}\\
 \res_{C,\infty}: & C\lr{\t}/\bC^{\na} \ar[r]^{\k_{C}} & \xcoch(T)_{\j{w}}/\tors \ar[r] & \frc_{0}\cong\frt_{\j{w}}
}
\end{equation*}
Here the maps indexed by $p$ are natural projections. 

When $k=\CC$, we may take the cokernel of the horizontal maps as complex tori, and get a canonical isomorphism
\begin{equation}\label{exp Tw}
\frc_{0}/\Im(\res_{C,\infty})\isom T_{\j{w}} \quad (\mbox{$\j{w}$-covariants on $T$})
\end{equation}

		\sss{} Let $\wh\cM_{\Hod,\psi}$ be the moduli space of $(\l, \cE, \a_{\infty}, \nb)$ where $(\l, \cE,\nb)$ is as in the definition of $\cM_{\Hod,\psi}$, and $\a_{\infty}$ is a trivialization of $\cE|_{D_{\infty}}$ (together with its $\bK_{\infty}$-level structure) under which $\nb|_{D^{\times}_{\infty}}$ takes the form $\l d+(\psi+\frg\lr{\t}_{\le0})d\t/\t$. Note that $\cM_{\Hod,\psi}=\wh\cM_{\Hod,\psi}/\bK_{\infty}$ where $\bK_{\infty}$ acts by changing the trivialization $\a_{\infty}$.
		
		Sending $(\l, \cE, \a_{\infty}, \nb)$ to the connection one-form of $\nb|_{D^{\times}_{\infty}}$ under the trivialization $\a_{\infty}$ gives a map
		\begin{equation*}
			\wh\cM_{\Hod,\psi}\to \psi+\frg\lr{\t}_{\le0}.
		\end{equation*}
		For each $\ph\in \psi+\frg\lr{\t}_{\le0}$, we have the centralizer group scheme $C_{\ph}$ over $D_{\infty}^{\times}$, and its maximal bounded subgroup $\bC_{\ph}^{\na}$, parahoric subgroup $\bC_{\ph}$ and pro-unipotent radical $\bC_{\ph}^{+}$. As $\ph$ varies, these groups form families over $\ph\in \psi+\frg\lr{\t}_{\le0}$. For example, the $C_{\ph}$ form a torus $J$ over $D_{R}^{\times}=\Spec R\lr{\t}$, where $R=\G(\psi+\frg\lr{\t}_{\le0}, \cO)$; we have integral models $\bJ^{\na}, \bJ$ and $\bJ^{+}$ of $J$ over $D_{R}=\Spec R\tl{\t}$ whose fibers over $\ph\in \psi+\frg\lr{\t}_{\le0}$ are $\bC^{\na}_{\ph}, \bC_{\ph}$ and $\bC_{\ph}^{+}$ respectively.  
		
		Let $J\lr{\t}$ be the loop group of $J$, which is a group ind-scheme over the infinite-dimensional affine space $\psi+\frg\lr{\t}_{\le0}=\Spec R$.  This is a subgroup of $G\lr{\t}\times(\psi+\frg\lr{\t}_{\le0})$. We have an action of $J\lr{\t}$ on $\wh\cM_{\Hod,\psi}$ over $\psi+\frg\lr{\t}_{\le0}$ by changing on the trivialization $\a_{\infty}$. Here we are using that, for $\ph\in \psi+\frg\lr{\t}_{\le0}$ and $g\in C_{\ph}\lr{\t}$, we have $g^{-1}dg\in \frg\lr{\t}_{\le0}d\t/\t$. 
		
		\begin{lemma}
			The group scheme $\bJ^{\na}$ admits a canonical trivialization: $\bJ^{\na}\cong \bC^{\na}_{\infty}\htimes_{k}\Spec R$ over $D_{R}$ (here we abuse the notation to view $\bC_{\infty}^{\na}$ as  a group scheme over $D_{\infty}$) whose restriction to $D_{\infty}$ (corresponding to $\psi$) is the identity. 
		\end{lemma}
		\begin{proof}
			We base change to the cyclic cover $D^{(m)}_{R}:=\Spec R\tl{\t^{1/m}}\to D_{R}=\Spec R\tl{\t}$. We choose  $g\in G\lr{\t^{1/m}}$ such that the adjoint action by $g$ sends $\psi+\frg\lr{\t}_{\le 0}=\Spec R$ into a closed subscheme of $\Spec R'=\psi_{0}\t^{-d/m}+\frg\tl{\t^{1/m}}$ for some regular semisimple $\psi_{0}\in\frt^{\rs}$. Consider the centralizer group scheme $J'$ over $D^{(m)}_{R'}$, and its integral model $\bJ'=\bJ'^{\na}$ (both maximally bounded and parahoric since $J'$ is split when specialized to each point of $R'$). Moreover, $\bJ'|_{D^{(m)}_{R}}$ carries a $\mu_{m}$-equivariant structure since it is the parahoric subgroup of $J'|_{D^{(m)}_{R}}=J\times_{D^{\times}_{R}}D^{(m),\times}_{R}$. Then we have
			\begin{equation}\label{des J}
				\bJ^{\na}=(\Res_{D^{(m)}_{R}/D_{R}}(\bJ'|_{D^{(m)}_{R}}))^{\mu_{m}}.
			\end{equation}

			Let $\bC'$ be the restriction of $\bJ'$ to $D^{(m)}_{\infty}$ (corresponding to $\psi_{0}\t^{-d/m}\in \Spec R'$).  We first give a canonical isomorphism of tori over $D^{(m)}_{R'}$
			\begin{equation*}
				\g'_{m}: \bJ'\cong \bC'\htimes_{k}\Spec R'.
			\end{equation*}
			By the rigidity of homomorphisms between tori,  it suffices to give a trivialization of the restriction  $J'|_{\Spec R'}$ (where $\Spec R'\incl D_{R'}^{(m)}$ is defined by $\t^{1/m}=0$). However, $J'$ is the group scheme of centralizers of $\psi_{0}+\t^{d/m}\frg\tl{\t^{1/m}}$, hence its reduction modulo $\t^{1/m}$ is canonically trivialized.
			
			Restricting both sides of $\g'$ to $D^{(m)}_{R}$ we get an isomorphism of tori
			\begin{equation*}
				\g_{m}: \bJ'|_{D^{(m)}_{R}}\cong \bC'\htimes_{k}\Spec R
			\end{equation*}
			whose restriction to $D^{(m)}_{\infty}$ is the identity. It can be checked that the isomorphism $\g_{m}$ is independent of how one conjugates $\psi+\frg\lr{\t}_{\le0}$ into $\psi_{0}\t^{-d/m}+\frg\tl{\t^{1/m}}$ inside $G\lr{\t^{1/m}}$ (since $J$ is commutative).
			
			Both sides of $\g_{m}$ admit $\mu_{m}$-equivariant structures. It is easy to show that the isomorphism $\g$ is compatible with the $\mu_{m}$-equivariant structures (again it suffices to check it over the center of the disk $\Spec R$). Taking restriction of scalars from $D^{(m)}_{R}$ to $D_{R}$ and taking $\mu_{m}$-fixed points (see \eqref{des J}), we get the desired isomorphism
			\begin{equation*}
				\g_{m}: \bJ^{\na}|_{D^{(m)}_{R}}\cong \bC^{\na}_{\infty}\htimes_{k}\Spec R.
			\end{equation*}
		\end{proof}
		
		By this lemma, $J, \bJ$ and $\bJ^{+}$ all admit canonical trivializations over $\psi+\frg\lr{\t}_{\le0}$. Using the trivializations, the action of the group ind-scheme $J\lr{\t}$ on $\wh\cM_{\Hod,\psi}$  over $\psi+\frg\lr{\t}_{\le0}$ becomes an action of the constant group $C\lr{\t}$ on $\wh\cM_{\Hod,\psi}$. On the quotient $\cM_{\Hod,\psi}=\wh\cM_{\Hod,\psi}/\bK_{\infty}$, the action of $\bC^{+}_{\infty}\subset C\lr{\t}$ is trivial, hence we get an action of $C\lr{\t}/\bC^{+}_{\infty}$ on $\cM_{\Hod,\psi}$.  
		
		Summarizing, we get:
		\begin{prop}\label{p:lattice action} There is a canonical action of $C\lr{\t}/\bC^{+}_{\infty}$ on $\cM_{\Hod,\psi}$. The formal residue map $\res_{\Hod,\infty}$ is equivariant under the $C\lr{\t}/\bC^{+}_{\infty}$-action. Here, the action on $\frc_{0}$ factors through the lattice quotient $C\lr{\t}/\bC^{\na}_{\infty}\cong \xcoch(T)_{\j{w}}/\tors$, and is given by translation under the map $\res_{C,\infty}$ in \eqref{resC}.

			The same statements hold for ${}_{\frh}\cM_{\Hod,\psi}$.
		\end{prop}

\section{The Betti moduli  space}
In this section, we first recall the moduli space $\cM(\b)$ defined using a positive braid $\b$. When $k=\CC$ and $G=\GL(n)$, we recall the interpretation of $\cM(\b)$ as a moduli space of Stokes filtered local systems. For arbitrary $G$, we define $\cM_{\Bet, \psi}$ as an enhanced version of $\cM(\b)$, and a set-theoretic map $\cM_{\dR, \psi} \to \cM_{\Bet, \psi}$ which is conjectured to be biholomorphic.

\subsection{The stack $\cM(\b)$} 
In this subsection, $G$ and $k$ are in the same generality as \S\ref{s:Dol}.

Let $(\bW,S)$ be the abstract Weyl group of $G$ with simple reflections $S$, i.e., $\bW$ is the set of $G$-orbits on $\cB^{2}$, where $\cB$ is the flag variety of $G$. For $w\in \bW$, let $BS(w)\subset \cB^{2}$ be the corresponding $G$-orbit.

Let $H$ be the universal Cartan of $G$, i.e., the reductive quotient of any Borel subgroup of $G$. The abstract Weyl group $\bW$ acts on $H$. It is conceptually important to distinguish between the maximal torus $T$ attached to $\psi$ and the universal Cartan $H$, and between the Weyl group $W=W(G,T)$ and the abstract Weyl group $\bW$.

Let $Br_{\bW}$ be the braid group of $\bW$ and $Br^{+}_{\bW}$ be the monoid of positive braids. For $w\in W$, let $\wt w$ be its canonical lifting to $Br^{+}_{\bW}$ as a reduced word in $S$.

Let $\b\in Br^{+}_{\bW}$ and write
\begin{equation}\label{braid decomp}
\b=\wt w_{1}\cdots \wt w_{n}
\end{equation}
for a sequence of elements $w_{1},\cdots, w_{n}\in \bW$. Let $w\in \bW$ be the image of $\b$ in $\bW$, i.e., $w=w_{1}\cdots w_{n}$.

\begin{defn}
	Let $\cM(\beta)$ be the moduli stack parametrizing:
	\begin{enumerate}
		\item An $n+1$-tuple $(E_0, ..., E_n)$ of $B$-torsors over a point (or a test scheme).
		\item For $0\le i\le n-1$, an isomorphisms of $G$-torsors $\io_{i}: E_i \times^{B} G \to E_{i+1} \times^{B} G$, such that the two $B$-reductions of the identified $G$-torsor are in relative position $w_i$.

		\item An isomorphism of $B$-torsors $\t: E_n\isom E_0$.
			\end{enumerate}
\end{defn}

By \cite[Application 2]{Del},  $\cM(\b)$ depends only on the positive braid $\b$ and not on the decomposition \eqref{braid decomp}, up to a canonical isomorphism. 

The composition of the isomorphisms $\io_{0},\cdots, \io_{n-1}$ together with $\t$ defines an automorphism of the $G$-torsor $E_0 \times^{B} G$, therefore a map 
\begin{equation} \label{eq:monodromymap} \mu_{\b, G}: \cM(\beta) \to [G/\Ad(G)]. \end{equation}

On the other hand, each $E_i$ induces a $T$-torsor $K_i$ via the surjection $B\surj T$. The map $E_{i} \times^{B} G \to E_{i+1} \times^{B} G$ induces an isomorphism between $K_{i}$ and $w_i(K_{i+1})=K_{i+1}\times^{T,w_{i}}T$. Taking the composition of all these maps we get an isomorphism between $K_0$
and $w(K_0)$. This gives a map 
\begin{equation} \label{eq:fmlmonodromy} \mu_{\b, H}: \cM(\beta) \to [H/\Ad_w(H)] \end{equation}
where $\Ad_{w}(H)$ means $t\in H$ is acting on $H$ by $x\mapsto txw(t^{-1})$.

We give an alternative description of $\cM(\b)$, following the construction in \cite{STWZ} and \cite{MT}. Let $\cM^{\sh}(\b)$ be the moduli of $(E_{0},\cdots, E_{n}, \io_{i}, \t)$ as in $\cM(\b)$ together with a trivialization of $E_{0}\times^{B}G$. The monodromy map \ref{eq:monodromymap} lifts to $\cM^{\sh}(\b) \to G$. Via the isomorphisms of $G$-torsors, $(E_{0},\cdots, E_{n})$ give $B$-reductions of $E_{0}\times^{B}G$, which via the trivialization give a tuple of Borel subgroups of $G$. We are led to the following description\begin{equation*}
	\cM^{\sh}(\b)\cong \{ B_0, ...., B_{n}, g) \in \cB^{n+1}\times G | (B_i, B_{i+1}) \in BS(w_i) \mbox{ for $0\le i\le n-1$, and } B_{n} = {}^{g}B_{0} \}. 
\end{equation*}
The $G$-action on $\cM^{\sh}(\b)$ by changing the trivialization of $E_{0}\times^{B}G$ corresponds to the diagonal action on $B_{i}\in \cB$ and the conjugation action on $g\in G$. Then
\begin{equation*}
	\cM(\b)=[G\bs \cM^{\sh}(\b)]. 
\end{equation*}

From this description we easily see that
\begin{lemma}\label{l:dimMbeta} For any $\b\in Br^{+}_{\bW}$, $\cM(\b)$ is a smooth algebraic stack over $k$ of dimension $\ell(\b)$, the length of $\b$.
\end{lemma}

\begin{exam} Consider the case $\b=\wt w_{0}^{2}$ is the ``full twist'', where $w_{0}$ is the longest element in $\bW$. This case will come up when we consider connections from $\cM_{\dR, \psi}$ when $\psi$ is homogeneous of slope $\nu=1$.

In this case, $\cM^{\sh}(\b)$ classifies $(B_{0},B_{1},B_{2}, g)$ where both pairs of Borel subgroups $(B_{0},B_{1})$ and $(B_{1},B_{2})$ are opposite, and ${}^{g}B_{0}=B_{2}$. Fix a pair of opposite Borel subgroups $B^{+}$ and $B^{-}$ with $T=B^{+}\cap B^{-}$. Up to $G$-action we may assume $B_{0}=B^{+}$ and $B_{1}=B^{-}$. Then $(B^{-}, {}^{g}B^{+})$ is in general position if and only if $g\in B^{-}B^{+}$. We get 
\begin{equation*}
\cM(\wt w_{0}^{2})\cong [B^{-}B^{+}/\Ad(T)].
\end{equation*}
The map $\mu_{\wt w_{0}^{2},G}: \cM(\wt w_{0}^{2})\to [G/\Ad(G)]$ is induced from the inclusion $B^{-}B^{+}\subset G$; the  map $\mu_{\wt w_{0}^{2},T}: \cM(\wt w_{0}^{2})\to [T/\Ad(T)]$ is induced from the projection $B^{-}B^{+}=N^{-}TN^{+}\to T$.
\end{exam}

\subsection{Stokes filtered local systems} From now on we set $k=\CC$. In this subsection, we consider the case $G=\GL(n)$.

Given a meromorphic connection $(\cE, \nb)$ on a punctured disk $\D^{\times}$ with coordinate $\t$, taking analytic flat sections defines a local system on $\D^{\times}$ endowed with a Stokes structure at $\t=0$. We now recall the definition of this Stokes structure, following Sabbah's presentation in \cite[Chapter 2]{Sabbah}. The material that follows, up to and including Section \ref{ss:RHmGLn}, is well-known, and we include it for clarity.

Let $\wt\D(0)$ denote the real blow-up  of the disk $\D$ at $\t=0$, with boundary circle $\realcircle$. Consider the  constant sheaf $\cJ_1$ on $\realcircle$  with fiber 
\[ \mathcal{P} := \CC\lr{\t}/\CC\tl{\t}. \]

Sections are given by finite sums
\begin{equation} \label{eq:negativepuiseux} \exponenta = \sum_{i < 0} a_i \tau^{i}. \end{equation}
The stalk of $\cJ_1$ over $\anglevar \in \realcircle$ is partially ordered by the rate of growth of a section as $\tau$ approaches $0$ along the ray $\arg(\tau) = \anglevar$. We denote this order by $\leq_\anglevar$. Write $\realcircle_{\exponenta \leq \exponentb}$ for the open subset of $\realcircle$ on which $\exponenta \leq_{\anglevar} \exponentb$.

Let $\cL$ be a local system on $\realcircle$. An {\em unramified pre-Stokes filtration} on $\cL$ is a collection of subsheaves $\cL_{\leq \exponenta} \subset \cL$ for $\exponenta \in \mathcal{P}$ such that for all $\nu\in \realcircle$
\[ \exponenta \leq_\anglevar \exponentb \implies \cL_{\leq \exponenta, \anglevar} \subset \cL_{\leq \exponentb, \anglevar}. \]

Let $\cL_{<\phi}$ be the subsheaf of $\cL_{\le\phi}$ such that $\cL_{<\phi,\anglevar}=\sum_{\chi<_{\anglevar}\phi} \cL_{\le\chi,\anglevar}$.  Let $\gr_{\exponenta }\cL=\cL_{\leq \exponenta} / \cL_{< \exponenta}$. We can associate to this filtration a $\mathcal{P}$-graded sheaf $\gr \cL = \bigoplus_{\exponenta \in \exponentset} \gr_\exponenta \cL$, where $\exponentset \subset \mathcal{P}$ is the finite subset consisting of $\exponenta$ such that $\gr_{\exponenta }\cL\ne0$. The subset $\exponentset$ is called the {\em exponential factors} of $(\cL, \cL_{\leq \bullet})$. The graded pieces are in general not locally constant on $\realcircle$. When they are, the result is a $\mathcal{P}$-graded local system on $\realcircle$, which in this context is called an {\em unramified Stokes-graded local system}. 

Conversely, to an unramified Stokes-graded local system $\cL_{\bullet}$, we can assign an unramified pre-Stokes filtration as follows.
\begin{defn} Let $\exponentset \subset \mathcal{P}$ be a finite subset, and let $\cL_{\bullet} = \bigoplus_{\exponenta \in \exponentset} \cL_{\exponenta}$ be an unramified Stokes-graded local system. 
	The graded unramified Stokes filtration on  $\cL$ is
	\[ \cL_{\leq \exponenta} = \bigoplus_{\exponenta \in \exponentset} \beta_{\exponentb \leq \exponenta} \cL_\chi \]
	where $\beta_{\chi \leq \exponenta}$ indicates restriction to the open $\realcircle_{\chi \leq \exponenta}$ followed by extension by zero.
\end{defn}

\begin{defn} 
	An unramified Stokes filtered local system $(\cL, \cL_{\leq \bullet})$ on $\realcircle$ is a pre-Stokes filtration which is locally isomorphic to a graded unramified Stokes filtration.
\end{defn}

The map $\tau \to \tau^m$ induces a map of real blow ups at $\tau = 0$. Restricting to the boundary circles, we obtain an $m$-fold cover $\rho_m : \realcircle' \to \realcircle$. Denote by $\sigma : \realcircle' \to \realcircle'$ the generator of the automorphism group of this cover given by $\s(\t^{1/m})=e^{2\pi i/m}\tau^{1/m}$. 

Define $\cP_{m}=\CC\lr{\t^{1/m}}/\CC\tl{\t^{1/m}}$. It carries a natural action of the Galois group $\langle \sigma \rangle \cong \ZZ/m\ZZ$.		

\begin{defn} 
	A Stokes-filtered local system is a triple $(\cL, \Psi, \cL'_{\le\bu})$ where
	\begin{enumerate}
		\item $\cL$ is a local system on $\realcircle$.
		\item  $\exponentset \subset \cP_{m}$ is a finite subset stable under the action of the Galois group $\langle \sigma \rangle$.
		\item $\cL'_{\le\bu}$ is a pre-Stokes filtration on $\cL':=\rho_{m}^{*}\cL$ such that  $(\cL', \cL'_{\leq \bullet})$ is an unramified Stokes-filtered local system on $\realcircle'$ with exponents $\exponentset$. Moreover, we require the canonical isomorphism $\sigma^* \cL'\cong\cL'$ to identify the subsheaves $\cL'_{\leq \exponenta}$ and $\cL'_{\leq \sigma(\exponenta)}$, for all $\phi\in \cP_{m}$. 
	\end{enumerate}

\end{defn}

From now on, we fix $\exponentset \subset \CC(\tau^{1/m})$ to be the set of eigenvalues of $\psi$. It consists of $n$ distinct monomials of degree $-d/m$. We will be concerned with Stokes-filtered local systems $(\cL, \exponentset, \cL'_{\le \bu})$ for which $\dim \gr_\exponenta \cL' = 1$ for each $\exponenta \in \exponentset$. We fix a degree $m$ cover $\realcircle' \to \realcircle$. Monodromy around $\tau = 0$ defines a permutation $w$ of $\exponentset$ of order $m$. 

Choose a base point $\anglevar\in \realcircle$.  Let $T\cong (\Gm)^{\Psi}$ denote the subgroup of graded automorphisms of the fiber of $\gr \cL'$ at $\nu$. Then $w$ acts on $T$ by permuting factors. 
\begin{defn}
	Let $T_{\j{w}}$ be the coinvariant torus of $T$ under the action of the cyclic group $\j{w}$.
	The formal monodromy of $(\cL, \exponentset, \cL'_{\leq \bullet})$ is the element in $T_{\j{w}}$ defined by parallel transport in $\gr \cL'$ from $\anglevar$ to $\s(\anglevar)$. 
\end{defn}				

\sss{Moduli of Stokes filtered local systems}\label{ss:M Bet as Stokes}
We describe the moduli of Stokes local systems in our setting, when $G = \GL(n)$. We will give a more general construction, which makes sense for arbitrary $G$, in Section \ref{RHm}. 

The exponents $\Psi$ define a braid as follows. Each $\exponenta \in \Psi$ defines a function $\Re(\exponenta) : \realcircle' \to \RR$. We say $\anglevar \in \realcircle$ is a Stokes ray if $\Re(\exponenta) = \Re(\exponentb)$ on a preimage of $\anglevar$ under $\realcircle' \to \realcircle$. To simplify the exposition, we assume that the Stokes rays of distinct pairs $(\exponenta, \exponentb)$ are distinct. This holds for a dense set of $\psi$. This assumption will be lifted in Section \ref{RHm}.

The Stokes rays divide $\realcircle$ into $k$ Stokes sectors. Fix $\startingangle$ in the interior of such a sector. As $G = GL(n)$, the braid $\b_{\nu}$ arises from a loop $S^1 \to \Config_n(\CC)$, whose basepoint we take to be $\startingangle$. The real projection of this loop is given by the union of graphs of $\Re(\exponenta), \exponenta \in \Phi$, viewed as multivalued functions on $\realcircle$. In fact, $\b_\nu$ can be reconstructed from these graphs. A Stokes sector determines a complete ordering on $\Psi$. A Stokes ray for the pair $(\exponenta, \exponentb)$ determines a positive half-twist interchanging $(\exponenta, \exponentb)$. The braid $\b_{\nu}$ is the product of these half-twists.

\begin{prop} Recall $G=\GL(n)$.
	The moduli stack of Stokes-filtered local systems with exponential factors $\exponentset$ is isomorphic to $\cM(\beta)$. This isomorphism identifies the maps \ref{eq:monodromymap} and \ref{eq:fmlmonodromy} with the monodromy and formal monodromy respectively.		
\end{prop}
\begin{proof}
	For any given Stokes sector, the set of framings of the fiber $\cL$ compatible with the Stokes filtration is a $B$-torsor. This defines the $k+1$-tuple $(E_0, ..., E_k)$, where $E_0$ and $E_k$ are both associated to the initial Stokes sector. The isomorphisms $E_i \times_B G \to E_{i+1} \times_B G$ are furnished by parallel transport. The resulting pair of $B$-reductions of the $G$-torsor are related by the reflection $s_i$ associated to that Stokes ray. 
	
	The identification of the monodromy and formal monodromy is a direct consequence of the definitions.
\end{proof}
\sss{Riemann-Hilbert map}\label{ss:RHmGLn}
	Let $\Connec$ be the category of meromorphic connections $(\cE, \nb)$ on the punctured disk. Let $\Stokes$ be the category of Stokes-filtered local systems on $\realcircle$. There is an `irregular Riemann-Hilbert' functor \[ \RH : \Connec \to \Stokes, \ \ \ (\cE, \nb) \to (\cL, \cL_{\leq \bullet}), \]
where $\cL$ is the local system of $\nb$-flat sections of $\cE$ away from $\infty$, and the filtration $\cL_{\bullet}$ is by order of growth near $\infty$, which in this formulation is due to Deligne. It is an equivalence of categories.

	The Stokes filtrations depend holomorphically on the connection $\nb$ in the following sense. Let $\nb_u = d + A(\tau, u)\frac{d\tau}{\tau}$ be a family of connections on the trivial bundle $\cE = \cO^n$ over the analytic punctured disk $\Delta^{\times}_\infty$, depending holomorphically on an auxiliary variable $u$ which varies in a domain $U \subset \CC^N$. Suppose that the irregular part of $A(\tau, u)$ is constant in $u$, with regular semisimple leading term, and only finitely many coefficients of $A(\tau, u)$ are non-constant. 

Each such connection determines the same set of Stokes directions $\Stdir(\cE, \nb) \subset \realcircle$. For each $u \in U$ and $\theta \in \realcircle \bs \Stdir(\cE, \nb)$, we obtain a flag $B_{\theta, u}$ in the fiber $\cE_\theta = \CC^n$. 
\begin{lemma} \label{lem:RHisholomorphic}
	The map $U \to \cB$ defined by the above construction is complex analytic. 
\end{lemma}
\begin{proof}
	This follows from \cite[Remark 1.8]{BJL}, where it is explained that the sectorial flat sections of $(\cE, \nb)$, out of which the Stokes filtrations are constructed, vary holomorphically with $\nb$.
\end{proof}

			\subsection{Riemann-Hilbert map  for $G$-connections }\label{RHm}
			
			Now we are back in the setting of general reductive group $G$, and $k=\CC$.

\sss{Stokes directions for $G$-connections}
Let $(\cE,\nb)$ be a meromorphic $G$-connection on the punctured disk $\D^{\times}$ with coordinate $\t$. The restriction of $(\cE, \nb)$ to the formal punctured disk $D_{\infty}^{\times}$ (around $\t=0$), after passing to a ramified cover with parameter $\t^{1/m}$, can be formally gauge transformed to $\nb\in d+(B(\t^{1/m})+\frg\tl{\t^{1/m}})d\t/\t$, for some irregular part $B(\t^{1/m})\in \frt[\t^{-1/m}]$ (where $\frt\subset\frg$ is a Cartan subalgebra). Recall that $\fra=\frt\sslash W$. The image of $B(\t^{1/m})$ under the projection $\frt[\t^{-1/m}]\to \fra[\t^{-1/m}]$ lies in $\fra[\t^{-1}]$, which we denote by $A(\t)\in \fra[\t^{-1}]$.

Below we assume that $B(\t^{1/m})$ is regular semisimple, i.e., for any root $\a$ of $\frg$ with respect to $\frt$, $\a(B(\t^{1/m}))\ne0$.

For any finite-dimensional representation $V$ of $G$, we have the associated local system $\cE(V)^{\nb}$ of flat sections of $\cE(V)=\cE\times^{G}V$ on $\PP^{1}\bs\{0\}$.  For $\th\in \realcircle$, we have a subspace $M(V)_{\th}\subset (\cE(V)^{\nb})_{\th}$ (meaning the stalk of $\cE(V)^{\nb}$ along any point in the ray $\th$) consisting of solutions of maximal decay in a small sector containing the ray of $\th$. Since $B(\t^{1/m})$ is regular semisimple, $\dim M(V)_{\th}=1$ for all but finitely many $\th\in \realcircle$. Note that
\begin{equation}\label{Plucker}
M(V_{1}\ot V_{2})_{\th}=M(V_{1})_{\th}\ot M(V_{2})_{\th}\subset \cE(V_{1}\ot V_{2})^{\nb}_{\th}.
\end{equation}

A point $\th\in \realcircle$ is called a {\em Stokes direction} for $(\cE,\nb)$ if for some irreducible representation $V$ of $G$, $\dim M(V)_{\th}>1$. We claim that there are only finitely many Stokes rays. Indeed, $V_{\l}\in\Rep(G)$ is the irreducible representation with highest weight $\l$, and for some finite subset $\{\l_{1},\cdots, \l_{N}\}$ that generate the monoid of dominant weights of $G$, $\dim M(V_{\l_{i}})_{\th}=1$ for $1\le i\le N$ implies $\dim M(V_{\l})_{\th}=1$ for all $\l$ (using \eqref{Plucker}).   

Let $\Stdir(\cE,\nb)\subset \realcircle$ be the set of Stokes directions. A connected component of $\realcircle\bs \Stdir(\cE,\nb)$ is called a {\em Stokes sector}.

Below, we fix a base point $\th_{0}\in \realcircle$. We label elements in $\Stdir(\cE,\nb)$ counterclockwisely as $\{\s_{1},\s_{2},\cdots, \s_{n}\}$, starting with the one immediately next to $\th_{0}$ in the counterclockwise direction. Let $I_{i}=(\s_{i}, \s_{i+1})$ be the Stokes sectors (for $0\le i\le n-1$, and $\s_{0}:=\s_{n}$). 

\begin{cons}[$B$-torsors for each sector]\label{cons:Bred}
For $\th\in \realcircle\bs \Stdir(\cE,\nb)$, we define a $B$-reduction of $\cE_{\th}$ as follows. Indeed, the assignment $V\mapsto M(V)_{\th}\subset \cE(V)^{\nb}_{\th}$ satisfies the relation \eqref{Plucker}  and $M(V)_{\th}$ is one-dimensional for $V$ irreducible. Such data determines a $B$-reduction of the $G$-torsor $\cE^{\nb}_{\th}$ along the ray $\th$. Clearly, this $B$-reduction is locally constant as $\th$ moves in a Stokes sector. Therefore, on each Stokes sector $I_{i}\subset \realcircle\bs \Stdir(\cE,\nb)$, we have a canonical $B$-torsor $E_{i}$ constructed from $(\cE,\nb)$, for $0\le i\le n-1$. We let $E_{n}=E_{0}$.

\end{cons}

\begin{cons}[The braid]\label{cons:braid}
Let $S^{1}_{\e}$ be the circle of radius $\e>0$ around $\t=0$. There is a canonical isomorphism $S^1_{\e}\cong \realcircle$. Restricting the map $A:\AA^{1}_{\t^{-1}}\to \fra=\frh\sslash \bW$ to $S^{1}_{\e}$ the image lands in $\fra^{\rs}$. 

Fix a base point $\th_{0}\in \realcircle$, which gives a corresponding base point $\e e^{i\th_{0}}\in S^{1}_{\e}$. If $\wt\th_{0}$ is a lifting of $\th_{0}$ to $\RR$, we abuse the notation to denote the interval $[\wt\th_{0}, \wt\th_{0}+2\pi]$ by $[\th_{0}, \th_{0}+2\pi]$. Let $a_{0}=A(\e e^{i\th_{0}})\in\fra^{\rs}$. Choose a lifting $\wt a_{0}\in\frh^{\rs}$ of $a_{0}$. Then $A|_{S^{1}_{\e}}: S^{1}_{\e}\to \fra^{\rs}$ lifts uniquely to $\wt A_{\e}: [\th_{0},\th_{0}+2\pi]\to \frh^{\rs}$ with $\wt A_{\e}(\th_{0})=a_{0}$.

Let $\wt\s_{i}$ be the preimage of $\s_{i}$ in $[\th_{0},\th_{0}+2\pi]$, and similarly let $\wt I_{i}\subset [\th_{0}, \th_{0}+2\pi]$ be the preimage of $I_{i}$. Note $\wt I_{0}=[\th_{0},\wt\s_{1}]\sqcup [\wt\s_{n}, \th_{0}+2\pi]$. We denote $J_{i}=[\th_{0},\wt\s_{1}]$ if $i=0$, $J_{i}=\wt I_{i}$ for $1\le i\le n-1$, and $J_{n}=[\wt\s_{n}, \th_{0}+2\pi]$. 

Consider the projection $\wt A_{\RR,\e}:  [\th_{0},\th_{0}+2\pi]\xr{\wt A_{\e}} \frh^{\rs}\to\frh_{\RR}$ (projection to the real part). Then $\{\wt\s_{1},\cdots, \wt\s_{n}\}$ is precisely the preimage of the root hyperplanes in $\frh_{\RR}$ under $\wt A_{\e}$. The image $\wt A_{\e}(J_{i})$ is contained in a unique Weyl chamber $C_{i}\subset \frh_{\RR}^{\rs}$. For $\e$ sufficiently small, $C_{i}$ are independent of $\e$. The relative positions of two Weyl chambers in $\frh_{\RR}$ are indexed by $\bW$. Let $w_{i}\in \bW$ be the relative position of the Weyl chambers $(C_{i-1},C_{i})$ for $1\le i\le n$. Then define
\begin{equation*}
\b_{\th_{0}}=\wt w_{1}\wt w_{2}\cdots \wt w_{n}\in Br^{+}_{\bW}.
\end{equation*}
Recall for $w\in \bW$, we write $\wt w\in Br^{+}_{\bW}$ its canonical lift. The braid $\b_{\th_{0}}$ does not depend on the lifting $\wt a_{0}$ of $a_{0}$. Changing the base point $\th_{0}$, $\b_{\th_{0}}$ changes by a cyclic shift of words.
\end{cons}

\begin{remark}\label{r:braid via pi1} There is another natural way to get a conjugacy class in $Br_{\bW}$ from $A(\t)$.  The map $A|_{S^{1}_{\e}}$ gives an element in $\pi_{1}(\fra^{\rs}, A(\e e^{i\th_{0}}))$, which is isomorphic to $Br_{\bW}$ (and the isomorphism is unique up to conjugacy). We thus get a conjugacy class $[\b]$ in $Br_{\bW}$. One can show that $[\b]$ is independent of $\e$ and $\th_{0}$ as long as $\e$ is sufficiently small, and that $\b_{\th_{0}}$ defined above belongs to the conjugacy class $[\b]$.
\end{remark}

Recall the irregularity of the adjoint connection $(\Ad(\cE),\nb)$ associated with $(\cE, \nb)$ is defined to be
\begin{equation*}
\Irr(\Ad(\cE),\nb)=\sum_{\a\in \Phi}-\ord_{\t=0}\a(B(\t^{1/m})).
\end{equation*}

\begin{lemma}\label{l:lb}
We have $\ell(\b)=\Irr(\Ad(\cE),\nb)$.
\end{lemma}
\begin{proof}
We first consider the case $m=1$, i.e., $B(\t)\in \frh[\t^{-1}]$. On the one hand, $\ell(\b)=\sum_{i=1}^{n}\ell(w_{i})$, and $\ell(w_{i})$ is the number of root hyperplanes separating $C_{i-1}$ and $C_{i}$. Therefore $\ell(\b)$ is the number of times the image of $B_{\RR}|_{S^{1}_{\e}}$ crosses the root hyperplanes (where $B_{\RR}$ is the real projection $\AA^{1}_{\t^{-1}}\to \frh\to \frh_{\RR}$). For the root hyperplane $H_{\a}$ defined by $\a=0$, $B_{\RR}|_{S^{1}_{\e}}$ intersects $H_{\a}$ exactly when $B_{\a}(\t)=\a(B(\t))$ takes values in $i\RR$ for $|\t|=\e$. For $\e\ll1$, the map $B_{\a}:S^{1}_{\e}\to \CC^{\times}$ has mapping degree $-\ord_{\t=0}\a(B(\t))$, and $B_{\a}|_{S^{1}_{\e}}$ intersects $i\RR$ transversely (with the same sign of intersection) in $-2\ord_{\t=0}\a(B(\t))$ times. The total number of times $B_{\RR}|_{S^{1}_{\e}}$ crosses the root hyperplanes is 
\begin{equation*}
\sum_{\a\in \Phi^{+}}-2\ord_{\t=0}\a(B(\t))=\sum_{\a\in \Phi}-\ord_{\t=0}\a(B(\t))
\end{equation*}
which is $\Irr(\Ad(\cE), \nb)$. 

In general, let $\pi: \D^{\times}_{\t^{1/m}}\to \D_{\t}^{\times}$ be the projection. Then the irregular part of $\pi^{*}(\cE,\nb)$ lies in $\frh[\t^{-1/m}]$. Let $\wt\b$ be the braid attached to $\pi^{*}(\cE,\nb)$. On the one hand we have $\Irr(\Ad(\cE),\nb)=\frac{1}{m}\Irr(\pi^{*}\Ad(\cE),\nb)$, which is equal to $\frac{1}{m}\ell(\wt\b)$ by the previous paragraph. On the other hand, from the construction of $\wt\b$ one sees that $\wt\b=\b_{1}\b_{2}\cdots \b_{m}$ where each $\b_{i}$ is a positive braid with the same length as $\b$. Therefore $\ell(\wt\b)=m\ell(\b)$. Combining these facts we get $\Irr(\Ad(\cE),\nb)=\frac{1}{m}\ell(\wt\b)=\ell(\b)$. 
\end{proof}

\begin{cons}[A point in $\cM(\b_{\th_{0}})$]\label{cons:RH} We shall construct a point in $\cM(\b_{\th_{0}})$ from $(\cE,\nb)$.  In Construction \ref{cons:Bred} we have constructed a $B$-torsor $E_{i}$ for $0
\le i\le n$, with $E_{n}=E_{0}$. Let $\t=\id: E_{n}\isom E_{0}$ be the identity isomorphism. For each $\s_{i}$, let $\s^{-}_{i}\in I_{i-1}$ and $\s^{+}_{i}\in I_{i}$, then parallel transport along the path from $\s^{-}_{i}$ to $\s^{+}_{i}$ that passes through $\s_{i}$ identifies the stalks $\cE^{\nb}_{\s^{-}_{i}}$ and $\cE^{\nb}_{\s^{+}_{i}}$. This gives a canonical isomorphism of $G$-torsors
\begin{equation}\label{EiG}
\io_{i-1}: E_{i-1}\times^{B}G\cong \cE^{\nb}_{\s^{-}_{i}}\cong \cE^{\nb}_{\s_{i}}\cong \cE^{\nb}_{\s^{+}_{i}}\cong E_{i}\times^{B}G.
\end{equation}
By the lemma below, the data $(E_{0},\cdots, E_{n}, \io_{0},\cdots ,\io_{n-1}, \id)$ defines a point in $\cM(\b_{\th_{0}})$.

\end{cons}

\begin{lemma}
The relative position of the two $B$-reductions $E_{i-1}$ and $E_{i}$ of $\cE^{\nb}_{\s_{i}}$ (see \eqref{EiG}) is equal to $w_{i}\in \bW$ defined in Construction \ref{cons:braid}.
\end{lemma}
\begin{proof}
We treat the case $m=1$ below; the general case is proved by the same argument after pullback to the cover $\D^{\times}_{\t^{1/m}}\to \D^{\times}_{\t}$. Also, without loss of generality we may assume $i=1$.  Let $w\in \bW$ be the relative position of $E_{0}$ and $E_{1}$.

Let $\s^{-}\in I_{0}$ and $\s^{+}\in I_{1}$. Let $x_{\e}^{-}=\Re B(\e e^{i\s^{-}})\in C_{0}$ (the dominant Weyl chamber, for $\e\ll1$), and $x_{\e}^{+}=\Re B(\e e^{i\s^{+}})\in C_{1}$ for $\e\ll1$.

Fix a regular anti-dominant weight $\l$, and denote by $V_{\l}$ the irreducible representation of $G$ with lowest weight $\l$. Let $\Om_{\l}$ be the set of weights of $V_{\l}$. The $B$-reduction $E_{-}$ gives an increasing filtration $F_{\mu}\cE(V_{\l})^{\nb}_{\s^{-}}$ on $\cE(V_{\l})^{\nb}_{\s^{-}}$ indexed by the poset $\Om_{\l}$ (where $\mu_{1}\ge\mu_{2}$ if and only if $\mu_{1}-\mu_{2}$ is a sum of positive coroots). Note that the smallest sub $F_{\l}\cE(V_{\l})^{\nb}_{\s^{-}}$ is the maximal decay line $M(V_{\l})_{\s^{-}}$. On the other hand, the maximal decay line $M(V_{\l})_{\s^{+}}\subset \cE(V_{\l})^{\nb}_{\s^{+}}$, after parallel transport to $\s^{-}$ via the path through $\s_{1}$, gives a line $\cL_{+}\subset \cE(V_{\l})^{\nb}_{\s^{-}}$.  Since $\l$ is regular, the relative position $w$ can be characterized as follows: it is the unique element $w\in \bW$ such that $\cL_{+}\subset F_{w\l}\cE(V_{\l})^{\nb}_{\s^{-}}$ and $\cL_{+}$ maps injectively to the associated graded $\Gr^{F}_{w\l}\cE(V_{\l})^{\nb}_{\s^{-}}$. We would like to show that this property of $w$ implies $w=w_{1}$.

Let $\Psi_{\l}\subset\cP$ be the set of exponential factors of the connection $(\cE(V_{\l}),\nb)$. Then $\Psi_{\l}$ is the image of $\Om_{\l}\to \cP$ given by $\mu\mapsto \j{\mu,B(\t)}$. As a Stokes filtered local system, $\cE(V_{\l})^{\nb}$ is locally (near $\s_{1}$) isomorphic to a Stokes graded local system $\cL\cong\op_{\chi\in \Psi_{\l}} \cL_{\chi}$. We may assume $\s^{\pm}$ to be close enough to $\s_{1}$ so that this isomorphism is defined along the arc $[\s^{-}, \s^{+}]$. In particular, there is a unique $\chi^{+}\in \Psi_{\l}$ such that $\cL_{\chi^{+}, \s^{+}}=M(V_{\l})_{\s^{+}}$. 

We claim that $\chi^{+}=\j{w_{1}\l,B(\t)}$. Indeed, write $\chi^{+}=\j{\mu, B(\t)}$ for some  $\mu\in \Om_{\l}$. Since $\cL_{\chi^{+}, \s^{+}}$ is the maximal decay line at $\s^{+}$, for $\e\ll1$, $\j{\mu,x_{\e}^{+}}\le \j{\mu', x_{\e}^{+}}$ for all $\mu'\in\Om_{\l}$. Equivalently, $\j{w_{1}^{-1}\mu,w_{1}^{-1}x_{\e}^{+}}\le \j{\mu', w_{1}^{-1}x_{\e}^{+}}$ for all $\mu'\in\Om_{\l}$. Since $w_{1}^{-1}x_{\e}^{+}\in w_{1}^{-1}C_{1}=C_{0}$, which is the dominant Weyl chamber, the above inequality implies $w_{1}^{-1}\mu=\l$ (the minimal element in $\Om_{\l}$), i.e., $\mu=w_{1}\l$, hence $\chi^{+}=\j{w_{1}\l,B(\t)}$.

Now $\cL_{+}=\cL_{\chi^{+},\s^{-}}$. By the definition of $w$,  $\cL_{+}=\cL_{\chi^{+},\s^{-}}$ lies in  $F_{w\l}\cE(V_{\l})^{\nb}_{\s^{-}}$ and maps injectively to the associated graded $\Gr_{w\l}\cE(V_{\l})^{\nb}_{\s^{-}}$. This means $\chi^{+}$ has the same decay rate as $\j{w\l,B(\t)}$ along the ray $\s^{-}$. Using $\chi^{+}=\j{w_{1}\l,B(\t)}$, this implies $\j{w_{1}\l, x_{\e}^{-}}=\j{w\l, x_{\e}^{-}}$ for $\e\ll1$. Since $x_{\e}^{-}$ is in the interior of $C_{0}$, this forces $w_{1}\l=w\l$. Since $\l$ is regular, we conclude that $w_{1}=w$.
\end{proof}

\subsection{Enhanced Riemann-Hilbert map}\label{ss:enh RH}
		
\sss{Specialization to our setting} Consider the context of a homogeneous element $\psi\in\frg\lr{t}$ of slope $\nu=d/m$. For any connection $(\cE,\nb)\in \cM_{\dR,\psi}$ and the choice of a base point $\th_{0}\in \realcircle$ that is not a Stokes direction, Construction \ref{cons:braid} gives a  positive braid $\b_{\th_{0}}$ that depends on $\psi$ but not otherwise on $(\cE,\nb)$. By Remark \ref{r:braid via pi1},  $\b_{\th_{0}}$ is a positive braid representing the loop $S^{1}\to \fra=\frh\sslash \bW$ given by restricting the map $\chi(\psi): \CC^{\times}\to \frg\xr{\chi} \fra$.   One can show that  $\b_{\th_{0}}$ depends only on the slope $\nu$ and not otherwise on $\psi$. We henceforth denote $\b_{\th_{0}}$ by $\b_{\nu,\th_{0}}$, or simply $\b_{\nu}$ is the base point $\th_{0}$ is fixed. The image of $\b$ in $\bW$ is conjugate to $w^{d}$, which is in turn conjugate to $w$ because both $w$ and $w^{d}$ are regular of order $m$.

When $\psi$ is elliptic, we can compute $\b_{\nu}$ as follows.  We have the regular element $w\in \bW$ of order $m$, unique up to conjugacy. Let us assume $w$ has minimal length in its conjugacy class (then $\ell(w)=|\Phi|/m$). One can take $\b_{\nu}=\wt w^{d}$, where $\wt w\in Br_{\bW}^{+}$ is the lifting of $w$ to a positive braid by any reduced expression. We have $\ell(\b_{\nu})=\nu|\Phi|$.  Note that $w$ (hence $\wt w$) is not unique, but different choices of minimal length $w$ differ by cyclic shift of words, as shown by He-Nie in \cite[Corollary 4.4]{HN}. Therefore, the resulting $\b_{\nu}$ also differ by a cyclic shift, which then give isomorphic $\cM(\b_{\nu})$.

Constructions \ref{cons:Bred} and \ref{cons:RH} give a set theoretic map
\begin{equation*}
\RH:\cM_{\dR, \psi}\to \cM(\b_{\nu}).
\end{equation*}
When $G = GL(n)$, Lemma \ref{lem:RHisholomorphic} shows it is a holomorphic map.

\sss{The enhanced RH map}
Recall that $\cM(\b_{\nu})$ is equipped with two monodromy maps: $\mu_{\b_{\nu}, G}$ to $[G/\Ad(G)]$ (which we simply write $[G/G]$ below) and $\mu_{\b_{\nu}, H}$ to $[H/\Ad_{w}(H)]$.  Since the connections classified by the global moduli $\cM_{\dR,\psi}$ also has regular singularities at $0$ with nilpotent monodromy lying in a Borel reduction at $0$. We have a map
\begin{equation*}
\res_{\dR, 0}: \cM_{\dR,\psi}\to [\wt\cN/G]=[\frn/B]
\end{equation*}
where $ \wt\cN\to \cN$ is the Springer resolution of the nilpotent cone $\cN\subset\frg$. Let $\wt \cU\to \cU$ be the Springer resolution of the unipotent variety $\cU\subset G$. The following diagram is commutative by construction
\begin{equation}\label{res mono zero}
\xymatrix{\cM_{\dR, \psi}\ar[d]^{\res_{\dR, 0}}\ar[rr]^{\RH} & & \cM(\b_{\nu})\ar[d]^{\mu_{\b_{\nu}, G}}\\
[\wt\cN/G] \ar[r]^{\exp}_{\sim} & [\wt\cU/G]  \ar[r]& [G/G]
}
\end{equation}

On the other hand, looking at the formal residue and formal monodromy at $\infty$, the following diagram should also be commutative
\begin{equation*}
\xymatrix{ \cM_{\dR, \psi}\ar[d]^{\res_{\dR, \infty}}\ar[rr]^{\RH} & & \cM(\b_{\nu})\ar[d]^{\mu_{\b_{\nu}, H}}\\
\frc_{0}\ar[r]^{\exp} & T^{w,\c} \ar[r]^-{\xi_{\th_{0}}} & [H/\Ad_{w}(H)]
}
\end{equation*}
Here $\xi_{\th_{0}}$ is defined as follows. Let $\th'_{0}\in \realcircle'$ be a preimage of $\th_{0}$ under the degree $m$ cyclic covering $\r_{m}: \realcircle'\to \realcircle$. The adjoint Cartan $\frt^{\ad}$ has a canonical real form $\frt^{\ad}_{\RR}$ consisting of elements whose value under any root is real. Recall $\ov\psi\in\frt$. For $\e>0$ consider the image of $\ov\psi(\e e^{i\wt\th_{0}})^{d}$ under $\frt\surj\frt^{\ad}\surj \frt^{\ad}_{\RR}$. Then for $\e\ll1$ the image lies in a Weyl group $C_{\th'_{0}}\subset \frt^{\ad}_{\RR}$ independent of $\e$. Let $B_{\th'_{0}}$ be the Borel subgroup of $G$ containing $T$ corresponding to the chamber $-C_{\th'_{0}}$. Thus we have an isomorphism of tori $\io_{\th'_{0}}:T\incl B_{\th'_{0}}\surj  H$. Changing the choice of $\th'_{0}$ changes $C_{\th'_{0}}$ by the action of the cyclic group $\j{w}$, therefore $\io_{\th'_{0}}|_{T^{w}}$ is independent of the choice of $\th'_{0}$, which we denote by $\io_{T^{w}}$. The map $\xi_{\th_{0}}$ is the composition $T^{w,\c}\xr{\io_{T^{w}}} H\to [H/\Ad_{w}(H)]$ where the last map is the composition.

More generally, we consider the variant ${}_{s}\cM_{\dR,\psi}$ defined in \S\ref{sss:var s} for an arbitrary residue $s\in\frh$ at $0$. Let $\kappa=\exp(s)\in H$. Let $\wt\frg\to \frg$ and $\wt G\to G$ be the Grothendieck alterations. Let $\pr_{\frh}: \wt\frg\to \frh$ and $\pr_{H}: \wt G\to H$ be the projections. Let $\wt\frg[s]=\pr_{\frh}^{-1}(s)\subset \wt\frg$ and $\wt G[\kappa]=\pr_{H}^{-1}(\kappa)\subset \wt G$. Then the diagram \eqref{res mono zero} becomes
\begin{equation*}
\xymatrix{{}_{s}\cM_{\dR, \psi}\ar[d]^{\res_{\dR, 0}}\ar[rr]^{\RH} & & \cM(\b_{\nu})\ar[d]^{\mu_{\b_{\nu}, G}}\\
[\wt\frg[s]/G] \ar[r]^{\exp}_{\sim} & [\wt G[\kappa]/G]  \ar[r]& [G/G]
}
\end{equation*}

These diagrams motivate the following definition. 

\begin{defn}\label{def:MBet} Let $\psi$ be a homogeneous element of slope $\nu$. 
\begin{enumerate}
\item (Trinh \cite{MT}) Form the {\em derived fiber product}
\begin{equation*}
\cM^{0}_{\Bet,\psi}:=\cM(\b_{\nu})\times^{\bR}_{[G/G]}[\wt \cU/G],
\end{equation*}

\item We define the Betti moduli space attached to $\psi$ to be the derived complex analytic stack
\begin{equation*}
\cM_{\Bet,\psi}:=\frc_{0}\times_{[H/\Ad_{w}(H)]}\cM_{\Bet,\psi}^{0}.
\end{equation*}
Here, the map $\frc_{0}\to [H/\Ad_{w}(H)]$ is given by the exponential map $\frc_{0}\cong \frt^{w}\xr{\exp}H^{w,\c}\subset H$ followed by the quotient map.

\item More generally, for $\kappa\in H$, we define
\begin{eqnarray*}
{}_{\kappa}\cM^{0}_{\Bet,\psi}:=\cM(\b_{\nu})\times^{\bR}_{[G/G]}[\wt G[\kappa]/G]\\
{}_{\kappa}\cM_{\Bet,\psi}:=\frc_{0}\times_{[H/\Ad_{w}(H)]}{}_{\kappa}\cM_{\Bet,\psi}^{0}.
\end{eqnarray*}

\end{enumerate}
\end{defn}

Note when $w$ is elliptic, ${}_{s}\cM_{\Bet,\psi}\to {}_{\kappa}\cM_{\Bet,\psi}^{0}$ is a $H^{w}$-torsor, and ${}_{\kappa}\cM_{\Bet,\psi}$ is a derived algebraic stack. 

\begin{conj}  Let $\kappa\in H$.
\begin{enumerate}
\item For $\nu>0$, the derived structure on ${}_{\kappa}\cM_{\Bet,\psi}^{0}$ is trivial, and ${}_{\kappa}\cM_{\Bet,\psi}^{0}$ is an algebraic stack smooth over $\CC$. 
\item The analytic stack ${}_{\kappa}\cM_{\Bet,\psi}$ is a complex analytic manifold with a canonical symplectic structure.
\end{enumerate}
\end{conj}

By the above diagrams and Definition \ref{def:MBet}, the map $\RH$ lifts to
\begin{equation}\label{enh RH}
\wt\RH: {}_{s}\cM_{\dR,\psi}\to \frc_{0}\times_{[H/\Ad_{w}(H)]}\cM(\b_{\nu})\times^{\bR}_{[G/G]}[\wt G[\kappa]/G]={}_{\kappa}\cM_{\Bet,\psi}.
\end{equation}

\begin{conj}\label{c:RH} For any reductive $G$ over $\CC$ and homogeneous element $\psi$ in $\frg\lr{t}$ of slope $\nu>0$, and any $s\in \frh$ with $\kappa=\exp(s)$, the map $\wt\RH$ in \eqref{enh RH} is an analytic isomorphism.
\end{conj}		
We plan to return to this conjecture in a future work. As a consistency check, we compare the dimensions of $\cM_{\dR,\psi}$ and $\cM_{\Bet,\psi}$. On one hand, we have by Theorem \ref{th:MHod sm} and Theorem \ref{th:geom M}\eqref{th part:M sm}
\begin{equation*}
\dim\cM_{\dR,\psi}=\dim \cM_{\psi}=\nu|\Phi|-r+\dim \frt^{w}.
\end{equation*}
On the other hand, by Lemma \ref{l:dimMbeta} we have
\begin{equation*}
\dim\cM(\b_{\nu})=\ell(\b_{\nu}).
\end{equation*}
By Lemma \ref{l:lb}, we have for any $(\cE,\nb)\in \cM_{\dR,\psi}$
\begin{equation*}
\ell(\b_{\nu})=\Irr(\Ad(\cE), \nb)=\nu|\Phi|.
\end{equation*}
Combining the two equations we get
\begin{equation*}
\dim\cM(\b_{\nu})=\nu|\Phi|.
\end{equation*}
Therefore the {\em derived dimension} of $\cM_{\Bet,\psi}$ is
\begin{equation*}
\dim\cM(\b_{\nu})+(\dim[\wt\cU/G]-\dim [G/G])+(\dim\frc_{0}-\dim[H/\Ad_{w}(H)])=\nu|\Phi|-r+\dim\frt^{w}.
\end{equation*}
Therefore $\cM_{\dR,\psi}$ has the same dimension as the derived version of $\cM_{\Bet,\psi}$.

\begin{remark} By Theorem \ref{th:comp coho s}, the cohomology $\cohog{*}{{}_{s}\cM_{\dR,\psi}}$ is canonically independent of $s\in \frh$. Combined with Conjecture \ref{c:RH}, it then implies that the cohomology $\cohog{*}{{}_{\kappa}\cM_{\Bet,\psi}}$ is independent of $\kappa\in H$ (canonically upon choosing a logarithm of $\kappa$). Keeping track of the symmetry by $(C\lr{\t}/\bC_{\infty}^{+})^{\red}$ on ${}_{s}\cM_{\dR,\psi}$, this also implies the statement that the cohomology of ${}_{\kappa}\cM^{0}_{\Bet,\psi}$ is independent of $\kappa$. In other words, the direct image complex of the map
\begin{equation*}
\cM(\b_{\nu})\times_{[G/G]}[\wt G/G]\to [\wt G/G]\xr{\pr_{H}} H
\end{equation*}
should be locally constant. This direct image complex can be interpreted as the parabolic restriction of character sheaves. It would be interesting to prove the local constancy statement directly for a wider class of positive braids $\b$ and not just those coming from homogeneous $\psi$.
\end{remark}

\begin{remark} Instead of making a base change of $\cM(\b_{\nu})$ along $\frc_{0}\to [H/\Ad_{w}(H)]$, one can reformulate Conjecture \ref{c:RH} by taking a quotient of $\cM_{\dR,\psi}$.

			It is clear that the map $\RH$ is invariant under the action of $C\lr{\t}/\bC_{\infty}^{+}$ on the domain, because that action does not change the isomorphism class of the connection on $X\bs \{\infty\}$. 
			
		Consider the action of the reduced part $(C\lr{\t}/\bC_{\infty}^{+})^{\red}$ on $\cM_{\dR, \psi}$. By \eqref{fil C}, this group is an extension of the lattice $(C\lr{\t}/\bC_{\infty}^{\na})^{\red}\cong\xcoch(T)_{\j{w}}/\tors$ by $\bC_{\infty}^{\na}/\bC_{\infty}^{+}\cong T^{\j{w}}$. By Proposition \ref{p:lattice action}, the action of $(C\lr{\t}/\bC_{\infty}^{+})^{\red}$ translates the formal residue via the lattice quotient $(C\lr{\t}/\bC_{\infty}^{\na})^{\red}$, therefore the quotient of $\cM_{\dR, \psi}$ by $(C\lr{\t}/\bC_{\infty}^{+})^{\red}$ makes sense as an analytic stack over the quotient torus $\frc_{0}/\Im(\res_{C,\infty})$, which is identified with $T_{\j{w}}$ in \eqref{exp Tw}.

		One can reformulate Conjecture \ref{c:RH} by saying that the map
		\begin{equation*}
\cM_{\dR, \psi}/(C\lr{\t}/\bC_{\infty}^{+})^{\red}\to\cM(\b_{\nu})\times_{[G/G]}[\wt\cU/G]
\end{equation*}
is an analytic isomorphism. This is consistent with Conjecture \ref{c:RH}, as one can check that $(C\lr{\t}/\bC_{\infty}^{+})^{\red}$ is naturally isomorphic to the fiber of $\frc_{0}\to [H/\Ad_{w}(H)]$.
\end{remark}

	\section{Microlocal sheaves on $\Fl_{\psi}$ and wildly ramified geometric Langlands}
	
	In this section we will expand on the conjectural equivalence in \S\ref{ss:muintro}.
	
	From work of \cite{KS} and \cite{Nadler}, we can construct a sheaf of categories $\muSh$ on conical open subsets $U\subset T^{*}X$, such that the global sections category $\muSh(T^*X)\cong D(X)$ is the derived category of constructible sheaves on X. 
	
	For a conical Lagrangian $\Lambda \subset T^*X$ we can define the full subcategory $\muSh_\Lambda$ of objects with singular support contained in $\Lambda$.
	
	If $G$ is a group acting on $X$, we can also use this to construct a sheaf of categories on the Hamiltonian reduction $T^*X//G:=\mu^{-1}_{G}(0)/G$. 
	
	Let $X$ an algebraic space with a $\Gm\ltimes\Ga$-action, where $\Gm$ acts linearly on $\Ga$. Let $\mu_{\Gm\ltimes\Ga}:T^*X\rightarrow \AA^1_{\Gm}\times\AA^1_{\Ga}$ be the moment map of the action. We have the following relations of Hamiltonian reductions 
	\begin{equation*}
		\begin{split}
			T^*X//_{1}\Ga&=\mu_{\Gm\ltimes\Ga}^{-1}(\AA^1_{\Gm}\times\{1\})/\Ga\\ &\cong\mu_{\Gm\ltimes\Ga}^{-1}(\AA^1_{\Gm}\times(\AA^1_{\Ga}\setminus\{0\}))/\Gm\ltimes\Ga\\
			&\cong\mu_{\Gm\ltimes\Ga}^{-1}(\{0\}\times(\AA^1_{\Ga}\setminus\{0\}))/\Gm\subset T^*X//\Gm,
		\end{split}	
	\end{equation*}
where the last is an open subset of the Hamiltonian reduction. We can thus identify the shifted $\Ga$ Hamiltonian reduction as an open subset of the $\Gm$ Hamiltonian reduction and use this to define the sheaf of categories $\muSh$ on $T^*X//_1\Ga$, and in particular the category of global sections $\muSh(T^*X//_1\Ga)$.
	
	Let $\bJ_{\infty}$ be the group defined in \S\ref{ss:red}, equipped with a homomorphism $\wt\psi: \bJ_{\infty}\to \Ga$ induced by $\psi$.  In Proposition \ref{p:red} we prove the identification $\cM_{\psi}\cong T^{*}\Bun_{G}(\bJ^{1}_{\infty}; \bI_{0})//_{1}\Ga$, where $\bJ^{1}_{\infty}=\ker(\wt\psi)$. We can thus apply the formalism above to the construction to get a category $\muSh(\cM_{\psi})$.

	\begin{remark} The definition of $\muSh(T^*X//_1\Ga)$ is inspired by a construction of Gaitsgory \cite[\S1.6]{Ga} called {\em Kirillov category}. It allows to define a version of $D_{(\Ga,\psi)}(X)$ (where $\psi$ is understood to be an Artin-Schreier sheaf on $\Ga$, which makes sense in characteristic $p>0$) for $X$ over an arbitrary base field, as long as the $\Ga$-action on $X$ extends to a $\Gm\ltimes\Ga$-action. Denote this Kirillov category by $\Kir(X)$. Then one can construct a functor
	\begin{equation*}
\Kir(X)\to \muSh(T^{*}X//_{1}\Ga).
\end{equation*}
More details will be explained in \cite{BBAMY}. In particular, we have a functor
\begin{equation*}
\Kir(\Bun_{G}(\bJ^{1}_{\infty}; \bI_{0}))\rightarrow \muSh(\cM_{\psi}).
\end{equation*}

	\end{remark}

	We can consider the full subcategory with fixed singular support along the image of $\Fl_{\psi}\rightarrow\cM_{\psi}$ constructed in \S\ref{th:geom M}(\ref{th part:ASF vs HF}) and denote this by $\muSh_{\Fl_{\psi}}(\cM_{\psi})$. 
	
	Consider $\cM^{0,\dG}_{\Bet, \psi}$ the Betti space for $\dG$, the Langlands dual group, using the positive braid defined by $\psi$ and the definition in \S\ref{def:MBet}.
	
	\begin{conj}
		There is a fully faithful functor 
		\begin{equation*}
			\muSh_{\Fl_{\psi}}(\cM_{\psi})\rightarrow \Ind\Coh(\cM^{0,\dG}_{\Bet, \psi}).
		\end{equation*}
	\end{conj}
	\begin{remark}
		\begin{enumerate}
			\item This conjecture can be viewed as a geometric Langlands correspondence for deeper level structures/wild ramifications. At the same time, it can be viewed as an instance of homological mirror symmetry between $\cM_{\psi}$ and $\cM^{0,\dG}_{\Bet, \psi}$.

			\item We expect the image of this functor on compact objects to consist of coherent sheaves that are supported over proper subschemes of $\cM^{0,\dG}_{\Bet, \psi}$. In particular, objects in $\muSh_{\Fl_{\psi}}(\cM_{\psi})$ that have finitely many components of $\Fl_{\psi}$ in their singular support should be sent to sheaves living over the smallest unipotent orbit appearing in $\cM^{0,\dG}_{\Bet, \psi}$.
			\item A possible way to upgrade the above conjecture is to use wrapped microlocal sheaves as defined in \cite{Nadler}. These should include objects living over non-proper subschemes of $\cM^{0,\dG}_{\Bet, \psi}$.
		\end{enumerate}
	\end{remark}

		\end{document}